\documentclass[lineno]{biometrika}
\usepackage{array}

 \usepackage{stmaryrd}

\usepackage[dvipsnames]{xcolor}

\RequirePackage{enumerate}
 \RequirePackage{enumitem}

 \RequirePackage{overpic}

\usepackage{mathtools}
\usepackage{amsfonts,amssymb}

\usepackage{graphicx}
\usepackage{booktabs}
\usepackage{multirow}

\newcommand{\C}{\mathbb{C}} 
\newcommand{\R}{\mathbb{R}}

\renewcommand{\hat}[1]{\widehat{#1}}
\renewcommand{\t}[1]{\widetilde{#1}}

\newcommand{\D}{\mathcal{D}}

\newcommand{\N}{\mathbb{N}}

\renewcommand{\P}{\mathbb{P}}
\newcommand{\E}{\mathbb{E}}

\newcommand{\ve}{\varepsilon}

\newcommand{\ttop}{^{\top}}

\newcommand{\n}{{\textstyle{\frac{1}{n}}}}

\newcommand{\ts}{\textstyle}

\newcommand{\var}{\operatorname{var}}

\newcommand{\supp}{\operatorname{supp}}
\newcommand{\xin}[1]{\xi_{n,#1}}

\newcommand{\lb}{\left(}
\newcommand{\rb}{\right)}
\newcommand{\Var}{\operatorname{Var}}
\newcommand{\bfSigmahat}{	\hat{ \mathbf{\Sigma} }}
\newcommand{\bfSigmahatu}{	\hat{ \underline{\mathbf{\Sigma} }}}

\newcommand{\bfSigma}{\mathbf{\Sigma}}
\newcommand{\bfX}{\mathbf{X}}
\newcommand{\bfI}{\mathbf{I}}

\newcommand{\MP}{Marchenko-Pastur }
\newcommand{\Fu}{\underline{F}}

\newcommand{\cond}{\stackrel{\mathcal{D}}{\to}}
\newcommand{\conp}{\stackrel{\mathbb{P}}{\to}}

\newcommand{\inv}{^{-1}}

\newcommand{\im}{\operatorname{Im}}
\newcommand{\re}{\operatorname{Re}}
\newcommand{\PR}{\mathbb{P}}

\newcommand{\bfx}{\mathbf{x}}

\newcommand{\op}{o_{\PR}(1)}

\newcommand{\bfA}{\mathbf{A}}

\newcommand{\su}{\underline{s}}
\newcommand{\lambdahat}{\lambda_1(\hat\bfSigma)}

\newcommand{\rev}{\textcolor{black}}

\definecolor{null}{RGB}{235,240,245}  
\definecolor{alt}{RGB}{255,235,220}

\usepackage{newtxtext}
\usepackage[subscriptcorrection]{newtxmath}

\graphicspath{{./art/}}

\usepackage[plain,noend]{algorithm2e}

\makeatletter
\renewcommand{\algocf@captiontext}[2]{#1\algocf@typo. \AlCapFnt{}#2} 
\def\@algocf@capt@plain{top}
\renewcommand{\algocf@makecaption}[2]{%
  \addtolength{\hsize}{\algomargin}%
  \sbox\@tempboxa{\algocf@captiontext{#1}{#2}}%
  \ifdim\wd\@tempboxa >\hsize
    \hskip .5\algomargin%
    \parbox[t]{\hsize}{\algocf@captiontext{#1}{#2}}
  \else%
    \global\@minipagefalse%
    \hbox to\hsize{\box\@tempboxa}
  \fi%
  \addtolength{\hsize}{-\algomargin}%
}
\makeatother

\begin{document}

\jname{}

\jyear{}
\jvol{}
\jnum{}
\cyear{}

\accessdate{}

\markboth{N. Dörnemann and M. E. Lopes}{Approximations for Leading Eigenvalues in High-Dimensional PCA}

\title{Tracy-Widom, Gaussian, and Bootstrap: Approximations for Leading Eigenvalues in High-Dimensional PCA}

\author{N. Dörnemann}
\affil{
 Department of Mathematics, Aarhus University, \\
Ny Munkegade 118,
8000 Aarhus C,
Denmark
\email{ndoernemann@math.au.dk}}

\author{\and M. E. Lopes}
\affil{Department of Statistics,  University of California, Davis, \\ Mathematical Sciences Building 4118, 399 Crocker Lane, Davis, CA 95616, USA \email{melopes@ucdavis.edu}}

\nolinenumbers
\maketitle
\thispagestyle{empty}

\begin{abstract}
Under certain conditions, the largest eigenvalue of a sample covariance matrix undergoes a well-known phase transition when the sample size $n$ and data dimension $p$ diverge proportionally. 
In the subcritical regime, this eigenvalue has fluctuations of order $n^{-2/3}$ that can be approximated by a Tracy-Widom distribution, while in the supercritical regime, it has fluctuations of order $n^{-1/2}$ that can be approximated with a Gaussian distribution. However, the statistical problem of determining which regime underlies a given dataset is far from resolved. We develop a new testing framework and procedure to address this problem. In particular, we demonstrate that the procedure has an asymptotically controlled level, and that it is power-consistent for certain alternatives. Also, this testing procedure enables the design of a new bootstrap method for approximating the distributions of functionals of the leading sample eigenvalues within the subcritical regime---which is the first such method that is supported by theoretical guarantees.
\end{abstract}


\section{Introduction} \label{sec_intro}
The eigenvalues of sample covariance matrices play a key role in many applications of high-dimensional inference that are related to principal component analysis (PCA), such as in signal processing, computer vision, finance, meteorology and others \citep{feng2000human, fabozzi2007robust, ruppert2011statistics,lorenz1956empirical}. 
Accordingly, an extensive literature has developed around the asymptotic analysis of the distributions of these eigenvalues, denoted $\lambda_1(\hat\bfSigma)\geq\cdots\geq\lambda_p(\hat\bfSigma)$, where $\hat\bfSigma$ is a $p\times p$ sample covariance matrix constructed from $n$ observations \citep{bai2004, yao2015, tao2012topics}.

Within this literature, one of the most prominent lines of research has dealt with a fundamental phenomenon known as the Baik-Ben Arous-P\'ech\'e (BBP) phase transition~\citep{bbp}. Under certain conditions, when $n$ and $p$ diverge proportionally, the BBP phase transition refers to the fact that $\lambda_1(\hat\bfSigma)$ exhibits two types of behavior, depending on whether the largest eigenvalue $\lambda_1(\bfSigma)$ of the population covariance matrix $\bfSigma$ resides in the ``subcritical'' or ``supercritical'' regime. (A precise formulation of these regimes will be given in~\eqref{hypothesis0} below.) In brief, the phase transition entails that
 if $\lambda_1(\bfSigma)$ is subcritical, then $\lambda_1(\hat\bfSigma)$ has fluctuations of order $n^{-2/3}$ that can be approximated with a Tracy-Widom distribution \citep{johnstone2001, bao2015, leeschnelli2016, knowles2017anisotropic}, whereas if $\lambda_1(\bfSigma)$ is supercritical, then $\lambda_1(\hat\bfSigma)$ has fluctuations of order $n^{-1/2}$ that can be approximated with a Gaussian distribution \citep{paul2007asymptotics, bai2008central, cai_han_pan,  Zhang:2022}. 
Due to this phase transition, it is crucial for practitioners to determine which regime underlies a given dataset when using inference procedures related to $\lambda_1(\hat\bfSigma)$.
Yet, the statistical problem of using data to answer this question is far from resolved.

In this paper, we propose a solution that is framed in terms of a hypothesis testing problem that can be informally stated as
\begin{align} 
 \label{testing_problem_heu}
    \mathsf{H}_0: \lambda_1(\bfSigma) \textnormal{ is subcritical \quad vs. \quad } 
     \mathsf{H}_1: \lambda_1(\bfSigma) \textnormal{ is supercritical}. 
\end{align}
To the best of our knowledge, 
the problem~\eqref{testing_problem_heu}  has not been systematically addressed in the literature.  But, at first sight, this problem might be conflated with one that has been extensively studied---the problem of detecting ``spiked'' population eigenvalues  \citep{ding2022tracy, onatski_et_al_2014, johnstone_onatski_2020}---and so it is important to clarify why the two problems are essentially different. For this purpose, it is helpful to consider the simplest scenario of a spike detection problem, which deals with a population covariance matrix of the form $\bfSigma=\textup{diag}(\lambda_1(\bfSigma),1,\dots,1)$, where the largest population eigenvalue is said to be a spike if $\lambda_1(\bfSigma)>1$. In this context, the goal is to test the null $\mathsf{H}_0':\lambda_1(\bfSigma)=1$ against the alternative $\mathsf{H}_1':\lambda_1(\bfSigma)>1$~\citep{Johnstone:2018}. When $\bfSigma$ has this particular structure, $\lambda_1(\bfSigma)$ is regarded as subcritical (resp.~supercritical) if it is slightly below (resp.~above) the threshold $1+\sqrt{p/n}$ \citep{BAI2012167, Johnstone:2018}.
The key distinction with respect to the problem~\eqref{testing_problem_heu} is that it is possible for $\lambda_1(\bfSigma)$ to be either subcritical or supercritical under $\mathsf{H}_1'$. \emph{Consequently, even when $\mathsf{H}_1'$ is correctly detected, this does not reveal the type of fluctuations that  $\lambda_1(\hat\bfSigma)$ has.} Furthermore, this issue  persists in more general forms of spiked covariance models.

\subsection{Contributions} In our effort to determine the regime that $\lambda_1(\bfSigma)$ occupies, 
we offer four main contributions:
\begin{enumerate}
    \item We provide a new perspective, by developing a novel formulation of the hypotheses in~\eqref{testing_problem_heu}. Also, this formulation is applicable to a broad class of population covariance matrices $\bfSigma$, and is not limited to stylized spiked models that are often used in the literature. 
    
    \item We propose a test statistic of the form 
    \begin{equation}
        T_n=\ts\frac{n^{2/3}}{\hat\sigma_n}(\lambda_1(\hat\bfSigma)-\lambda_2(\hat\bfSigma))
    \end{equation}
     which involves a new estimate $\hat\sigma_n$ for a certain scale parameter $\sigma_n$ of the leading sample eigengap $\lambda_1(\hat\bfSigma)-\lambda_2(\hat\bfSigma)$. (See~\eqref{def_sigma} and~\eqref{def_sigma_hat} for definitions of $\sigma_n$ and $\hat\sigma_n$ respectively.) The estimation of $\sigma_n$ has been recognized as an important  problem in the literature \citep{Passemier:2014, cai_han_pan, ding2022tracy}, and our proposed estimate $\hat\sigma_n$ is the first that is supported by a consistency guarantee in the subcritical regime.
   It is also notable that previous works on related test statistics, such as the widely-used \emph{Onatski's statistic}~\eqref{eqn:onatskinew}, have sidestepped this scale estimation problem by relying on ratios of sample eigengaps. As we will explain in Section~\ref{sec_gap_statisitcs}, the use of gap ratios has substantial drawbacks that our approach is able to overcome. Furthermore, we demonstrate empirically that $T_n$ often achieves higher power than Onatski's statistic in the context of the problem~\eqref{testing_problem_heu}.
   
\item We prove that the level of the test statistic $T_n$ can be controlled asymptotically, and that it is power-consistent for certain alternatives.

\item As an application of the ideas underlying our proposed test, we develop a new bootstrap method for functionals of leading sample eigenvalues---which is the first method of this type that is known to be consistent in the subcritical regime. This is of particular interest, as the breakdown of the standard non-parametric bootstrap in the subcritical regime is known to be a serious issue~\citep{Karoui:2019}.
\end{enumerate}

\subsection{Problem formulation}\label{sec_formulation} A complete set of theoretical assumptions will be given in Section~\ref{sec_prelim}, but it is possible to explain our problem formulation here with only a bit of notation. For any $k=0,1,\dots,p-1$ such that $\lambda_{k+1}(\bfSigma)>0$, define the parameter $\xi_{n,k}$ to be the unique value in the interval $(0,1/\lambda_{k+1}(\bfSigma))$ that solves the equation
\begin{equation} \label{eq_def_xi}
    \frac{1}{p}\sum_{j=k+1}^p \Big(\frac{\lambda_j(\bfSigma)\xi_{n,k}}{1-\lambda_j(\bfSigma)\xi_{n,k}}\Big)^2 = \frac{n}{p}.
\end{equation}
In the random matrix theory literature, the subcritical regime is often defined in an asymptotic manner by requiring that $\lambda_1(\bfSigma)$ satisfy
\begin{equation}\label{eqn:usualdef}
    \limsup_{n\to\infty} \lambda_1(\bfSigma)\xi_{n,0} < 1,
\end{equation}
where $\bfSigma$ is implicitly regarded as a function of $n$, and $p/n$ converges to a positive constant as $n\to\infty$ \citep{leeschnelli2016, elkaroui2007}. Consequently, the reciprocal $1/\xi_{n,0}$ is often interpreted as a threshold, such that the size of $\lambda_1(\bfSigma)$ relative to $1/\xi_{n,0}$ determines whether the subcritical or supercritical regime holds.

However, with regard to hypothesis testing, the condition~\eqref{eqn:usualdef} is unsuitable as a null hypothesis in several respects:
\begin{itemize}
\item First, we wish to formulate the testing problem~\eqref{testing_problem_heu} from the perspective of a practitioner who is trying to understand the distribution that generated the data at hand.
From this standpoint, the condition~\eqref{eqn:usualdef} is not a natural hypothesis to test, because it involves a sequence of many distributions, and it does not refer specifically to the data-generating distribution of interest.

\item Second, if the condition~\eqref{eqn:usualdef} is viewed as a null hypothesis, it allows the class of ``null models'' to contain members that are arbitrarily close to the ``decision boundary'' where $\lambda_1(\bfSigma)\xi_{n,0}= 1$. This is problematic, because when the quantity $\lambda_1(\bfSigma)\xi_{n,0}$ is very close to 1, asymptotic approximations based on the subcritical and supercritical regimes may \emph{both} be inappropriate, as will be discussed below. 

\item Third, the definition of $\xi_{n,0}$ makes it impossible for the condition $\lambda_1(\bfSigma)>1/\xi_{n,0}$ to hold, which complicates the issue of defining an alternative hypothesis that negates~\eqref{eqn:usualdef}. 
\end{itemize}

To deal with the considerations just discussed, we propose to formalize the problem~\eqref{testing_problem_heu} as
 \begin{align} \label{hypothesis0}
      \mathsf{H}_{0,n}:  
       \lambda_1(\bfSigma)\leq \frac{1}{(1+\ve)\xin{1}}
      \quad \textnormal{vs.} \quad
   \mathsf{H}_{1,n}:  \lambda_1(\bfSigma)\geq \frac{1}{(1 -\ve)\xin{1}},
 \end{align}
 where $\varepsilon\in (0,1)$ is a user-specified parameter that is regarded as fixed with respect to $n$. 
 Here,  we use $\xi_{n,1}$ rather than $\xi_{n,0}$ for two reasons. One is that the definition of $\xi_{n,1}$ allows $\lambda_1(\bfSigma)>1/\xi_{n,1}$, whereas the definition of $\xi_{n,0}$ does not allow  $\lambda_1(\bfSigma)>1/\xi_{n,0}$. Another reason is that, \rev{under the null}, $\xi_{n,1}$ is close enough to $\xi_{n,0}$ so that it can still be interpreted in the same way as $\xi_{n,0}$. In particular, the difference $\xi_{n,1}-\xi_{n,0}$ turns out to be positive and of order $\mathcal{O}(1/p)$ under conventional assumptions (see Lemma~\ref{lem_xin_xi1}), and so if $\mathsf{H}_{0,n}$ holds for all large $n$, then it follows that the usual subcritical condition~\eqref{eqn:usualdef} also holds. \\

 \noindent\emph{The necessity of $\varepsilon$}. The parameter $\varepsilon$ represents a degree of separation between $\mathsf{H}_{0,n}$ and $\mathsf{H}_{1,n}$, and it provides flexibility to control the difficulty of the testing problem. 
Parameters of this type appear frequently 
in non-parametric settings, where testing problems can become intractable without such a degree separation~\citep[][\textsection 6.2]{Gine:2021}~\citep[][\textsection 1.3]{Ingster:2003}. Further examples of settings where similar types of parameters are used include property testing problems in large-scale computing~\citep[][\textsection 11.3]{Ron:2010}~\citep[][\textsection 12.1]{Goldreich:2017}, as well as the testing of ``relevant hypotheses,'' which has recently gained traction in the statistical literature~\citep{Dette:2019,Dette:2020}.

In the present context, there is a specific reason why the phase transition leads to formulating the problem~\eqref{hypothesis0} in terms of $\varepsilon$. It is due to a very fine-grained effect that can occur when \smash{$|\lambda_1(\bfSigma)-1/\xi_{n,1}|=\mathcal{O}(n^{-1/3})$}, i.e.~when $\lambda_1(\bfSigma)$ is in a vanishing interval of radius $\mathcal{O}(n^{-1/3})$ centered at the threshold $1/\xi_{n,1}$. In this vanishing critical region, the fluctuations of $\lambda_1(\hat\bfSigma)$ may not be well approximated by a Tracy-Widom or a Gaussian distribution, but by some hybrid of the two~\citep{bbp,Bao:2022}.  Hence, the separation provided by $\varepsilon$ ensures that the hypotheses in~\eqref{hypothesis0} meaningfully distinguish between Tracy-Widom and Gaussian fluctuations. 

\rev{
Currently, a detailed theoretical description
of the radius of the critical region (e.g.~its dependence on various model parameters) is not known for data-generating models as general as the one we use in \ref{ass_mp_regime}-\ref{ass_lsd} of Section 2. Nevertheless, because $\varepsilon$ is held fixed with respect to $n$, the interval $(\frac{1}{(1+\varepsilon)\xi_{n,1}},\frac{1}{(1 -\varepsilon)\xi_{n,1}})$ will eventually contain the critical region when $n$ is large enough. Thus, the use of a fixed $\varepsilon$ provides a mathematically tractable way to exclude the critical region from the testing problem without knowing the precise extent of the region. Also, there is another reason for holding $\varepsilon$ fixed, which is that if $\varepsilon$ were allowed to approach $0$ as $n\to\infty$, then $\lambda_1(\bfSigma)$ could become arbitrarily close to $1/\xi_{n,1}$, and this would interfere with our use of existing asymptotic theory for $\lambda_1(\hat\bfSigma)$. Namely, existing results in the subcritical regime that are applicable under \ref{ass_mp_regime}-\ref{ass_lsd} require a fixed amount of separation between $\lambda_1(\bfSigma)$ and $1/\xi_{n,1}$ in order to ensure a Tracy-Widom limit for $\lambda_1(\hat\bfSigma)$. 
See also \citep{bykhovskaya2025weak} for additional discussion of how the delicate nature of the critical region creates challenges for inference.
}

Regarding the selection and interpretation of $\varepsilon$, it should be noted first that this parameter is \emph{unitless}, as it specifies a proportion of the unitless quantity $\lambda_1(\bfSigma)\xi_{n,1}$, and so there is no need to match the scale of $\varepsilon$ to data.
Likewise, the simplest way of selecting $\varepsilon$ is in the same type of discretionary manner that a significance level is usually selected. 
Another way of thinking about $\varepsilon$ is to consider the relationship between significance levels and p-values, and then pursue an analogous relationship involving $\varepsilon$. More specifically, it is possible to compute the smallest $\varepsilon$ for which rejection of $\mathsf{H}_{0,n}$ occurs while holding the significance level $\alpha$ fixed. If this computed value is denoted as $\hat \varepsilon$ (i.e.~a p-value analogue), and if it turns out that $\hat\varepsilon$ is small compared to $n^{-1/3}$, then this could be interpreted as solid evidence against the subcritical regime. This is because smaller values of $\varepsilon$ expand the set of distributions satisfying $\mathsf{H}_{0,n}$ (making rejection harder), while the notion of ``small'' should be judged relative to $n^{-1/3}$, based on the radius of the interval around $1/\xi_{n,1}$ discussed earlier.
The quantity $\hat \varepsilon$ can also provide further insights, such as by producing an asymptotically valid one-sided confidence interval for $\lambda_1(\bfSigma)$ under $\mathsf{H}_{0,n}$, which will be established in Corollary~\ref{cor:CI}.

In any case, these various possibilities should not overshadow a more fundamental point:
A certain degree of separation $\varepsilon$ really is necessary to make the distinction between the Tracy-Widom and Gaussian approximations for $\lambda_1(\hat\bfSigma)$ conceptually meaningful.

\subsection{Gap statistics and scale estimation}  \label{sec_gap_statisitcs}
The leading sample eigengap $\lambda_1(\hat\bfSigma)-\lambda_2(\hat\bfSigma)$ is an intuitive statistic for detecting the alternative $\mathsf{H}_{1,n}$. However, this statistic by itself is not adequate for testing, since its scale is unknown. This issue also occurs when \smash{$\lambda_1(\hat\bfSigma)-\lambda_2(\hat\bfSigma)$} is considered as a possible statistic for the distinct task of spike detection. In the spike detection context, a well-established strategy for dealing with the unknown scale is to use statistics that are constructed from \emph{gap ratios}. This strategy owes much of its popularity and influence to the seminal work of Onatski~\citep{onatski2009testing}, who introduced the statistic
\begin{align}\label{eqn:onatskinew}
    R_n(\kappa) = \max_{1 \leq j \leq \kappa} \frac{\lambda_j(\bfSigmahat) - \lambda_{j+1}(\bfSigmahat)}{\lambda_{j+1}(\bfSigmahat) - \lambda_{j+2}(\bfSigmahat)},
\end{align}
where $\kappa\geq 1$ is a fixed integer that plays the role of a tuning parameter. A valuable feature of the statistic $R_n(\kappa)$ is that it is asymptotically pivotal in the subcritical regime~\citep{onatski2009testing}, so that no parameters in the limiting distribution need to be estimated. Likewise, it would be reasonable to consider $R_n(\kappa)$ as a candidate test statistic for the proposed testing problem~\eqref{hypothesis0}, but as we explain below, this statistic has serious limitations that we seek to overcome. \\

\noindent\emph{Limitations of gap ratios.} The asymptotically pivotal property of $R_n(\kappa)$ comes at a steep price, as the performance of this statistic is contingent on the selection of $\kappa$. In essence, the source of difficulty is that the statistic $R_n(\kappa)$ will typically not be power-consistent for $\mathsf{H}_{1,n}$ whenever $\bfSigma$ has more than $\kappa$ leading eigenvalues that are well-separated from the bulk eigenvalues. This phenomenon arises because if $\lambda_1(\bfSigma),\dots,\lambda_{\kappa+1}(\bfSigma)$ are well-separated from the bulk, then for each $j=1,\dots,\kappa$, both the numerator and the denominator in~\eqref{eqn:onatskinew} tend to be of the same order of magnitude \citep{Zhang:2022}, and in such cases, the statistic $R_n(\kappa)$ will fail to diverge asymptotically.

To avoid this breakdown in the power of $R_n(\kappa)$, users might try to ``hedge their bets'' by favoring conservatively large choices of $\kappa$.
That is, if $\kappa$ is large enough so that there exists some $j\in\{1,\dots,\kappa\}$ such that $\lambda_j(\bfSigma)$ is well separated from the bulk and $\lambda_{j+1}(\bfSigma)$ is near the bulk, then the sample eigengaps $\lambda_j(\hat \bfSigma)-\lambda_{j+1}(\hat\bfSigma)$ and $\lambda_{j+1}(\hat \bfSigma)-\lambda_{j+2}(\hat\bfSigma)$ will tend to be of different orders---which will cause $R_n(\kappa)$ to diverge asymptotically. (See~\cite{knowles2017anisotropic, Zhang:2022} or the proof of Theorem \ref{thm_alternative}.) However, it turns out that conservatively large choices of $\kappa$ still erode power, because the critical values of $R_n(\kappa)$ increase substantially with $\kappa$. Thus, there is a higher ``hurdle'' that must be cleared in order for rejections to occur, and the negative impact on power has previously been noted in the literature~\citep{ding2022tracy}\label{Onatski_disc}.\label{introdisc} A conceptual illustration of how the choice of $\kappa$ affects the power of $R_n(\kappa)$ in the testing problem~\eqref{hypothesis0} is given in Figure~\ref{fig:TW}. Furthermore, we demonstrate these drawbacks of $R_n(\kappa)$ in the numerical results presented in Section~\ref{sec_expt}.\\

\begin{figure} 
    \centering
    \includegraphics[width=0.8\linewidth]{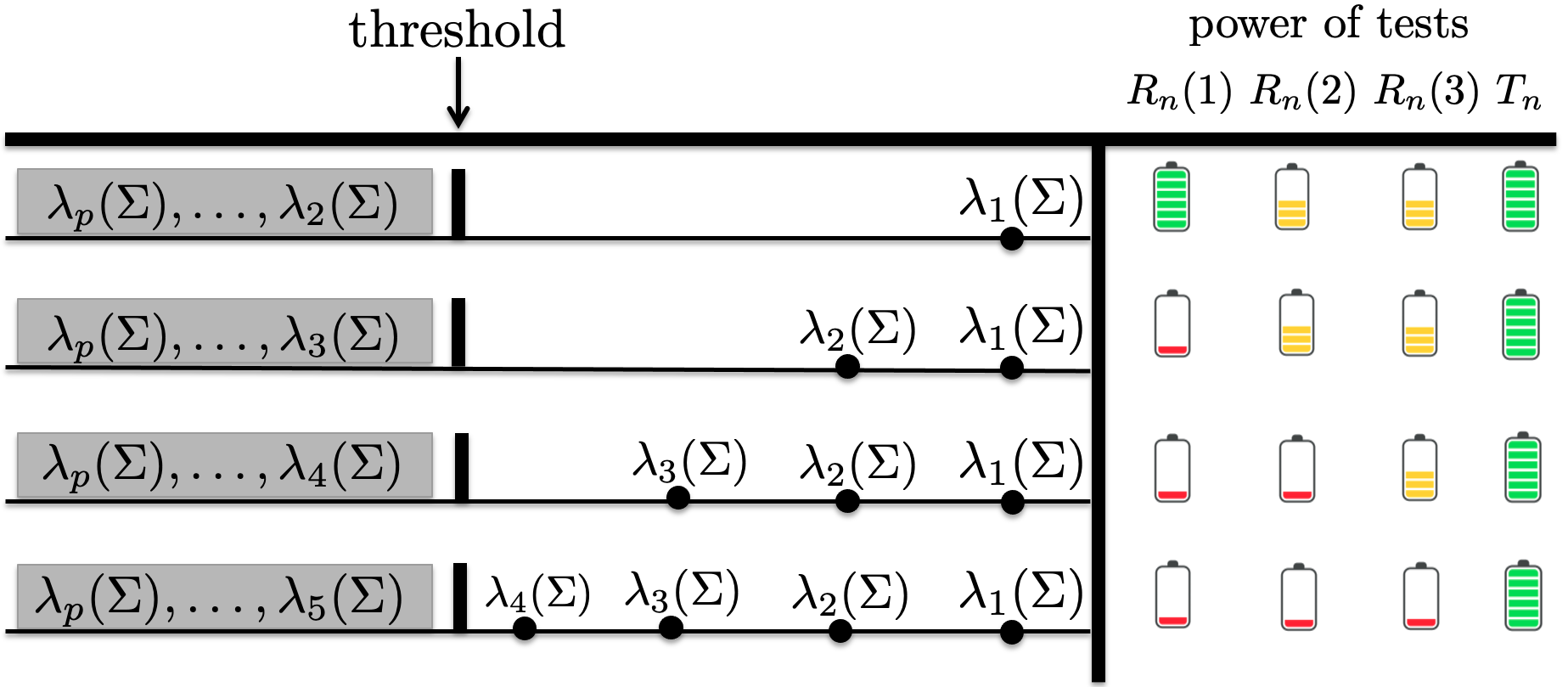}
    \caption{ A conceptual comparison of the proposed statistic $T_n$ with $R_n(\kappa)$ for $\kappa=1,2,3$ in detecting four instances of $\mathsf{H}_{1,n}$. Green cells correspond to cases where a test is expected to have power approaching 1 as $n\to\infty$. Yellow cells correspond to cases where $R_n(\kappa)$ may have power approaching 1 as $n\to\infty$, but with reduced power from taking a maximum over several gap ratios.  Red cells correspond to cases where $R_n(\kappa)$ is not expected to have power approaching 1 as $n\to\infty$. }
    \label{fig:TW}
\end{figure}

\noindent\emph{Direct scale estimation.} Based on the limitations of gap ratios, we propose a different approach, which is to construct a \emph{direct estimate} $\hat\sigma_n$ for the unknown scale of $\lambda_1(\hat\bfSigma)-\lambda_2(\hat\bfSigma)$, and use a test statistic of the form 
\begin{equation}\label{eqn:Tndef}
    T_n =  \frac{n^{2/3}}{\hat\sigma_n} \big(  \lambdahat - \lambda_2 (\bfSigmahat ) \big). 
\end{equation}
Although this statistic may appear simple at first sight, estimating the scale of $\lambda_1(\hat\bfSigma)-\lambda_2(\hat\bfSigma)$ is a challenging problem. As a matter of fact, we are not aware of any consistency guarantees for scale estimates of $\lambda_1(\hat\bfSigma)-\lambda_2(\hat\bfSigma)$ in settings like those considered here, but there are some estimates that have been proposed without guarantees \citep{cai_han_pan, Passemier:2014}. In Section~\ref{sec_sigma}, we propose a new method for this scale estimation problem, and in Proposition~\ref{thm_consistency_xi} we establish its consistency within the subcritical regime. Consequently, it is not necessary for us to rely on gap ratios, and the proposed test statistic $T_n$ is able to avoid the power losses associated with the statistic $R_n(\kappa)$ discussed earlier. Furthermore, in Section~\ref{sec_expt}, we confirm that this power advantage holds in simulations across a range of settings, and in Section \ref{sec_stock} we present a real-data example in which $T_n$ yields smaller p-values than $R_n(\kappa)$. 
More generally, considering that $R_n(\kappa)$ has been extensively used in
applications~\citep[e.g.][]{Heckman:2013,Bauer:2015,Miranda:2015,Kabundi:2020} and has been widely embraced in theoretical works \citep{ding2022tracy, leeschnelli2016, ding2018necessary}, the statistic $T_n$ may have benefits that extend beyond the scope of the problem~\eqref{hypothesis0}.

\subsection{A consistent bootstrap for the subcritical regime}

The use of bootstrap methods to approximate the distributions of statistics that arise in PCA is a fundamental topic that has sustained long-term interest~\citep[e.g.][]{Diaconis:1983,BeranS85,EatonT91,Dumbgen:1993,Hall:2009,Han:2018,Johnstone:2018,Karoui:2019,naumov2019bootstrap,Lopes:2019:Biometrika,Ding:2023,Wang:Lopes:2023,Lopes:Erichson:Mahoney:2023,Yu:2023,Yu:2024}. One of the primary drivers of this interest is that bootstrap methods allow the distributions of many types of statistics to be approximated in a user-friendly and unified way---which can bypass the inconveniences of working with intricate asymptotic formulas on a case-by-case basis. 

In recent years, attention has been drawn to the fact that in high dimensions, ordinary sampling with replacement can often fail to approximate the distributions of the leading sample eigenvalues~\citep{Johnstone:2018,Karoui:2019}. In particular, it has been an open question as to whether or not it is possible to develop an alternative bootstrap method that is supported by a consistency guarantee in the subcritical regime. This question is our primary motivation for a new bootstrap method that we propose in Section~\ref{sec_stat_appl}, which provides an affirmative answer.

The difficulty of this question can be attributed to the fact that the behavior of $\lambda_1(\hat\bfSigma),\dots,\lambda_k(\hat\bfSigma)$ for fixed $k\geq 1$ depends heavily on which regime the model occupies. Consequently, because bootstrap methods are based on using an estimated model to generate simulated data, they are unlikely to work for statistics depending on $\lambda_1(\hat\bfSigma),\dots,\lambda_k(\hat\bfSigma)$ if the estimated model is in the wrong regime. 

To deal with this challenge, we apply our proposed test to the problem~\eqref{hypothesis0}, and if $\mathsf{H}_{0,n}$ is not rejected, then we generate data from a ``bootstrap world'' that is designed to satisfy an approximate version of $\mathsf{H}_{0,n}$. More specifically, the population eigenvalues in the bootstrap world
are constructed so that they are upper bounded by an estimate of $1/((1+\varepsilon)\xi_{n,1})$, which is detailed in Section~\ref{sec_xi}. Another key idea in designing the bootstrap world is to take advantage of a fundamental universality property of $\lambda_1(\hat\bfSigma),\dots,\lambda_k(\hat\bfSigma)$ that holds in the subcritical regime. Namely, the joint distribution of $\lambda_1(\hat\bfSigma),\dots,\lambda_k(\hat\bfSigma)$ is asymptotically the same as if the data in the real world are generated from the Gaussian distribution $\mathcal{N}(0,\bold{\Lambda})$, where $\bold{\Lambda}=\textup{diag}(\lambda_1(\bfSigma),\dots,\lambda_p(\bfSigma))$~\citep{knowles2017anisotropic}. Accordingly, our proposed bootstrap method generates data from a fitted Gaussian distribution $\mathcal{N}(0,\tilde{\bold{\Lambda}})$, where $\tilde{\bold{\Lambda}}$ is a specially designed estimate of $\bold{\Lambda}$ that approximately conforms with the subcritical regime.

\subsection{Examples of use cases }

Whereas the preceding paragraphs have described the importance of the testing problem~\eqref{hypothesis0} at a high level, it is worth mentioning some of the concrete inference tasks that our work can inform. \\

 \noindent \emph{Selecting principal components and factors.} One of the standard uses of Onatski's statistic is in choosing the number of factors in factor models, or the number of principal components in PCA. 
Likewise, our approach to the testing problem~\eqref{hypothesis0} can also be adapted to estimating this number, and we provide a theoretical guarantee in Theorem~\ref{thm_est_spike}. See also Section~\ref{sec_factor} for further discussion of factor models.\\

 \noindent \emph{Eigenvector inference.} The location of $\lambda_1(\bfSigma)$ not only determines a phase transition in the behavior of $\lambda_1(\hat\bfSigma)$, but also in the behavior of the leading sample eigenvector. That is, in the subcritical regime, the leading sample eigenvector will be delocalized asymptotically, whereas in the supercritical regime, it will have a non-vanishing inner product with the leading population eigenvector~\citep[e.g.][]{paul2007asymptotics,BloemendalKnowlesYauYin2016,ding:yang:2021}. Thus, in distinguishing the two regimes, the proposed test clarifies the informativeness of the leading sample eigenvector.\\

 \noindent \emph{Estimation and inference for eigenvalues.} Estimating the eigenvalues of $\bfSigma$ is a problem of major interest in high-dimensional statistics~\citep{ledoit2015spectrum,donoho2018}. One of the main challenges is the upward bias of the leading sample eigenvalues. In some specialized models, analytical formulas have been developed to quantify the bias, which can be inverted to yield bias-corrected eigenvalue estimates~\citep[e.g.][]{baik2006eigenvalues,donoho2018}. By contrast, the proposed bootstrap method offers a route to bias correction that has the potential to be used in cases where analytical formulas for bias are not available. See Sections~\ref{sec_stat_appl} and~\ref{sec_sim_bootstrap} for additional details. Lastly, with regard to inference, we show how to construct explicit one-sided confidence intervals for $\lambda_1(\bfSigma)$ in Section~\ref{sec_null}, and we provide a corresponding coverage guarantee in Corollary~\ref{cor:CI}.\\

  \noindent \emph{Bootstrapping.} In high-dimensional settings, it is not always clear when ordinary sampling with replacement accurately approximates the distributions of functionals of sample covariance matrices. For example, in some cases where the leading population eigenvalues are well separated from the bulk, it is known that sampling with replacement can succeed even when the dimension is high~\citep{Yao:Lopes:2023b}, but that it may fail in the subcritical regime~\citep{Karoui:2019}. In this way, the proposed test and the accompanying bootstrap method in Section~\ref{sec_stat_appl} can offer both guidance and a solution.

\rev{\subsection{Other related work} \label{sec_related_work}
Here, we briefly discuss the related work \cite{bykhovskaya2025weak}, 
who develop a novel confidence interval for $\lambda_1(\bfSigma)$ when the data are Gaussian and $\bfSigma$ is a spiked covariance matrix, i.e.~a low-rank perturbation of the identity matrix. The important property of this interval is that it can succeed regardless of the regime of the underlying model: subcritical, critical, or supercritical. In comparison with our work, a key distinction is that the covariance structure used in~\citep{bykhovskaya2025weak} ensures that the critical threshold is equal to $1+\sqrt{p/n}$, and does not require estimation. Similarly, the scale parameter of $\lambda_1(\hat\bfSigma)-\lambda_2(\hat\bfSigma)$ need not be estimated in that context. By contrast, the challenges of estimating these parameters are central to the present work (see Proposition \ref{thm_consistency_xi}), where the data need not be Gaussian and $\bfSigma$ need not be a spiked covariance matrix.
}

\section{Preliminaries and setup} \label{sec_prelim}
Let $\bold Y\in\R^{n\times p}$ be a random data matrix whose rows consist of centered i.i.d.~observations in $\R^p$. The associated  sample covariance matrix is defined as
\begin{equation}
    \bfSigmahat  = \frac{1}{n}\bold Y\ttop \bold Y
\end{equation}
and its population counterpart is 
\begin{equation}
    \bfSigma=\E(\bfSigmahat).
\end{equation}
In addition, define the $n\times n$ companion matrix associated with $\bfSigmahat$ as
\begin{align}  \label{eq_def_companion}
\bfSigmahatu = \frac{1}{n} \bold Y\bold Y\ttop.
\end{align}
For any real $p\times p$ symmetric matrix $\bfA$, let $F^{\bfA}$ denote the empirical spectral distribution function of $\mathbf A$, which satisfies
\begin{equation*}
    F^{\bold A}(t)=\frac{1}{p}\sum_{j=1}^p 1\{\lambda_j(\bold A)\leq t\}
\end{equation*}
for any $t\in \R$, where $1\{\cdot\}$ is an indicator function. In the particular case when $\bold A=\bfSigma$, we write $$H_n=F^{\bfSigma}.$$
\noindent\emph{Stieltjes transforms and related parameters.} For any distribution $F$, we denote its Stieltjes transform by
\begin{align}\label{eq_stieltjes}
	s_F (z) = \int \frac{1}{\lambda - z} d F(\lambda), \quad z \in \mathbb{C}^+ = \{ z \in \mathbb{C} : \operatorname{Im}(z)>0 \}. 
\end{align}
In the particular case when $F=F^{\underline{\hat\Sigma}}$, we use the shorthand $\su_n=s_{F^{\underline{\hat\bfSigma}}}$.
{\rev{We also use $\su_n^0(z)$ to denote the unique solution to the equation
\begin{align}\label{eq_mp_stieltjes}
		z = - \frac{1}{\su_n^0(z)} \ + \ y_n \int \frac{\lambda}{1 + \su_n^0(z) \lambda} dH_n(\lambda), \quad z \in \mathbb{C}^+,
	\end{align}
    where $y_n=p/n$.
It can be shown as a consequence of this equation that $\su_n^0(z)$ is the Stieltjes transform of a distribution, which is denoted as $\Fu^{y_n, H_n}$. (Due to the Stieltjes-Perron inversion formula, it follows that $\Fu^{y_n, H_n}$ is implicitly defined by $\su_n^0(z)$.) Likewise, using the notation introduced in~\eqref{eq_stieltjes}, we have $\su_n^0(z)=s_{\Fu^{y_n, H_n}}(z)$. From a heuristic standpoint, $\Fu^{y_n, H_n}$ may be regarded as a finite-sample proxy for the weak almost-sure limit of  $F^{\underline{\hat\bfSigma}}$. Further background on this equation and the characterization of the distribution $\Fu^{y_n, H_n}$ can be found in~\citep[][\textsection 2.4]{yao2015}.
}
}

The rightmost endpoint of the support of the distribution $\Fu^{y_n, H_n}$ is denoted by $r_n$ and is given by the formula~\citep[Lemmas 2.5-2.6]{knowles2017anisotropic},
 \begin{align}
 \label{eqn:rndef}
 	r_n = \frac{1}{\xi_{n,0}} \lb 1 + y_n \int \frac{\lambda \xi_{n,0}}{1 - \lambda \xi_{n,0}} dH_n(\lambda) \rb.
 \end{align}
The parameter $\xi_{n,0}$ is also related to $\su_n^0$ and $r_n$ through the relation 
 \begin{align} \label{eq_xi_formula}
 	\xi_{n,0} =  - \lim_{\eta \to 0} \su_n^0 ( r_n + i \eta)=-\su_n^0 ( r_n), 
 \end{align}
 which was established in~\citep{silverstein1995analysis}. Furthermore, $\xi_{n,0}$ is needed to define the scale parameter $\sigma_n$ for the leading sample eigengap $\lambda_1(\hat \bfSigma)-\lambda_2(\hat\bfSigma)$ in the subcritical regime through the formula
 \begin{align} \label{def_sigma}
  \sigma_n^3 =  \frac{1}{\xi_{n,0}^3} \bigg( 1 + y_n \int \Big( \frac{\lambda \xi_{n,0}}{1 - \lambda\xi_{n,0}}\Big)^3 dH_n(\lambda) \bigg).
 \end{align}
 The interpretation of $\sigma_n$ as a scale parameter arises from  the following fundamental result~\citep{bao2015,leeschnelli2016,knowles2017anisotropic}: If the subcritical regime holds within the data-generating model given by~\ref{ass_mp_regime},~\ref{ass_mom}, and~\ref{ass_lsd} below, then for any fixed number $d\geq 1$, the random vector $\frac{n^{2/3}}{\sigma_n}(\lambda_j(\hat\bfSigma)-r_n)_{1\leq j\leq d}$ converges weakly to a $d$-dimensional Tracy-Widom distribution of type 1, as defined below.\\

\noindent\emph{The multivariate Tracy-Widom distribution.} 
Let ${\bf{G}}\in \R^{N\times N}$ be a random matrix with i.i.d.~$\mathcal{N}(0,1)$ entries, and let {\rev{${\bf{W}}=\frac{1}{\sqrt{2N}}({\bf{G}}+{\bf{G}}\ttop)$}} be a Wigner matrix. For any fixed $d\geq 1$, it is known that as $N\to\infty$, the random vector $N^{2/3}(\lambda_j(\mathbf{W})-2)_{1\leq j\leq d}$
has a weak limit that is referred to as the $d$-dimensional Tracy-Widom distribution \rev{of type $1$}. In particular, this distribution can be numerically approximated by generating Wigner matrices $\bf{W}$ for large $N$.
 ~\\

\noindent\emph{Data-generating model.} Throughout the paper, our work will be based on a data-generating model that is defined by the following conditions.
  \begin{enumerate}[label=A\arabic*]
  \item \label{ass_mp_regime}\!\!. There is a number  $y \in (0,\infty)\setminus \{1\} $ such that $y_n = p/n$ satisfies $\lim_{n\to\infty} y_n = y$.
  \item \label{ass_mom}\!\!.  The data matrix $\bold Y\in\R^{n\times p}$ has the form $\bold Y=\bold X \bfSigma^{1/2}$, where the matrix $\bold X$ is the upper-left $n\times p$ block of a doubly-infinite array of i.i.d.~random variables $\{x_{ij}:i\geq 1, j\geq 1\}$  such that $\E ( x_{11})=0, \E ( x_{11}^2) = 1$ and $\sup_{q\geq 1} q^{-1/\beta}(\E|x_{11}|^q)^{1/q}<\infty$ for some $\beta>0$.
  \item \label{ass_lsd}\!\!. As $n\to\infty$, the distribution $H_n$ has a weak limit $H$. Also, the support of $H$ is a \smash{finite}  union of compact intervals in $(0,\infty)$, and there is a fixed compact interval in $(0, \infty)  $ containing the support of $H_n$ for all large $n$.
 \end{enumerate}
\rev{The condition $y \neq 1$ in \ref{ass_mp_regime} arises from our use of theoretical guarantees for the QuEST estimator~\citep{ledoit2015spectrum}, which rely on it.}

 \section{Testing procedure}\label{sec_method}
Before formally defining our testing procedure, we will assemble estimators for $\xi_{n,0}$, $\xi_{n,1}$, $H_n$, and $\sigma_n$ in the following three paragraphs.\\

 \noindent\emph{Estimation of $\xi_{n,0}$ and $\xi_{n,1}$.} Recall the formula $\xi_{n,0}= -\su_n^0 ( r_n)$ from equation~\eqref{eq_xi_formula}. Since $r_n$ approximately corresponds to the right edge of the limiting bulk sample spectrum, it is naturally estimated by $\lambda_1(\hat\bfSigma)$. Meanwhile, the Stieltjes transform $\su_n^0$ is the limit of $\su_n=s_{F^{\underline{\hat\bfSigma}}}$, but because the latter function is not defined at $\lambda_1(\hat\bfSigma)$, we will use the following modified version that drops the term involving $\lambda_1(\hat\bfSigma)$,
 \label{sec_xi}
  \begin{align} \label{def_su_n_tilde}
        \tilde\su_n ( z ) 
  =  \frac{1}{n} \sum\limits_{j=2}^n
  \frac{1}{\lambda_j ( \underline {\bfSigmahat} ) - z } , \quad z \in \mathbb{C} \setminus [ \lambda_n(\underline{\bfSigmahat}), \lambda_2(\underline{\bfSigmahat}) ]. 
  \end{align}
  Likewise, we define the estimate
  \begin{align}\label{eqn:xihatdef}
   \hat\xi_n = 
   \begin{cases}
   - \tilde\su_n ( \lambda_1(\hat\bfSigma)  ) 
  & \textnormal{ if } \lambda_1(\bfSigmahat) > \lambda_2(\bfSigmahat), \\ 
  -\tilde{\su}_n(\lambda_1(\bfSigmahat)+1) & \textnormal{ if } \lambda_1(\bfSigmahat) = \lambda_2(\bfSigmahat).
  \end{cases}
  \end{align}
  We have intentionally chosen to write $\hat\xi_n$ rather than $\hat\xi_{n,0}$, because \rev{under the null}, the parameters $\xi_{n,0}$ and $\xi_{n,1}$ are sufficiently close asymptotically that $\hat\xi_n$ may be regarded as an estimate of either of them. (It will be shown in Lemma~\ref{lem_xin_xi1} that $0< \xi_{n,1}-\xi_{n,0}=\mathcal{O}(1/p)$.)\\

 \noindent\emph{Estimation of $H_n$.}
We estimate the empirical spectral distribution $H_n$ using a modified version of the QuEST procedure (see \cite{ledoit2015spectrum, ledoit2017numerical}). \rev{ Let $\tilde H_{n,Q}$ denote the QuEST estimator for $H_n$, and for each $j=1,\dots,p$, let $\tilde\lambda_{j,Q}$ denote the $(1-(j-1)/p)$-quantile of $\tilde H_{n,Q}$. That is, $ \tilde \lambda_{j,Q} = \inf \{ x \in \R : \tilde H_{n,Q}(x) \geq 1-(j-1)/p \}$}.
The values $\tilde\lambda_{j,Q}$ need to be modified to ensure that they conform with the subcritical regime, and this is achieved with a specially designed truncation. The truncated eigenvalue estimates are defined by 
\begin{align} \label{def_lambda_hat}
    \tilde\lambda_j = 
    \tilde\lambda_{j,Q}\wedge \ts \frac{1}{\hat\xi_n ( 1 + \ve) }
        \quad 1 \leq j \leq p,
\end{align}
so that each $\tilde\lambda_j$ is at most $\frac{1}{\hat\xi_n ( 1 + \ve)}$.
In turn, our estimate of $H_n$ is defined as 
\begin{align} \label{def_H_hat}
    \tilde H_{n}(t) = \frac{1}{p} \sum\limits_{j=1}^p  1\{\tilde\lambda_{j}\leq t\}
\end{align}
for all real $t$, {\rev{so that $\tilde H_n$ mimics $H_n$ by replacing each $\lambda_j(\bfSigma)$ with $\tilde\lambda_j$.}}\\

 \noindent\emph{Estimation of $\sigma_n$.} \label{sec_sigma} Based on the formula~\eqref{def_sigma} for the scale parameter $\sigma_n$, we estimate it by directly substituting the estimates for $\xi_{n,0}$ and $H_n$ constructed above. That is, we define
 \begin{equation}
     \hat \sigma_n^3 =  \frac{1}{\hat\xi_n^3 } \bigg( 1 +  y_n \int \bigg( \frac{\lambda \hat\xi_n }{1 - \lambda\hat\xi_n } \bigg)^3 d\tilde{H}_{n}(\lambda)\bigg).
     \label{def_sigma_hat}
 \end{equation}

 \noindent\emph{Test statistic.} 
To test the null hypothesis \eqref{hypothesis0} for a specified value of $\varepsilon\in(0,1)$, we propose to use the statistic 
 \begin{align*}
      T_n = \frac{n^{2/3}}{\hat\sigma_n} \big( \lambda_1 (\hat \bfSigma) - \lambda_2 (\hat \bfSigma) \big).
 \end{align*}
 Note that the statistic depends on $\varepsilon$ through the eigenvalue estimates in~\eqref{def_lambda_hat} that underpin $\tilde H_n$ and $\hat\sigma_n$.

To specify the rejection criterion, let $(\zeta,\zeta')$ follow a two-dimensional Tracy-Widom distribution \rev{of type $1$}
 and let $q_{1-\alpha}$ be the $(1-\alpha)$-quantile of the difference $\zeta - \zeta'$.
For a prescribed level $\alpha$, the proposed testing procedure rejects the null hypothesis in~\eqref{hypothesis0} when
\begin{align} \label{test1}
T_n>q_{1-\alpha}.
\end{align}

If the null hypothesis in \eqref{hypothesis0} is rejected, it is natural to ask for the number of supercritical eigenvalues in the underlying model. Indeed, this quantity is of fundamental interest in the literature on high-dimensional sample covariance matrices~\cite[e.g.][]{onatski2009testing, passemier2012determining, ding2021spiked, ke2023estimation, cai_han_pan, zhang:zhang:pan:2025}.  
To illustrate  the versatility of our testing approach, we will show how it can be adapted to this task using a sequence of gap statistics.

\rev{For each $k=1,\dots, (p\wedge n) -1$, define the $k$th gap statistic 
\begin{align*}
    T_{n,k}
    &=
    \frac{n^{2/3}}{\hat\sigma_{n,k}}
    \left\{
        \lambda_{k}(\bfSigmahat)
        -
        \lambda_{k+1}(\bfSigmahat)
    \right\},
\end{align*}
where $\hat\sigma_{n,k}$ is a scale estimate that is conceptually similar to $\hat\sigma_n$, but whose precise definition is deferred to equation~\eqref{eqn_hatsigmak}, as it requires introducing several analogues of the quantities used in the definition of $\hat\sigma_n$. Note also that $T_{n,k}$ specializes to $T_n$ when $k=1$. In essence, the proposed estimator $\hat K_n$ can be thought of in terms of a sequential procedure, such that $\hat K_n$ is defined to be smallest $k\geq 0$
for which the event $\{T_{n,k+1}\leq q_{1-\alpha}\}$ occurs. More formally, the estimate is given by the formula %
  \begin{align} \label{eq_def_hat_K}
     \hat K_n = 
 \min  \bigg(\Big\{ 0 \leq k \leq (p\wedge n) -2: T_{n,k+1}\leq q_{1-\alpha} \Big\} \cup \{p\wedge n \} \bigg).
  \end{align}
In this formula, the union with the singleton $\{p\wedge n\}$ avoids the possibility of taking the minimum of an empty set. The performance of  $\hat K_n$ will be described later in Theorem~\ref{thm_est_spike} of Section~\ref{sec_alternative}. }

 \section{Theoretical results under the null hypothesis}\label{sec_null}
 In this section, we establish the limiting null distribution of the proposed test statistic $T_n$ in Theorem~\ref{thm_conv_T_null}. Since this theorem hinges on the performance of the estimates  $\hat \xi_{n}$, $\hat \sigma_n$, and $\tilde H_n$, we first demonstrate their consistency in the following proposition.  For the purpose of stating this result, convergence in probability and distribution are respectively denoted by $\xrightarrow{\P}$ and $\cond$.

 \begin{proposition} \label{thm_consistency_xi}
 Suppose that assumptions \ref{ass_mp_regime}-\ref{ass_lsd} are satisfied and that $\mathsf{H}_{0,n}$ holds for all large $n$. Then, the estimates $\hat\xi_n$, $\hat\sigma_n$, and $\tilde H_n$ satisfy the following limits as $n\to\infty$
 \begin{align} 
 \hat\xi_n - \xi_{n,k}
 &\xrightarrow{\P} 0  \textup{ \ \  for } k\in\{0,1\},\tag{a}\label{eqn:hatxithm}\\
  \tilde H_n & \cond H \textnormal{ in probability},
 \tag{b} \label{eqn:hatHthm}\\
\frac{\hat\sigma_n}{\sigma_n}&\xrightarrow{\P} 1 \tag{c}\label{eqn:hatsigmathm}.
  \end{align}
 \end{proposition}
\noindent These three limits are proven in Section \ref{sec_proof_consistency_xi}. The limit $\hat\sigma_n / \sigma_n \xrightarrow{\P} 1$ may be of  independent interest, since consistent estimation of $\sigma_n$ has not previously been established.

The following theorem is the main result of this section and ensures that the rejection criterion~\eqref{test1} corresponds to a test with an asymptotically controlled level. \rev{Histograms of $T_n$ under the null for different values of $n$ and $p$ are shown in Figure \ref{fig:hist}.}
\begin{figure}[t]
    \centering
    \includegraphics[width=0.7\linewidth]{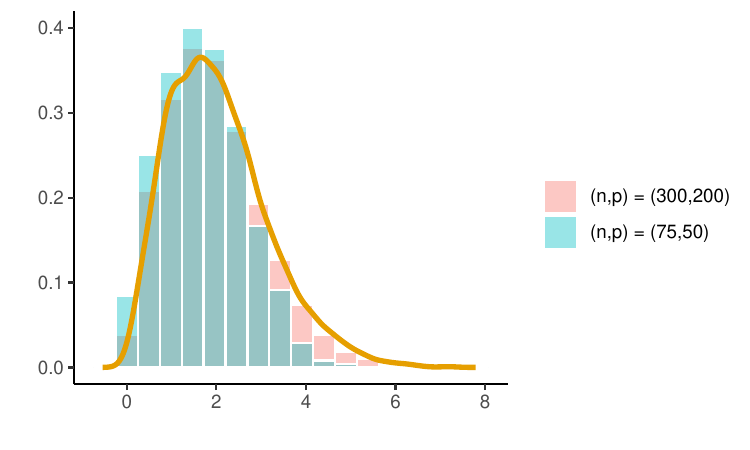}
\caption{\rev{Histograms of $T_n$ for $(n,p)=(75,50),(300,200)$  with $x_{11} \sim \mathcal{N}(0,1)$, based on 5,000 repetitions. The orange curve shows the approximated density of the limiting distribution $\zeta - \zeta'$ obtained from simulations with Wigner matrices.} 
} 
    \label{fig:hist}
\end{figure}

 \begin{theorem} \label{thm_conv_T_null}
    If  \ref{ass_mp_regime}-\ref{ass_lsd} are satisfied and $\mathsf{H}_{0,n}$ holds for all large $n$, then as $n\to\infty$, 
    \begin{align*}
        T_n \cond \zeta - \zeta',
    \end{align*}
    where $(\zeta, \zeta')$ follows a two-dimensional Tracy-Widom distribution \rev{of type $1$}. In particular, for any fixed $\alpha\in(0,1)$, we have $\P(T_n>q_{1-\alpha})\to \alpha$ as $n\to\infty$.
 \end{theorem}
Theorem \ref{thm_conv_T_null} follows by combining Slutsky's theorem with Proposition \ref{thm_consistency_xi}\eqref{eqn:hatsigmathm} and the limit \eqref{eqn:jointTWlimit} given later in Section~\ref{sec_proof_thm_conv_T_null}. Thus, the primary technical challenge in establishing Theorem \ref{thm_conv_T_null} lies in proving Proposition \ref{thm_consistency_xi}. 

As an added benefit of our work so far, it is possible to develop confidence intervals for the parameters $\lambda_1(\bfSigma)$ and $r_n$.
For this purpose, define  $\hat \ve$  as the smallest $\ve>0$ for which $\mathsf{H}_{0,n}$ is rejected based on the statistic $T_n=T_n(\ve)$ at a fixed nominal level $\alpha$. More formally, since $T_n(\ve)$ is a non-decreasing function of $\ve$, we may define $\hat\ve = T_n\inv (q_{1-\alpha}),$ where $T_n\inv (q) = \inf\{\ve>0: T_n(\ve)\geq q\}$ denotes the generalized inverse.

\begin{corollary} \label{cor:CI}
  Fix any $\alpha\in (0,1)$.  If assumptions \ref{ass_mp_regime}-\ref{ass_lsd} are satisfied and if $\mathsf{H}_{0,n}$ holds for all large $n$, then as $n\to\infty$,
     \begin{align*}
         \PR \bigg( \lambda_1(\bfSigma) \in \Big( \textstyle\frac{1}{( 1 + \hat\ve )  \hat\xi_n  },\infty\Big) \bigg)  \geq 1 - \alpha + o(1)
     \end{align*}
     and 
     \begin{align*}
         \PR \bigg(  r_n \in \rev{\Big[\lambdahat - \ts\frac{  q_{1-\alpha/2}' \hat\sigma_n }{n^{2/3}}, \lambdahat - \ts\frac{  q_{\alpha/2}' \hat\sigma_n }{n^{2/3}}\Big]} \bigg) = 1- \alpha+o(1),
     \end{align*}
     where $ q_{\alpha}'$ denotes the $\alpha$-quantile of the one-dimensional Tracy-Widom distribution of type 1.
\end{corollary}
 \noindent The proof of Corollary \ref{cor:CI} is given in Section \ref{sec_proof_consistency_xi}.

\section{Theoretical results under alternative hypotheses} \label{sec_alternative}
 In this section, we study the power of $T_n$ under alternative hypotheses that are particular instances of $\mathsf{H}_{1,n}$, defined in~\eqref{hypothesis0}. These instances specify the particular number of supercritical population eigenvalues, whereas $\mathsf{H}_{1,n}$  allows for various numbers of such eigenvalues. In detail, if   $\varepsilon\in (0,1)$ and $K\geq 1$ are fixed with respect to $n$, then the alternative $\mathsf{H}_{1,n}(K)$ defined below has the interpretation that there are exactly $K$ supercritical population eigenvalues,
\begin{align} \label{def_alternative_general_NEW}
    \mathsf{H}_{1,n}(K): \ \   \lambda_j(\bfSigma)   \geq \frac{1}{(1 - \ve)\xi_{n,j}} \textup{ \ for \  $j=1,\dots,K$, \ \ \ \ and \ \ \ \ }
    \lambda_{K+1}(\bfSigma)  \leq \frac{1}{( 1 + \ve) \xi_{n,K+1}}.
\end{align}

Before stating the main result on power consistency, it is worth pointing out a connection between the hypothesis $\mathsf{H}_{1,n}(K)$ and a standard formulation of supercritical eigenvalues in the random matrix theory literature. For any $\beta\not\in \textup{supp}(H)$, define the function
\begin{equation} \label{eq_def_psi}
    \psi(\beta) = \beta +  y \beta \int \frac{\lambda}{\beta - \lambda} dH(\lambda).
\end{equation}
The derivative $\psi'(\beta)$ exists for all $\beta\not\in\supp(H)$, and for convenience of presentation, we define $\psi'(\beta)=0$ for all $\beta\in\supp(H)$. Conceptually, the derivative is important, because the condition 
$\psi'(\lambda_j(\bfSigma))>0$
is often interpreted to mean that $\lambda_j(\bfSigma)$ is supercritical~\citep{li2020asymptotic, Zhang:2022, jiang_bai}. 
As a concrete example, note that in a simple model where $\bfSigma=\textup{diag}(\lambda_1(\bfSigma),\dots,\lambda_K(\bfSigma),1,\dots,1)$ and $H$ is the pointmass at 1, it can be verified by direct calculation that $\psi'(\lambda_j(\bfSigma))>0$ is equivalent to $\lambda_j(\bfSigma) > 1 + \sqrt{y}$. The main point of Proposition~\ref{prop_sub_cond} below is that it gives general conditions under which $\mathsf{H}_{1,n}(K)$ implies that $\psi'(\lambda_j(\bfSigma))>0$ holds for all $1\leq j\leq K$. 
Its proof is given in Section \ref{sec_proof_thm_alternative}.
\begin{proposition}\label{prop_sub_cond}
    Suppose that \ref{ass_mp_regime} and \ref{ass_lsd} hold. Also, suppose that $\lambda_1(\bfSigma)>\dots >\lambda_K(\bfSigma)$ are fixed with respect to $n$, and that $\mathsf{H}_{1,n}(K)$ holds for all large $n$. Then, the condition $\psi'(\lambda_j(\bfSigma))>0$ holds for all $1 \leq j \leq K.$
\end{proposition}
With regard to the effective meaning of supercriticality in terms of the behavior of the leading $K$ sample eigenvalues, it is known from~\citep[][Theorem 2.2 and Remark 2]{Zhang:2022} that if the conditions of this proposition hold, then $\lambda_1(\hat\bfSigma),\dots,\lambda_K(\hat\bfSigma)$ will have three properties. Namely, they will be well separated from the bulk sample eigenvalues, and they will be asymptotically Gaussian with fluctuations of order $n^{-1/2}$.

The following theorem establishes the power consistency of the proposed test, and is proved in Section \ref{sec_proof_thm_alternative}. 
\begin{theorem} \label{thm_alternative}
    Suppose that \ref{ass_mp_regime}-\ref{ass_lsd} hold. Also, suppose that $\lambda_1(\bfSigma)>\dots >\lambda_K(\bfSigma)$ are fixed with respect to $n$, and that $\mathsf{H}_{1,n}(K)$ holds for all large $n$. Then,
   \begin{align*}
      \lim_{n\to\infty} \PR \lb T_n>q_{1-\alpha} \rb = 1.
   \end{align*}
\end{theorem}
We close this section with a result describing the accuracy of the proposed estimate for $K$ given in Section~\ref{sec_method}.

\begin{theorem}\label{thm_est_spike}
If the conditions of Theorem~\ref{thm_alternative} hold, then for any fixed $\alpha \in (0,1)$, the estimate $\hat K_n$ defined by~\eqref{eq_def_hat_K} satisfies the following limits as $n\to\infty$,
\begin{align*}
    \P(\hat K_n<K) &\to 0, \quad \quad \textnormal{and} \quad \quad 
    \P(\hat K_n=K) \to 1-\alpha.
\end{align*}
\end{theorem}
The proof is deferred to Section \ref{sec_est_spike}.

\section{A consistent bootstrap for the subcritical regime} \label{sec_stat_appl}
In this section, we introduce a new bootstrap method for approximating the distributions of functionals of leading sample eigenvalues in the subcritical regime. That is, the statistics of interest $\hat\theta$ have the form
$$\hat\theta=\varphi(\lambda_1(\hat\bfSigma),\dots,\lambda_k(\hat\bfSigma)),$$
where $\varphi:\R^k\to\R^m$ is a generic function, and $k,m\geq 1$ are fixed integers. As has been highlighted elsewhere in the literature~\citep{Karoui:2019}, the non-parametric bootstrap based on sampling with replacement often fails to approximate the distributions of such statistics in the subcritical regime. The bootstrap method proposed here provides a solution to this problem.

At a conceptual level, the method is based on three ingredients. The first ingredient is the test from Section~\ref{sec_method}, which guides the user to determine whether or not a bootstrap for the subcritical regime is appropriate. The second ingredient is a fundamental universality result, which ensures that within the subcritical regime, the sample eigenvalues $\lambda_1(\hat\bfSigma),\dots,\lambda_k(\hat\bfSigma)$ behave asymptotically as if the data were generated from the Gaussian distribution $\mathcal{N}(\mathbf{0},\mathbf{\Lambda})$, where $\bold{\Lambda}=\textup{diag}(\lambda_1(\bfSigma),\dots,\lambda_p(\bfSigma))$~\citep{knowles2017anisotropic}. The third ingredient is the matrix of truncated eigenvalue estimates~$\tilde{\mathbf{\Lambda}}=\textup{diag}(\tilde\lambda_1,\dots,\tilde\lambda_p)$ defined in~\eqref{def_lambda_hat}, which is specially designed to conform with an approximate version of the subcritical condition. Altogether, the method generates bootstrap data from the fitted Gaussian distribution $\mathcal{N}(\mathbf{0},\tilde{\mathbf{\Lambda}})$, as shown in the following algorithm.

\begin{algo}
Bootstrap method for the subcritical regime.
\begin{tabbing}
\qquad \enspace \textbf{Require:} matrix $\tilde{\mathbf{\Lambda}}$ \\
\qquad \enspace For $b=1$ to $B$ \\
\qquad \qquad Generate independent random vectors 
$\mathbf{y}_1^{\star}, \ldots, \mathbf{y}_n^{\star}
\sim \mathcal{N}(\mathbf{0}, \tilde{\mathbf{\Lambda}})$ \\
\qquad \qquad Compute the sample covariance matrix \\
\qquad \qquad\qquad 
$\bfSigmahat^\star
= \frac{1}{n} \sum\limits_{i=1}^n
\mathbf{y}_i^{\star}(\mathbf{y}_i^{\star})\ttop$ \\
\qquad \qquad Compute the statistic \\
\qquad \qquad\qquad
$\hat\theta^{\star}_b
= \varphi(
\lambda_1(\bfSigmahat^\star), \dots,
\lambda_k(\bfSigmahat^\star))$ \\
\qquad \enspace \textbf{Output:} empirical distribution of
$\hat\theta_1^{\star}, \dots, \hat\theta_B^{\star}$
\end{tabbing}
\label{alg_bootstrap}
\end{algo}

\noindent While the outward appearance of the algorithm is straightforward, this appearance conceals two major challenges that the method overcomes. These are determining whether or not the subcritical regime holds, and estimating $\lambda_1(\bfSigma),\dots,\lambda_p(\bfSigma)$ in a way that respects the subcritical regime. So, in essence, the simple form of Algorithm~\ref{alg_bootstrap} is made possible by the insights that are embedded in the design of the test $T_n$ and the eigenvalue estimates $\tilde{\boldsymbol{\Lambda}}$.\\

\noindent \emph{Consistency.} To formulate a consistency guarantee for the proposed bootstrap, let $\tilde\xi_{n,0}\in (0,1/\tilde\lambda_1)$ be the unique value that solves the equation
\begin{align}\label{eqn:tildexidef}
     \frac{1}{p} \sum_{j=1}^p  \bigg(\frac{\tilde\lambda_j \tilde\xi_{n,0} }{  1 - \tilde \lambda_j \tilde\xi_{n,0} } \bigg)^2 = \frac{n}{p} .
\end{align}
Also, define 
\begin{align}
    \tilde r_n  & 
=
    \frac{1}{\tilde\xi_{n,0}} \bigg( 1 + y_n \int  \frac{\lambda \tilde\xi_{n,0}}{1 - \lambda\tilde\xi_{n,0}} d\tilde H_n(\lambda) \bigg),
    \label{def_hat_r_Q} \\ 
      \tilde\sigma_n^3 & = \frac{1}{ \tilde\xi_{n,0}^3} \bigg( 1 + y_n \int \bigg( \frac{\lambda \tilde\xi_{n,0}}{1 - \lambda\tilde\xi_{n,0}}\bigg)^3 d\tilde H_n(\lambda) \bigg).
      \label{def_hat_sigma_Q}
\end{align}
These quantities only serve a formal purpose, which is to normalize $\lambda_1(\hat \bfSigma^{\star}),\dots, \lambda_k(\hat \bfSigma^{\star})$ so that their fluctuations can be theoretically compared with those of $\lambda_1(\hat \bfSigma),\dots, \lambda_k(\hat \bfSigma)$. 
In detail, for any fixed integers $k,m\geq 1$ and continuous function $g:\R^k\to\R^m$, the comparison will be made between the random vectors 
\begin{align*}
V_n &=g\Big(\ts\frac{n^{2/3}}{\sigma_n}(\lambda_1(\hat \bfSigma)-r_n),\dots,\ts\frac{n^{2/3}}{\sigma_n}(\lambda_k(\hat\bfSigma)-r_n)\Big)\\
    V_{n}^{\star} &= g\Big(\ts\frac{n^{2/3}}{\tilde\sigma_n}(\lambda_1(\hat \bfSigma^{\star})-\tilde r_n),\dots,\ts\frac{n^{2/3}}{\tilde \sigma_n}(\lambda_k(\hat\bfSigma^{\star})-\tilde r_n)\Big).
\end{align*}
\rev{Recall from Section~\ref{sec_prelim} that the centering and scale parameters $r_n$ and $\sigma_n$ ensure that as $n\to\infty$ the random vector $\frac{n^{2/3}}{\sigma_n}(\lambda_j(\hat\bfSigma)-r_n)_{1\leq j\leq k}$ converges weakly to a $k$-dimensional Tracy-Widom distribution of type 1 under \ref{ass_mp_regime}-\ref{ass_lsd} and $\mathsf{H}_{0,n}$. See also Remark \ref{rem:1} in Section~\ref{sec_proof_thm_conv_T_null} for further details. The quantities $\tilde r_n$ and $\tilde\sigma_n$ play analogous roles so that the conditional distribution of the random vector $\frac{n^{2/3}}{\tilde \sigma_n}(\lambda_j(\hat\bfSigma^{\star})-\tilde r_n)_{1\leq j\leq k}$,  given the data $\mathbf{Y}$, converges weakly in probability to the same limit.}
As a shorthand, we refer to the conditional distribution of $V_n^{\star}$ given the data $\mathbf{Y}$ as $\mathcal{L} \lb  V_{n}^\star | \mathbf{Y} \rb$. Lastly, let $d$ denote any metric on the space of probability distributions on $\R^m$ whose topology coincides with that of weak convergence.  For example, it suffices to take $d$ to be the L\'evy-Prohorov metric.
With this notation in hand, we are now in position to state the following bootstrap consistency result.

\begin{theorem}
    \label{thm_bootstrap_cont_map}
      If $\mathsf{H}_{0,n}$ holds for all large $n$ and \ref{ass_mp_regime}-\ref{ass_lsd} are satisfied, then as $n\to\infty$
    \begin{align*}
     d\!\lb \mathcal{L} \lb  V_{n}^\star | \mathbf{Y} \rb , \mathcal{L} \lb V_n \rb  \rb 
        \xrightarrow{\PR} 0.
    \end{align*}
\end{theorem}

From a practical standpoint, it should be emphasized that the bootstrapped eigenvalues $\lambda_1(\hat\bfSigma^{\star}),\dots,\lambda_k(\hat\bfSigma^{\star})$ do not need to be normalized as above in order to apply Algorithm~\ref{alg_bootstrap} to solve various inference tasks. For instance, consider the problem of estimating the bias $\E ( \lambda_1(\bfSigmahat)) -  \lambda_1(\bfSigma)$ of the largest sample eigenvalue. In this case, the bootstrap estimate is defined as $\E(\lambda_1(\bfSigmahat^\star) | \mathbf{Y}) - \tilde\lambda_{1}$, where the conditional expectation is numerically approximated by averaging many realizations of $\lambda_1(\hat\bfSigma^{\star})$ in Algorithm~\ref{alg_bootstrap}. In the experiments presented in Section \ref{sec_sim_bootstrap}, we show that the bootstrap estimates the bias well across a range of conditions.

\section{Numerical experiments}\label{sec_expt}
In this section, we use simulated data to compare the proposed statistic $T_n$ with Onatski's statistic $R_n(\kappa)$ in testing $\mathsf{H}_{0,n}$. This comparison shows that both statistics have a well-controlled level, while $T_n$ often delivers higher power. Throughout all of the experiments, we put $\varepsilon=0.2$ in the definition of $\mathsf{H}_{0,n}$ and the alternatives (except in Section~\ref{sec_sim_epsilon}, where many values for $\varepsilon$ are considered).

\rev{Several additional sets of numerical results are presented in Section~\ref{sec_supp_sim}. First, in Section~\ref{sec_sim_bootstrap}, we look at the empirical performance of the proposed bootstrap, showing that it can accurately approximate the distributions of normalized versions of $\lambda_1(\hat\bfSigma)$ and $\lambda_1(\hat\bfSigma)-\lambda_2(\hat\bfSigma)$, and also accurately estimate the bias $\E(\lambda_1(\hat\bfSigma))-\lambda_1(\bfSigma)$. Next, in Section~\ref{sec_stock}, we apply $T_n$ and $R_n(\kappa)$ to test $\mathsf{H}_{0,n}$ using two high-dimensional datasets derived from the stock market returns of companies in the S\&P 500. In these real-data examples, $T_n$ produces smaller p-values, which aligns with the power advantage observed in the simulations. Lastly, in Sections~\ref{sec_epsilon_hat}-\ref{sec_robustness} we address the finite-sample behavior of $\hat\varepsilon$, sensitivity analysis with respect to the choice of $\varepsilon$, and robustness checks.}
The code for running the experiments is available at {\tt{https://github.com/ndoernemann/testing-TW}}.

\subsection{Testing the subcritical condition}\label{sec:expt:test} 
To generate data under the null and alternative, we used data matrices $\mathbf{Y}$ of the form given in \ref{ass_mom}  with  $(n,p)=(600,400)$ and varying choices of the leading eigenvalues of $\bfSigma$.\\

\noindent\emph{Experiments with varying $\lambda_1(\bfSigma)$.} First, we describe a subset of simulations in which only $\lambda_1(\bfSigma)$ was varied over a grid of values in the interval $[1,3.5]$, while keeping $\lambda_2(\bfSigma),\dots,\lambda_p(\bfSigma)$ fixed. This was done using $\bfSigma=\textup{diag}(\lambda_1(\bfSigma),\dots,\lambda_p(\bfSigma))$ and choosing the eigenvalues $\lambda_2(\bfSigma),\dots,\lambda_p(\bfSigma)$ to have two types of structures, which we refer to as the `spiked spectrum' and `decaying spectrum'. In the case of the spiked spectrum, we put $\lambda_j(\bfSigma)=1$ for all $j=2,\dots,p$, while in the case of the decaying spectrum, we put $\lambda_j(\bfSigma)=1$ for $j\in\{2,\dots,151\}$ and $\lambda_j(\bfSigma)=j^{-c}$ for $j\in\{152,\dots,p\}$ with $c\in\{0.5,1\}$. For each setting of $\bfSigma$, we generated 800 realizations of the data matrix $\mathbf Y=\mathbf{X}\bfSigma^{1/2}$ with $\mathbf{X}\in\R^{n\times p}$ having i.i.d.~entries, and taking the random variable $x_{11}$ to follow either a $\mathcal{N}(0,1)$ distribution or a standardized $t_{10}$ distribution. Next, for each realization of $\mathbf{Y}$, we computed test statistics $T_n$ and $R_n(\kappa)$ with $\kappa\in\{1,10\}$, and recorded whether or not $\mathsf{H}_{0,n}$ was rejected at a nominal level of $\alpha=0.05$. We refer to the proportion of rejections occurring among these 800 trials as the rejection probability. Note also that the critical value for $R_n(\kappa)$ is given by the $(1-\alpha)$-quantile of the random variable $\max_{1\leq j\leq \kappa}\{(\zeta_j-\zeta_{j+1})/(\zeta_{j+1}-\zeta_{j+2})\}$, where $(\zeta_1,\dots,\zeta_{\kappa+2})$ follows a $(\kappa+2)$-dimensional Tracy-Widom distribution of type 1.

\begin{figure} 
  \centering
  \begin{minipage}{.47\textwidth}
    \centering
    \begin{overpic}[width=\textwidth]{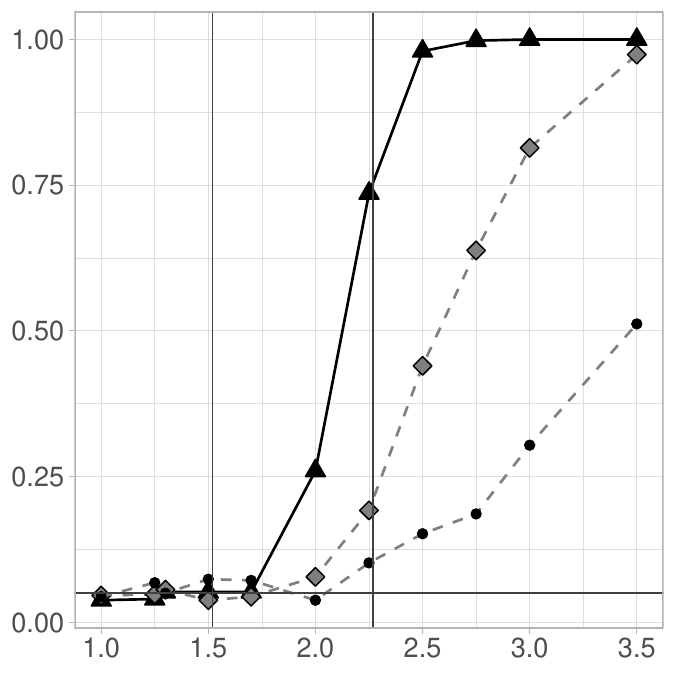}
\put(-6,35){\rotatebox{90}{rejection probability}}
\put(6,11.5){$\alpha$}
\put(9,-12){  $\leftarrow\!\!\mathsf{H}_{0,n}\!\!\rightarrow$  }
\put(53,-12){ $\xleftarrow{  \ \ \ \ \ \ }\mathsf{H}_{1,n}(1)\xrightarrow{ \ \ \ \ }$  }
\put(35,-4){ value of $\lambda_1(\bfSigma)$  }
	\end{overpic}
    ~\\
  \end{minipage}
  \hfill
  \begin{minipage}{.47\textwidth}
    \centering
    \begin{overpic}[width=\textwidth]{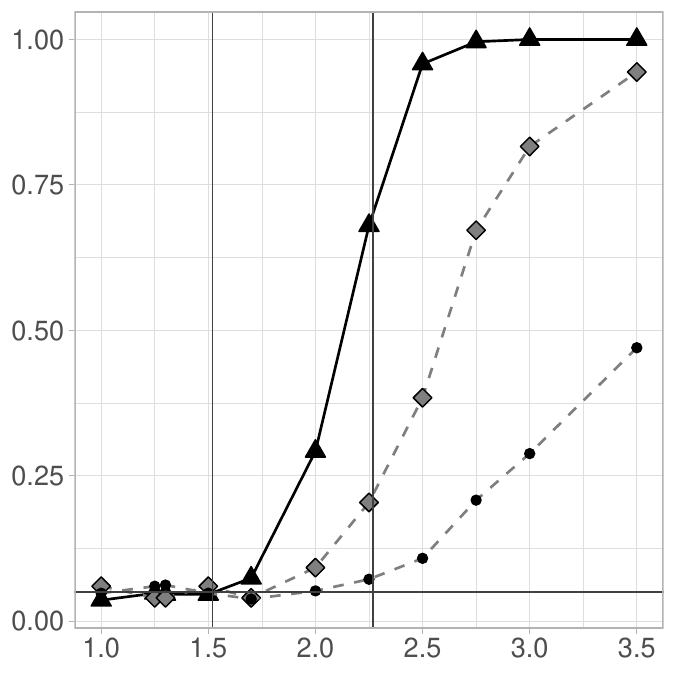}
\put(-6,35){\rotatebox{90}{rejection probability}}
\put(6,11.5){$\alpha$}
\put(9,-12){  $\leftarrow\!\!\mathsf{H}_{0,n}\!\!\rightarrow$  }
\put(53,-12){ $\xleftarrow{   \ \ \ \ \ \ }\mathsf{H}_{1,n}(1)\xrightarrow{ \ \ \ \ }$  }
\put(35,-4){ value of $\lambda_1(\bfSigma)$  }
    \end{overpic}
    ~\\
  \end{minipage} 
  \\[6ex]
\begin{minipage}{.47\textwidth}
    \centering
 \begin{overpic}[width=\textwidth]{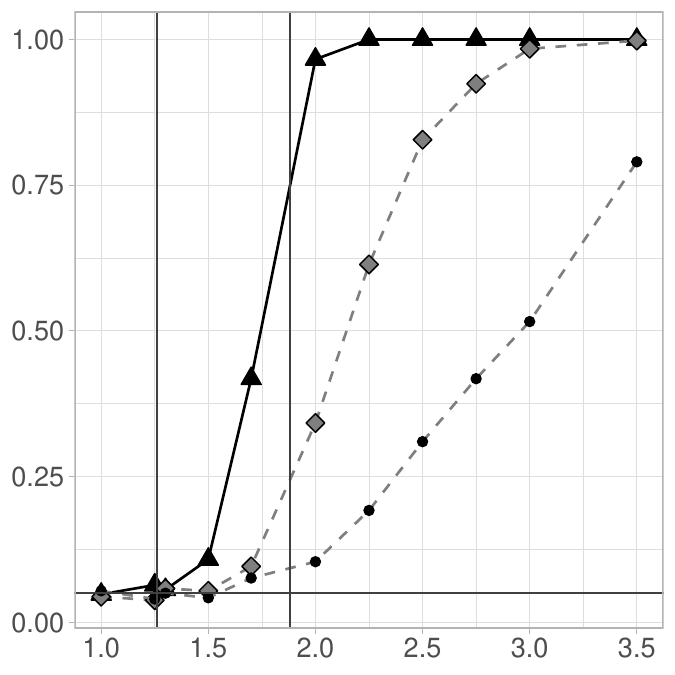}
 \put(-6,35){\rotatebox{90}{rejection probability}}
\put(6,11.5){$\alpha$}
\put(7,-12){  $\footnotesize\shortleftarrow\!\!\displaystyle\mathsf{H}_{0,n}\!\!\!\!\footnotesize\shortrightarrow$  }
\put(43,-12){ $\xleftarrow{ \ \  \ \ \ \ \ \  }\mathsf{H}_{1,n}(1)\xrightarrow{ \ \ \ \ \ \  \ \ \  }$  }
\put(35,-4){ value of $\lambda_1(\bfSigma)$  }
 \end{overpic}
\end{minipage}  
\hfill
 \begin{minipage}{.47\textwidth}
    \centering             \begin{overpic}[width=\textwidth]{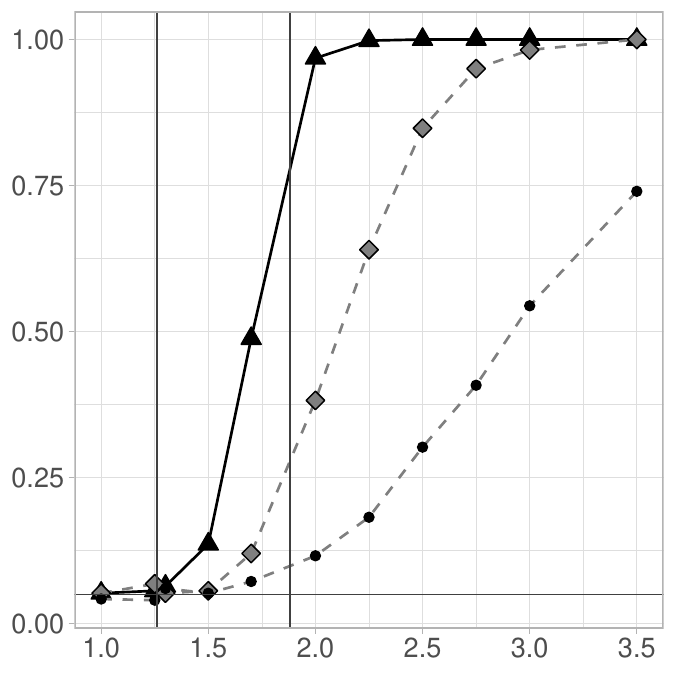}
   \put(-6,35){\rotatebox{90}{rejection probability}}
\put(6,11.5){$\alpha$}
\put(7,-12){  $\shortleftarrow\!\!\mathsf{H}_{0,n}\!\!\!\shortrightarrow$  }
\put(43,-12){ $\xleftarrow{ \ \  \ \ \ \ \ \  }\mathsf{H}_{1,n}(1)\xrightarrow{ \ \ \ \ \ \  \ \ \  }$  }
\put(35,-4){ value of $\lambda_1(\bfSigma)$  }
    \end{overpic}
    \end{minipage}  
    \\ [6ex]
  \caption{ Rejection probabilities for the tests  $T_n$ (triangle), $R_n(1)$ (diamond) and $R_n(10)$ (circle), plotted as a function of $\lambda_1(\bfSigma)$. The arrows below the x-axis specify the values of $\lambda_1(\bfSigma)$ corresponding to $\mathsf{H}_{0,n}$ and $\mathsf{H}_{1,n}(1)$. The nominal level of $\alpha=0.05$ is marked with a horizontal line in each panel. First row: Spiked spectrum with  $x_{11} \sim \mathcal{N}(0,1)$ (left panel) and $x_{11} \sim t_{10}/\sqrt{\var(t_{10})}$ (right panel). Second row: Decaying spectrum with $x_{11}\sim\mathcal{N}(0,1)$, and decay parameter $c=1$ (left panel) and $c=0.5$ (right panel).  }
      \label{fig1}
\end{figure}

The results are displayed in Figure~\ref{fig1} by plotting the rejection probabilities of the tests as a function of $\lambda_1(\bfSigma)$. The four panels correspond to different model configurations that are indicated in the caption. As $\lambda_1(\bfSigma)$ increases from $1$ to 3.5 in each panel, there is a transition from $\mathsf{H}_{0,n}$ to $\mathsf{H}_{1,n}(1)$, and the range of values of $\lambda_1(\bfSigma)$ corresponding to each hypothesis are marked below the x-axis. Under $\mathsf{H}_{0,n}$, all three statistics produce rejection probabilities that agree closely with the nominal level of $\alpha=0.05$. On the other hand, when $\mathsf{H}_{1,n}(1)$ holds, we see that $T_n$ has higher power than $R_n(1)$ and $R_n(10)$.

There are a few more aspects of the results in Figure~\ref{fig1} to discuss. First, the ideal choice of $\kappa$ for detecting $\mathsf{H}_{1,n}(1)$ with the statistic $R_n(\kappa)$ is $\kappa=1$, and so it is notable that even when $R_n(\kappa)$ is favorably tuned, the statistic $T_n$ still provides higher power. Second, the loss of power suffered by $R_n(10)$ in comparison to $R_n(1)$ illustrates the practical drawback of using a conservatively large value of $\kappa$ that was explained in the introduction (page~\pageref{introdisc}). Third, the decaying spectrum violates the theoretical condition~\ref{ass_lsd}, and the standardized $t_{10}$ distribution has tails that are too heavy to satisfy the moment assumptions in~\ref{ass_mom}. Thus, the results in Figure~\ref{fig1} show that the statistic $T_n$ has a scope of application that extends beyond those conditions.\\

\noindent\emph{Experiments with varying $\lambda_1(\bfSigma),\lambda_2(\bfSigma),\lambda_3(\bfSigma)$.} As an extension of the previous simulations, we compared the statistics $T_n$, $R_n(1)$ and $R_n(10)$ under varying choices of the leading three  population eigenvalues when $\bfSigma=\textup{diag}(\lambda_1(\bfSigma),\lambda_2(\bfSigma),\lambda_3(\bfSigma),1,\dots,1)$ and $x_{11}\sim \mathcal{N}(0,1)$. The choices of the leading eigenvalues correspond to the hypotheses $\mathsf{H}_{0,n}$, $\mathsf{H}_{1,n}(1)$ and $\mathsf{H}_{1,n}(2)$, and are listed in Table~\ref{table_three_spikes}. Overall, the relative performance of the statistics is similar to the previous settings---insofar as their rejection probabilities are close to the nominal level of $\alpha=0.05$ under the null, and $T_n$ achieves greater power under the alternatives. 
A distinct phenomenon shown in the boxed entries of Table~\ref{table_three_spikes} is that $R_n(1)$ is unable to achieve power under $\mathsf{H}_{1,n}(2)$, as anticipated by the discussion in the introduction (page~\pageref{introdisc}). In particular, this illustrates the sensitivity of the statistic $R_n(\kappa)$ to the choice of $\kappa$.

    \begin{table}[h]
    \centering 
\begin{tabular}{lclccc}
& & & \multicolumn{3}{c}{\underline{ \ \ \ rejection probabilities \ \ \ }}\\[0.1cm]
hypothesis & $K$ & $(\lambda_1(\bfSigma),\lambda_2(\bfSigma),\lambda_3(\bfSigma))$   &  $T_n$    & $R_n(1)$    & $R_n(10)$   \\ \hline 
$\mathsf{H}_{0,n}$ & 0 & $(1,1,1)$     
& 0.05 & 0.05 & 0.04 \\
$\mathsf{H}_{0,n}$ & 0 & $(1.25,1.25,1.25)$ 
& 0.04 & 0.04 & 0.07 \\
$\mathsf{H}_{0,n}$ & 0 & $(1.25,1.25,1)$ 
& 0.06 & 0.07 & 0.06 \\
$\mathsf{H}_{0,n}$ & 0 & $(1.3,1.3, 1.3)$   
& 0.05 & 0.06 & 0.06 \\
$\mathsf{H}_{0,n}$ & 0 & $(1.3,1.3,1)$ 
& 0.05 & 0.06 & 0.06 \\
$\mathsf{H}_{0,n}$ & 0 & $(1.5,1.5,1.3)$ 
& 0.05 & 0.04 & 0.05 \\
$\mathsf{H}_{0,n}$ & 0 & $(1.5,1.5, 1)$ 
& 0.06 & 0.05 & 0.05 \\
$\mathsf{H}_{1,n}(K)$ & 1 & $(2.5,1.5,1.3)$  
& 0.96 & 0.42 & 0.12 \\
$\mathsf{H}_{1,n}(K)$ & 1 & $(3, 1.5, 1.3)$  
& 1.00 & 0.79 & 0.28 \\
$\mathsf{H}_{1,n}(K)$ & 1 & $(3.5,1.5, 1.3)$  
& 1.00 & 0.97 & 0.46 \\
$\mathsf{H}_{1,n}(K)$ & 1 & $(4, 1.5, 1.3)$  
& 1.00 & 1.00 & 0.66 \\
$\mathsf{H}_{1,n}(K)$ & 1 & $(4.5, 1.5, 1.3)$ 
& 1.00 & 1.00 & 0.84 \\
$\mathsf{H}_{1,n}(K)$ & 1 & $(5, 1.5, 1.3)$  
& 1.00 & 1.00 & 0.93 \\
$\mathsf{H}_{1,n}(K)$ & 2 & $(4,3,1)$  
& 1.00 & \boxed{0.00} & 0.28 \\
$\mathsf{H}_{1,n}(K)$ & 2 & $(5,3,1)$ 
& 1.00 & \boxed{0.00} & 0.27 
\end{tabular}
\caption{
Rejection probabilities for the tests $T_n$, $R_n(1)$ and $R_n(10)$ under the null $\mathsf{H}_{0,n}$ and the alternatives $\mathsf{H}_{1,n}(K)$, $K\in\{1,2\}$ when $\bfSigma=\textup{diag}(\lambda_1(\bfSigma),\lambda_2(\bfSigma),\lambda_3(\bfSigma),1,\dots,1)$ and $x_{11}\sim \mathcal{N}(0,1)$. Recall that $K$ denotes the number of supercritical population eigenvalues, and the nominal level is $\alpha=0.05$. The boxed entries illustrate cases where $R_n(\kappa)$ has no power when $\kappa<K$.
}
\label{table_three_spikes}
\end{table}

\section*{Acknowledgement}
N. Dörnemann gratefully acknowledges the support by the DFG Research unit 5381 {\it Mathematical Statistics in the Information Age}, project number 460867398, and the Aarhus University Research Foundation (AUFF), project numbers 47221 and 47388. \rev{We thank the Associate Editor and the referees for their careful reading and constructive comments, which have substantially improved the paper.}

\section*{Supplementary material}
\label{SM}

\rev{ 
Sections \ref{sec_proofs},~\ref{sec_proof_thm_alternative}, and~\ref{sec_proof_boot} contain proofs of the theoretical results in Sections~\ref{sec_null},~\ref{sec_alternative}, and~\ref{sec_stat_appl} respectively.
Section \ref{sec_supp_sim} contains additional numerical results. 
}

\bibliographystyle{biometrika}
\bibliography{references, NSF2023bib}

\begin{thebibliography}{69}
\expandafter\ifx\csname natexlab\endcsname\relax\def\natexlab#1{#1}\fi

\bibitem[{Bai \& Ng(2002)}]{BaiNg2002}
\textsc{Bai, J.} \& \textsc{Ng, S.} (2002).
\newblock Determining the number of factors in approximate factor models.
\newblock \textit{Econometrica} \textbf{70}, 191--221.

\bibitem[{Bai \& Ng(2008)}]{bai2008large}
\textsc{Bai, J.} \& \textsc{Ng, S.} (2008).
\newblock Large dimensional factor analysis.
\newblock \textit{Foundations and Trends in Econometrics} \textbf{3}, 89--163.

\bibitem[{Bai \& Silverstein(2010)}]{bai2004}
\textsc{Bai, Z.} \& \textsc{Silverstein, J.~W.} (2010).
\newblock \textit{Spectral Analysis of Large Dimensional Random Matrices}.
\newblock Springer.

\bibitem[{Bai \& Yao(2012)}]{BAI2012167}
\textsc{Bai, Z.} \& \textsc{Yao, J.} (2012).
\newblock On sample eigenvalues in a generalized spiked population model.
\newblock \textit{Journal of Multivariate Analysis} \textbf{106}, 167--177.

\bibitem[{Bai \& Yao(2008)}]{bai2008central}
\textsc{Bai, Z.} \& \textsc{Yao, J.-F.} (2008).
\newblock Central limit theorems for eigenvalues in a spiked population model.
\newblock \textit{Annales de l'IHP Probabilit{\'e}s et Statistiques} \textbf{44}, 447--474.

\bibitem[{Baik et~al.(2005)Baik, Ben~Arous \& P{\'e}ch{\'e}}]{bbp}
\textsc{Baik, J.}, \textsc{Ben~Arous, G.} \& \textsc{P{\'e}ch{\'e}, S.} (2005).
\newblock {Phase transition of the largest eigenvalue for nonnull complex sample covariance matrices}.
\newblock \textit{The Annals of Probability} \textbf{33}, 1643 -- 1697.

\bibitem[{Baik \& Silverstein(2006)}]{baik2006eigenvalues}
\textsc{Baik, J.} \& \textsc{Silverstein, J.~W.} (2006).
\newblock Eigenvalues of large sample covariance matrices of spiked population models.
\newblock \textit{Journal of Multivariate Analysis} \textbf{97}, 1382--1408.

\bibitem[{Bao et~al.(2022)Bao, Ding, Wang \& Wang}]{Bao:2022}
\textsc{Bao, Z.}, \textsc{Ding, X.}, \textsc{Wang, J.} \& \textsc{Wang, K.} (2022).
\newblock Statistical inference for principal components of spiked covariance matrices.
\newblock \textit{The Annals of Statistics} \textbf{50}, 1144--1169.

\bibitem[{Bao et~al.(2013)Bao, Pan \& Zhou}]{bao2013local}
\textsc{Bao, Z.}, \textsc{Pan, G.} \& \textsc{Zhou, W.} (2013).
\newblock Local density of the spectrum on the edge for sample covariance matrices with general population.
\newblock \textit{Preprint. {\tt{https://personal.ntu.edu.sg/gmpan/publications.html}}} .

\bibitem[{Bao et~al.(2015)Bao, Pan \& Zhou}]{bao2015}
\textsc{Bao, Z.}, \textsc{Pan, G.} \& \textsc{Zhou, W.} (2015).
\newblock {Universality for the largest eigenvalue of sample covariance matrices with general population}.
\newblock \textit{The Annals of Statistics} \textbf{43}, 382 -- 421.

\bibitem[{Bauer(2015)}]{Bauer:2015}
\textsc{Bauer, M.~D.} (2015).
\newblock Nominal interest rates and the news.
\newblock \textit{Journal of Money, Credit and Banking} \textbf{47}, 295--332.

\bibitem[{Beran \& Srivastava(1985)}]{BeranS85}
\textsc{Beran, R.} \& \textsc{Srivastava, M.~S.} (1985).
\newblock Bootstrap tests and confidence regions for functions of a covariance matrix.
\newblock \textit{Annals of Statistics} \textbf{13}, 95--115.

\bibitem[{Bloemendal et~al.(2016)Bloemendal, Knowles, Yau \& Yin}]{BloemendalKnowlesYauYin2016}
\textsc{Bloemendal, A.}, \textsc{Knowles, A.}, \textsc{Yau, H.-T.} \& \textsc{Yin, J.} (2016).
\newblock On the principal components of sample covariance matrices.
\newblock \textit{Probability Theory and Related Fields} \textbf{164}, 459--552.

\bibitem[{Bykhovskaya et~al.(2026)Bykhovskaya, Gorin \& Sodin}]{bykhovskaya2025weak}
\textsc{Bykhovskaya, A.}, \textsc{Gorin, V.} \& \textsc{Sodin, S.} (2026).
\newblock How weak are weak factors? {U}niform inference for signal strength in signal plus noise models.
\newblock \textit{arXiv:2507.18554} .

\bibitem[{Cai et~al.(2020)Cai, Han \& Pan}]{cai_han_pan}
\textsc{Cai, T.~T.}, \textsc{Han, X.} \& \textsc{Pan, G.} (2020).
\newblock {Limiting laws for divergent spiked eigenvalues and largest nonspiked eigenvalue of sample covariance matrices}.
\newblock \textit{The Annals of Statistics} \textbf{48}, 1255 -- 1280.

\bibitem[{Dette et~al.(2020)Dette, Kokot \& Volgushev}]{Dette:2020}
\textsc{Dette, H.}, \textsc{Kokot, K.} \& \textsc{Volgushev, S.} (2020).
\newblock Testing relevant hypotheses in functional time series via self-normalization.
\newblock \textit{Journal of the Royal Statistical Society Series B: Statistical Methodology} \textbf{82}, 629--660.

\bibitem[{Dette \& Wu(2019)}]{Dette:2019}
\textsc{Dette, H.} \& \textsc{Wu, W.} (2019).
\newblock Detecting relevant changes in the mean of nonstationary processes—a mass excess approach.
\newblock \textit{Annals of Statistics} \textbf{47}, 3578--3608.

\bibitem[{Diaconis \& Efron(1983)}]{Diaconis:1983}
\textsc{Diaconis, P.} \& \textsc{Efron, B.} (1983).
\newblock Computer-intensive methods in statistics.
\newblock \textit{Scientific American} \textbf{248}, 116--131.

\bibitem[{Ding(2021)}]{ding2021spiked}
\textsc{Ding, X.} (2021).
\newblock Spiked sample covariance matrices with possibly multiple bulk components.
\newblock \textit{Random Matrices: Theory and Applications} \textbf{10}, 2150014.

\bibitem[{Ding et~al.(2023)Ding, Xie, Yu \& Zhou}]{Ding:2023}
\textsc{Ding, X.}, \textsc{Xie, J.}, \textsc{Yu, L.} \& \textsc{Zhou, W.} (2023).
\newblock Extreme eigenvalues of sample covariance matrices under generalized elliptical models with applications.
\newblock \textit{arXiv:2303.03532} .

\bibitem[{Ding \& Yang(2018)}]{ding2018necessary}
\textsc{Ding, X.} \& \textsc{Yang, F.} (2018).
\newblock A necessary and sufficient condition for edge universality at the largest singular values of covariance matrices.
\newblock \textit{The Annals of Applied Probability} \textbf{28}, 1679--1738.

\bibitem[{Ding \& Yang(2021)}]{ding:yang:2021}
\textsc{Ding, X.} \& \textsc{Yang, F.} (2021).
\newblock {Spiked separable covariance matrices and principal components}.
\newblock \textit{The Annals of Statistics} \textbf{49}, 1113 -- 1138.

\bibitem[{Ding \& Yang(2022)}]{ding2022tracy}
\textsc{Ding, X.} \& \textsc{Yang, F.} (2022).
\newblock Tracy-{W}idom distribution for heterogeneous {G}ram matrices with applications in signal detection.
\newblock \textit{IEEE Transactions on Information Theory} \textbf{68}, 6682--6715.

\bibitem[{Donoho et~al.(2018)Donoho, Gavish \& Johnstone}]{donoho2018}
\textsc{Donoho, D.~L.}, \textsc{Gavish, M.} \& \textsc{Johnstone, I.~M.} (2018).
\newblock Optimal shrinkage of eigenvalues in the spiked covariance model.
\newblock \textit{Annals of Statistics} \textbf{46}, 1742--1778.

\bibitem[{D{\"u}mbgen(1993)}]{Dumbgen:1993}
\textsc{D{\"u}mbgen, L.} (1993).
\newblock On nondifferentiable functions and the bootstrap.
\newblock \textit{Probability Theory and Related Fields} \textbf{95}, 125--140.

\bibitem[{Eaton \& Tyler(1991)}]{EatonT91}
\textsc{Eaton, M.~L.} \& \textsc{Tyler, D.~E.} (1991).
\newblock On {W}ielandt's inequality and its application to the asymptotic distribution of the eigenvalues of a random symmetric matrix.
\newblock \textit{Annals of Statistics} \textbf{19}, 260--271.

\bibitem[{El~Karoui(2007)}]{elkaroui2007}
\textsc{El~Karoui, N.} (2007).
\newblock {Tracy–Widom limit for the largest eigenvalue of a large class of complex sample covariance matrices}.
\newblock \textit{The Annals of Probability} \textbf{35}, 663 -- 714.

\bibitem[{El~Karoui \& Purdom(2019)}]{Karoui:2019}
\textsc{El~Karoui, N.} \& \textsc{Purdom, E.} (2019).
\newblock The non-parametric bootstrap and spectral analysis in moderate and high-dimension.
\newblock In \textit{AISTATS 2019}.

\bibitem[{Fabozzi et~al.(2007)Fabozzi, Kolm, Pachamanova \& Focardi}]{fabozzi2007robust}
\textsc{Fabozzi, F.~J.}, \textsc{Kolm, P.~N.}, \textsc{Pachamanova, D.~A.} \& \textsc{Focardi, S.~M.} (2007).
\newblock \textit{Robust Portfolio Optimization and Management}.
\newblock John Wiley \& Sons.

\bibitem[{Feng et~al.(2000)Feng, Yuen \& Dai}]{feng2000human}
\textsc{Feng, G.-C.}, \textsc{Yuen, P.~C.} \& \textsc{Dai, D.-Q.} (2000).
\newblock Human face recognition using {PCA} on wavelet subband.
\newblock \textit{Journal of Electronic Imaging} \textbf{9}, 226--233.

\bibitem[{Gin{\'e} \& Nickl(2021)}]{Gine:2021}
\textsc{Gin{\'e}, E.} \& \textsc{Nickl, R.} (2021).
\newblock \textit{Mathematical Foundations of Infinite-Dimensional Statistical Models}.
\newblock Cambridge University Press.

\bibitem[{Goldreich(2017)}]{Goldreich:2017}
\textsc{Goldreich, O.} (2017).
\newblock \textit{Introduction to Property Testing}.
\newblock Cambridge.

\bibitem[{Hall et~al.(2009)Hall, Lee, Park \& Paul}]{Hall:2009}
\textsc{Hall, P.}, \textsc{Lee, Y.~K.}, \textsc{Park, B.~U.} \& \textsc{Paul, D.} (2009).
\newblock Tie-respecting bootstrap methods for estimating distributions of sets and functions of eigenvalues.
\newblock \textit{Bernoulli} \textbf{15}, 380--401.

\bibitem[{Han et~al.(2018)Han, Xu \& Zhou}]{Han:2018}
\textsc{Han, F.}, \textsc{Xu, S.} \& \textsc{Zhou, W.-X.} (2018).
\newblock {On Gaussian comparison inequality and its application to spectral analysis of large random matrices}.
\newblock \textit{Bernoulli} \textbf{24}, 1787 -- 1833.

\bibitem[{Heckman et~al.(2013)Heckman, Pinto \& Savelyev}]{Heckman:2013}
\textsc{Heckman, J.}, \textsc{Pinto, R.} \& \textsc{Savelyev, P.} (2013).
\newblock Understanding the mechanisms through which an influential early childhood program boosted adult outcomes.
\newblock \textit{American Economic Review} \textbf{103}, 2052--2086.

\bibitem[{Ingster \& Suslina(2003)}]{Ingster:2003}
\textsc{Ingster, Y.} \& \textsc{Suslina, I.~A.} (2003).
\newblock \textit{Nonparametric Goodness-of-Fit Testing under Gaussian Models}.
\newblock Springer.

\bibitem[{Jiang \& Bai(2021)}]{jiang_bai}
\textsc{Jiang, D.} \& \textsc{Bai, Z.} (2021).
\newblock {Generalized four moment theorem and an application to CLT for spiked eigenvalues of high-dimensional covariance matrices}.
\newblock \textit{Bernoulli} \textbf{27}, 274 -- 294.

\bibitem[{Johnstone(2001)}]{johnstone2001}
\textsc{Johnstone, I.~M.} (2001).
\newblock On the distribution of the largest eigenvalue in principal components analysis.
\newblock \textit{The Annals of Statistics} \textbf{29}, 295--327.

\bibitem[{Johnstone \& Onatski(2020)}]{johnstone_onatski_2020}
\textsc{Johnstone, I.~M.} \& \textsc{Onatski, A.} (2020).
\newblock {Testing in high-dimensional spiked models}.
\newblock \textit{The Annals of Statistics} \textbf{48}, 1231 -- 1254.

\bibitem[{Johnstone \& Paul(2018)}]{Johnstone:2018}
\textsc{Johnstone, I.~M.} \& \textsc{Paul, D.} (2018).
\newblock {PCA} in high dimensions: An orientation.
\newblock \textit{Proceedings of the IEEE} \textbf{106}, 1277--1292.

\bibitem[{Kabundi \& De~Simone(2020)}]{Kabundi:2020}
\textsc{Kabundi, A.} \& \textsc{De~Simone, F.~N.} (2020).
\newblock Monetary policy and systemic risk-taking in the euro area banking sector.
\newblock \textit{Economic Modelling} \textbf{91}, 736--758.

\bibitem[{Ke et~al.(2023)Ke, Ma \& Lin}]{ke2023estimation}
\textsc{Ke, Z.~T.}, \textsc{Ma, Y.} \& \textsc{Lin, X.} (2023).
\newblock Estimation of the number of spiked eigenvalues in a covariance matrix by bulk eigenvalue matching analysis.
\newblock \textit{Journal of the American Statistical Association} \textbf{118}, 374--392.

\bibitem[{Knowles \& Yin(2017)}]{knowles2017anisotropic}
\textsc{Knowles, A.} \& \textsc{Yin, J.} (2017).
\newblock Anisotropic local laws for random matrices.
\newblock \textit{Probability Theory and Related Fields} \textbf{169}, 257--352.

\bibitem[{Laloux et~al.(2000)Laloux, Cizeau, Potters \& Bouchaud}]{laloux2000random}
\textsc{Laloux, L.}, \textsc{Cizeau, P.}, \textsc{Potters, M.} \& \textsc{Bouchaud, J.-P.} (2000).
\newblock Random matrix theory and financial correlations.
\newblock \textit{International Journal of Theoretical and Applied Finance} \textbf{3}, 391--397.

\bibitem[{Ledoit \& Wolf(2015)}]{ledoit2015spectrum}
\textsc{Ledoit, O.} \& \textsc{Wolf, M.} (2015).
\newblock Spectrum estimation: A unified framework for covariance matrix estimation and {PCA} in large dimensions.
\newblock \textit{Journal of Multivariate Analysis} \textbf{139}, 360--384.

\bibitem[{Ledoit \& Wolf(2017)}]{ledoit2017numerical}
\textsc{Ledoit, O.} \& \textsc{Wolf, M.} (2017).
\newblock Numerical implementation of the {QuEST} function.
\newblock \textit{Computational Statistics \& Data Analysis} \textbf{115}, 199--223.

\bibitem[{Lee \& Schnelli(2016)}]{leeschnelli2016}
\textsc{Lee, J.~O.} \& \textsc{Schnelli, K.} (2016).
\newblock {Tracy–Widom distribution for the largest eigenvalue of real sample covariance matrices with general population}.
\newblock \textit{The Annals of Applied Probability} \textbf{26}, 3786 -- 3839.

\bibitem[{Li et~al.(2020)Li, Han \& Yao}]{li2020asymptotic}
\textsc{Li, Z.}, \textsc{Han, F.} \& \textsc{Yao, J.} (2020).
\newblock Asymptotic joint distribution of extreme eigenvalues and trace of large sample covariance matrix in a generalized spiked population model.
\newblock \textit{The Annals of Statistics} \textbf{48}, 3138--3160.

\bibitem[{Lopes et~al.(2019)Lopes, Blandino \& Aue}]{Lopes:2019:Biometrika}
\textsc{Lopes, M.~E.}, \textsc{Blandino, A.} \& \textsc{Aue, A.} (2019).
\newblock Bootstrapping spectral statistics in high dimensions.
\newblock \textit{Biometrika} \textbf{106}, 781--801.

\bibitem[{Lopes et~al.(2023)Lopes, Erichson \& Mahoney}]{Lopes:Erichson:Mahoney:2023}
\textsc{Lopes, M.~E.}, \textsc{Erichson, N.~B.} \& \textsc{Mahoney, M.~W.} (2023).
\newblock Bootstrapping the operator norm in high dimensions: Error estimation for covariance matrices and sketching.
\newblock \textit{Bernoulli} \textbf{29}, 428--450.

\bibitem[{Lorenz(1956)}]{lorenz1956empirical}
\textsc{Lorenz, E.~N.} (1956).
\newblock \textit{Empirical orthogonal functions and statistical weather prediction}, vol.~1.
\newblock Massachusetts Institute of Technology, Department of Meteorology Cambridge.

\bibitem[{Miranda-Agrippino \& Rey(2015)}]{Miranda:2015}
\textsc{Miranda-Agrippino, S.} \& \textsc{Rey, H.} (2015).
\newblock World asset markets and the global financial cycle.
\newblock \textit{National Bureau of Economic Research} .

\bibitem[{Naumov et~al.(2019)Naumov, Spokoiny \& Ulyanov}]{naumov2019bootstrap}
\textsc{Naumov, A.}, \textsc{Spokoiny, V.} \& \textsc{Ulyanov, V.} (2019).
\newblock Bootstrap confidence sets for spectral projectors of sample covariance.
\newblock \textit{Probability Theory and Related Fields} \textbf{174}, 1091--1132.

\bibitem[{Onatski(2009)}]{onatski2009testing}
\textsc{Onatski, A.} (2009).
\newblock Testing hypotheses about the number of factors in large factor models.
\newblock \textit{Econometrica} \textbf{77}, 1447--1479.

\bibitem[{Onatski et~al.(2014)Onatski, Moreira \& Hallin}]{onatski_et_al_2014}
\textsc{Onatski, A.}, \textsc{Moreira, M.~J.} \& \textsc{Hallin, M.} (2014).
\newblock {Signal detection in high dimension: The multispiked case}.
\newblock \textit{The Annals of Statistics} \textbf{42}, 225 -- 254.

\bibitem[{Passemier \& Yao(2012)}]{passemier2012determining}
\textsc{Passemier, D.} \& \textsc{Yao, J.-F.} (2012).
\newblock On determining the number of spikes in a high-dimensional spiked population model.
\newblock \textit{Random Matrices: Theory and Applications} \textbf{1}, 1150002.

\bibitem[{Passemier \& Yao(2014)}]{Passemier:2014}
\textsc{Passemier, D.} \& \textsc{Yao, J.~F.} (2014).
\newblock Estimation of the number of spikes, possibly equal, in the high-dimensional case.
\newblock \textit{Journal of Multivariate Analysis} \textbf{127}, 173--183.

\bibitem[{Paul(2007)}]{paul2007asymptotics}
\textsc{Paul, D.} (2007).
\newblock Asymptotics of sample eigenstructure for a large dimensional spiked covariance model.
\newblock \textit{Statistica Sinica} \textbf{17}, 1617--1642.

\bibitem[{Ron(2010)}]{Ron:2010}
\textsc{Ron, D.} (2010).
\newblock \textit{Algorithmic and Analysis Techniques in Property Testing}.
\newblock Now Publishers.

\bibitem[{Ruppert \& Matteson(2011)}]{ruppert2011statistics}
\textsc{Ruppert, D.} \& \textsc{Matteson, D.~S.} (2011).
\newblock \textit{Statistics and Data Analysis for Financial Engineering}.
\newblock Springer.

\bibitem[{Silverstein \& Choi(1995)}]{silverstein1995analysis}
\textsc{Silverstein, J.~W.} \& \textsc{Choi, S.-I.} (1995).
\newblock Analysis of the limiting spectral distribution of large dimensional random matrices.
\newblock \textit{Journal of Multivariate Analysis} \textbf{54}, 295--309.

\bibitem[{Tao(2012)}]{tao2012topics}
\textsc{Tao, T.} (2012).
\newblock \textit{Topics in Random Matrix Theory}.
\newblock American Mathematical Society.

\bibitem[{Wang \& Lopes(2023)}]{Wang:Lopes:2023}
\textsc{Wang, S.} \& \textsc{Lopes, M.~E.} (2023).
\newblock A bootstrap method for spectral statistics in high-dimensional elliptical models.
\newblock \textit{Electronic Journal of Statistics} \textbf{17}, 1848--1892.

\bibitem[{Yao \& Lopes(2023)}]{Yao:Lopes:2023b}
\textsc{Yao, J.} \& \textsc{Lopes, M.~E.} (2023).
\newblock Rates of bootstrap approximation for eigenvalues in high-dimensional {PCA}.
\newblock \textit{Statistica Sinica} \textbf{33}, 1461--1481.

\bibitem[{Yao et~al.(2015)Yao, Zheng \& Bai}]{yao2015}
\textsc{Yao, J.}, \textsc{Zheng, S.} \& \textsc{Bai, Z.} (2015).
\newblock \textit{Large Sample Covariance Matrices and High-Dimensional Data Analysis}.
\newblock Cambridge University Press.

\bibitem[{Yu et~al.(2023)Yu, Xie \& Zhou}]{Yu:2023}
\textsc{Yu, L.}, \textsc{Xie, J.} \& \textsc{Zhou, W.} (2023).
\newblock Testing {K}ronecker product covariance matrices for high-dimensional matrix-variate data.
\newblock \textit{Biometrika} \textbf{110}, 799--814.

\bibitem[{Yu et~al.(2025)Yu, Zhao \& Zhou}]{Yu:2024}
\textsc{Yu, L.}, \textsc{Zhao, P.} \& \textsc{Zhou, W.} (2025).
\newblock Testing the number of common factors by bootstrapped sample covariance matrix in high-dimensional factor models.
\newblock \textit{Journal of the American Statistical Association} \textbf{120}, 448--459.

\bibitem[{Zhang et~al.(2026)Zhang, Zhang \& Pan}]{zhang:zhang:pan:2025}
\textsc{Zhang, B.}, \textsc{Zhang, Z.} \& \textsc{Pan, G.} (2026).
\newblock Estimating the number of significant components in high-dimensional principal component analysis.
\newblock \textit{Biometrika} \textbf{113}, asaf092.

\bibitem[{Zhang et~al.(2022)Zhang, Zheng, Pan \& Zhong}]{Zhang:2022}
\textsc{Zhang, Z.}, \textsc{Zheng, S.}, \textsc{Pan, G.} \& \textsc{Zhong, P.-S.} (2022).
\newblock Asymptotic independence of spiked eigenvalues and linear spectral statistics for large sample covariance matrices.
\newblock \textit{The Annals of Statistics} \textbf{50}, 2205--2230.

\end{thebibliography}

 \newpage
\setcounter{page}{1}
\nolinenumbers
\thispagestyle{empty}
\appendix 
\makeatletter

\@addtoreset{equation}{section}
\makeatother

\renewcommand{\theequation}{\thesection\arabic{equation}}

\section*{Supplementary Material to "Tracy-Widom, Gaussian, and Bootstrap: Approximations for Leading Eigenvalues in High-Dimensional PCA"}
\vspace{1em}

\centerline{N. Dörnemann and M. E. Lopes}

\vspace{2em}

\section*{Notation and preliminaries for proofs}
For two sequences of non-negative real numbers $(a_n)$ and $(b_n)$, we write $a_n \lesssim b_n$, if there exists a constant $C>0$, independent of $n$, such that $a_n \leq C\, b_n$ for all large $n\in\N$. We write $a_n \asymp b_n$ if $a_n \lesssim b_n$ and $b_n \lesssim a_n$. Symbols such as $c$ and $C$ will sometimes be used to denote positive constants not depending on $n$ whose value may differ at each occurrence.

  \section{Proofs for Section~\ref{sec_null}}\label{sec_proofs}

 \subsection{Preparations for the proof of Proposition \ref{thm_consistency_xi}} \label{sec_proof_thm_conv_T_null}
Let $[l_n,r_n]$ denote the rightmost interval in $\operatorname{supp}(\underline{F}^{y_n,H_n})$, which consists of a finite union of compact intervals under the assumptions of Proposition~\ref{thm_consistency_xi}~\citep[Lemma 2.6]{knowles2017anisotropic}. We say that the rightmost edge of $\operatorname{supp}(\underline{F}^{y_n,H_n})$ is \textit{regular} if there is a constant $c > 0$ fixed with respect to $n$ such that  
\begin{align} \label{def_regular_edge}
    r_n - l_n \geq c \quad \quad \text{and }\quad \quad
    \lambda_1(\bfSigma)(\xi_{n,0}+c)\leq 1\tag{RE}
\end{align}
hold for all large $n$.

The following proposition is proved in Section \ref{sec_proof_aux_results}. 

\begin{proposition}
    \label{thm_tw_law}
    Suppose that assumptions \ref{ass_mp_regime}-\ref{ass_lsd} are satisfied and $\mathsf{H}_{0,n}$ holds for all large $n$. Then, the rightmost edge of $\operatorname{supp}(\underline{F}^{y_n,H_n})$ is regular in the sense of \eqref{def_regular_edge} for all large $n$.
    Moreover, under these conditions we have
    \begin{align}
        \xi_{n,k} & \asymp 1, \quad k\in \{0,1\}\tag{a},
        \label{eq_xi_order_1}\\
        \sigma_n &\asymp 1,\tag{b} \label{eq_sigma_order_1}\\
        r_n &\asymp 1.\tag{c} \label{eqn:rnorder1}
    \end{align}
\end{proposition}

\begin{remark}    \label{rem:1} 
Proposition \ref{thm_tw_law} shows that all the conditions of \cite[Corollary 3.19]{knowles2017anisotropic} are satisfied, which implies that for every $k\in\N$, we have
    \begin{align}\label{eqn:jointTWlimit}
       \frac{n^{2/3}}{\sigma_n} \lb \lambda_i (\bfSigmahat) - r_n \rb _{1 \leq i \leq k}
       \cond (\zeta_1,\dots,\zeta_k)
    \end{align}
    as $n\to\infty$, where $(\zeta_1,\dots,\zeta_k)$ has a $k$-dimensional Tracy-Widom distribution of type 1. It is also worth noting that many other variations of this result have been established, such as in~\citep{bao2015, leeschnelli2016, ding2018necessary}, among others.
\end{remark}

 \subsection{Proof of Proposition \ref{thm_consistency_xi} and Corollary \ref{cor:CI}} \label{sec_proof_consistency_xi}
The consistency result given in Proposition \ref{thm_consistency_xi}\eqref{eqn:hatxithm} is a consequence of the following lemma.

 \begin{lemma} \label{lem1_xi}
  Suppose, $\mathsf{H}_{0,n}$ holds for all large $n$ and \ref{ass_mp_regime}-\ref{ass_lsd} are satisfied. Under these conditions, there exists a fixed number $\delta\in(0,1/6)$ such that if $\eta_n := n^{-1+\delta}$, then as $n\to\infty$
 \begin{align}\label{eqn:hatxlim}
 \hat\xi_n \,
 + \,\su_n^0 ( \lambda_1(\hat\bfSigma) + i \eta_n ) \tag{a}
  = \op, 
 \end{align}
 and
   \begin{align}\label{eqn:xilim}
     \xi_{n,k} \, + \, \su_n^0 ( \lambdahat + i \eta_n ) \tag{b}
 = \op, \quad k \in \{0,1\}. 
 \end{align}

 \end{lemma}

 This result will be proven later in this subsection.

The next result provides an upper bound on the number of sample eigenvalues in a small interval around $\lambdahat$ and is based on an approximation in terms of the \MP distribution $\Fu^{y_n,H_n}$. To state the result, we define 
\begin{align*}
    \hat N_n := \sum_{j=1}^n 1\{  \lambda_j (\bfSigmahatu) > \lambda_1(\bfSigmahatu) - n^{-1/3} \}
\end{align*}
as the number of sample eigenvalues greater than $\lambda_1(\bfSigmahatu)-n^{-1/3}$.
\begin{proposition} \label{prop_order_hat_N_n}
    Suppose that $\mathsf{H}_{0,n}$ holds for all large $n$ and that assumptions \ref{ass_mp_regime}-\ref{ass_lsd} are satisfied. Then, 
    \begin{align*}
        \hat N_n = \mathcal{O}_{\PR}\lb n^{1/2}\rb. 
    \end{align*}
\end{proposition}
In the proof of Proposition \ref{prop_order_hat_N_n},  
 we utilize the square-root behavior of the density of the \MP distribution close to the rightmost edge of its support, as recorded below.
\begin{lemma}\cite[Lemma 2.1]{bao2013local}
    \label{lem_sqrt_mp_density}
      Suppose that Assumptions~\ref{ass_mp_regime} and \ref{ass_lsd} hold, and that $\mathsf{H}_{0,n}$ holds for all large $n$.
      Then, $\underline{F}^{y_n,H_n}$ has a continuous derivative, denoted $\rho_n^0$, on $\R \setminus \{0\}$, and there exist constants $C,c>0$ not depending on $n$ such that 
        \begin{align*}
           \frac{1}{C}\sqrt{r_n - \lambda}\leq  \rho_n^0 (\lambda) \leq C \sqrt{r_n - \lambda} \textnormal{ \ \ holds for all } \lambda \in [r_n - c, r_n].
        \end{align*}
\end{lemma}
This puts us in position to prove Proposition \ref{prop_order_hat_N_n}.
\begin{proof}[of Proposition \ref{prop_order_hat_N_n}]
Observe that 
\begin{equation}
    \hat N_n \ = \  n \Big(1-F^{\bfSigmahatu}(\lambda_1(\bfSigmahatu)-n^{-1/3})\Big).
\end{equation}
By Theorem~3.3(i) in~\citep{bao2015}, there is a constant $C>0$ not depending on $n$ such that the event $\lambda_1(\bfSigmahatu)\in [r_n\pm Cn^{-1/3}]$ occurs with probability approaching $1$ as $n\to\infty$, and so the event
\begin{equation}
    \hat N_n \ \leq \  n \Big(1-F^{\bfSigmahatu}(r_n-(C+1)n^{-1/3})\Big)
\end{equation}
also holds with probability $1-o(1)$. Next, Theorem~3.3(ii) in~\citep{bao2015} implies that the right side of the previous inequality can be further upper bounded by replacing $F^{\bfSigmahatu}$ with $\Fu^{y_n,H_n}$ at the expense of an error of order $\mathcal{O}(n^{1/2})$. In other words, the inequality
\begin{equation}
    \hat N_n \ \leq \  n \Big(1-\Fu^{y_n,H_n}(r_n-(C+1)n^{-1/3})\Big) \ + \ \mathcal{O}(n^{1/2})
\end{equation}
holds with probability $1-o(1)$.  Consequently, the square-root behavior in Lemma~\ref{lem_sqrt_mp_density} implies that
\begin{equation}
\begin{split}
    \hat N_n & \ \lesssim \ n \int_{r_n-(C+1)n^{-1/3}}^{r_n} \sqrt{r_n-\lambda}d\lambda \ + \ n^{1/2}\\
    & \ \lesssim \ n^{1/2}  
    \end{split}
\end{equation}
holds with probability approaching 1, as needed.
\end{proof}

For the proof of Lemma \ref{lem1_xi}\eqref{eqn:hatxlim}, it is convenient to recall the local law at the rightmost edge $r_n$. For this purpose, we define the spectral domain 
$$ 
\mathbf{D}_n(\tau,\tau')=\Bigg\{z\in\C^+:  |\re(z)-r_n|\leq \tau' \ , \  n^{-1+\tau}\leq \im(z)\leq 1\Bigg\},
 $$
 where $\tau,\tau' >0$ are fixed with respect to $n.$
In the statement of this result, recall that $\su_n=s_{F^{\underline{\hat\bfSigma}}}$ is the Stieltjes transform of the empirical spectral distribution corresponding to the companion matrix $\underline{\bfSigmahat}$ defined in \eqref{eq_def_companion}.
\begin{lemma}\label{thm_local_law} 
    Suppose that $\mathsf{H}_{0,n}$ holds for all large $n$ and that assumptions \ref{ass_mp_regime}-\ref{ass_lsd} are satisfied.  Let $\tau \in (0,1)$ be fixed. 
Then, there exist a constant $ \tau'>0$ such that the following limit holds as $n\to\infty$,
  \begin{align*}
\sup_{z\in \mathbf{D}_n(\tau, \tau')}  |\su_n (z) - \su_n^0(z)|  = o_{\P}(1) . 
    \end{align*}
    \end{lemma} 
Lemma~\ref{thm_local_law} follows from the local law given in \cite[Theorem 3.2]{bao2015} and Proposition \ref{thm_tw_law}. 
 Note that, although the result of \cite{bao2015} is stated under the boundedness condition on the random variables $\{x_{ij}\}$,  the authors explain that this condition is without loss of generality, and can be dropped using truncation arguments.
 Now we are in the position to prove Lemma \ref{lem1_xi}\eqref{eqn:hatxlim}. 

 \begin{proof}[of Lemma \ref{lem1_xi}\eqref{eqn:hatxlim}] 
 We will first show that
 \begin{equation}\label{goal1}
     \hat\xi_n
 +  \su_n ( \lambda_1(\hat\bfSigma) + i \eta_n ) = o_{\P}(1).
 \end{equation}
 Recall from~\eqref{eqn:xihatdef} that $\hat\xi_n=-\underline{\tilde s}_n(\lambda_1(\hat\bfSigma))$ holds when $\lambda_1(\hat\bfSigma)>\lambda_2(\hat\bfSigma)$. Using the limit~\eqref{eqn:jointTWlimit} and the fact that $\zeta_1 \neq \zeta_2$ holds almost surely when $(\zeta_1,\zeta_2)$ has a (continuous) two-dimensional Tracy-Widom distribution, it follows that the event $\lambda_1(\hat\bfSigma)>\lambda_2(\hat\bfSigma)$ occurs with probability approaching 1 as $n\to\infty$. For this reason, we may work asymptotically under assumption that $\hat\xi_n=-\underline{\tilde s}_n(\lambda_1(\hat\bfSigma))$. This leads to
  \begin{align*}
  	| \hat\xi_n
 +  \su_n ( \lambda_1(\hat\bfSigma) + i \eta_n ) |
 & \leq  \frac{1}{n \eta_n} + \sum_{j=2}^n  \frac{1}{n}\left| 
 \frac{1}{\lambda_j(\underline{\bfSigmahat} ) - \lambda_1(\underline{\bfSigmahat} ) - i \eta_n } -
 \frac{1}{\lambda_j(\underline{\bfSigmahat} ) - \lambda_1(\underline{\bfSigmahat} )}\right| 
\\ 
 & =: \frac{1}{n \eta_n}+\mathcal{K}_{1,n} + \mathcal{K}_{2,n} ,
  \end{align*}
  where we define $\mathcal{K}_{1,n}$ and $\mathcal{K}_{2,n}$ to respectively denote the sums over the indices $j$ such that $\lambdahat - \lambda_j(\bfSigmahatu) < n^{-1/3} $ and $\lambdahat - \lambda_j(\bfSigmahatu) \geq n^{-1/3} $.

   Recalling $\eta_n=n^{-1+\delta}$, we have $\frac{1}{n\eta_n}=o(1)$, and so proving \eqref{goal1} reduces to showing that $\mathcal{K}_{1,n}$, and  $\mathcal{K}_{2,n}$ are both $ \op$. From the limit \eqref{eqn:jointTWlimit}, we have 
   $n^{-4/3}=\mathcal{O}_{\P}(( \lambda_1(\underline{\bfSigmahat} ) - \lambda_2(\underline{\bfSigmahat} ) ) ^2)$ and so Proposition \ref{prop_order_hat_N_n} and $\delta \in (0,1/6)$ imply
   \begin{align*}
       \mathcal{K}_{1,n} \leq \frac{\hat N_n}{n} \frac{\eta_n}{\lb \lambda_1(\underline{\bfSigmahat} ) - \lambda_2(\underline{\bfSigmahat} ) \rb ^2 } 
       = \mathcal{O}_{\PR} \lb n^{-1/2} n^{-1+\delta} n^{4/3} \rb = \op. 
   \end{align*}
  Turning to $\mathcal{K}_{2,n}$, the constraint $\lambdahat - \lambda_j(\bfSigmahatu) \geq n^{-1/3}$ implies
   \begin{align*}
       \mathcal{K}_{2,n} & = \mathcal{O}_{\PR} \lb  \eta_n n^{2/3} \rb = \op. 
   \end{align*}
  Thus, \eqref{goal1} holds true, and it remains to show that
  \begin{align}
  \label{goal20}
  \su_n ( \lambda_1(\hat\bfSigma) + i \eta_n ) - \su_n^0 ( \lambda_1(\hat\bfSigma) + i \eta_n ) = \op. 
  \end{align}
 By Lemma \ref{thm_local_law}, it suffices to verify that for any choices of $\tau'>0$ and $\tau \in (0,1/6)$,  the constant $\delta\in(0,1/6)$ can be chosen so that $\eta_n=n^{-1+\delta}$ implies
 \begin{align}\label{eqn:domainevent}
 \P\Big(\lambdahat + i \eta_n \in \mathbf{D}_n(\tau, \tau')\Big)\to 1
 \end{align}
 as $n\to\infty$.
For any $\tau'>0$, the limit~\eqref{eqn:jointTWlimit} implies that the event 
\begin{align*}
    | \lambdahat - r_n | \leq \tau' 
\end{align*}
has probability approaching 1. Also, if we take $\delta= \tau$ in the definition of $\eta_n$, then we have $n^{-1+\tau} = \eta_n \leq 1$ for all $n$. Thus, the limit \eqref{eqn:domainevent} holds, and the proof of \eqref{goal20} concludes. 
 \end{proof}

Before turning to the proof of Lemma \ref{lem1_xi}\eqref{eqn:xilim}, we need two auxiliary results, which will be proven in Section \ref{sec_proof_aux_results}. 
The following lemma shows that the critical parameters $\xi_{n,0}$ and $\xi_{n,1}$ are asymptotically equivalent.   
 \begin{lemma} \label{lem_xin_xi1}
If assumptions \ref{ass_mp_regime} and \ref{ass_lsd} are satisfied, and if $\mathsf{H}_{0,n}$ holds for all large $n$, then 
\begin{align*}
0 < \xin{1} - \xi_{n,0}  \lesssim  \textstyle \frac{1}{p}.
\end{align*}
\end{lemma} 

The next lemma says that the sequences $(\xi_{n,k})_{n\in\N}$, $k\in \{0,1\},$ of critical parameters are individually convergent.
\begin{lemma}\label{lem:xilim}
    Suppose that Assumptions~\ref{ass_mp_regime} and \ref{ass_lsd} hold, and that $\mathsf{H}_{0,n}$ holds for all large $n$.  Then, there is a value $\xi_0\in(0,\infty)$ such that
    \begin{equation}
        \lim_{n\to\infty}\xi_{n,k}=\xi_0 
    \end{equation}
    holds for $k\in\{0,1\}$. Moreover, if $u$ denotes the rightmost endpoint of the support of $H$, then $\xi_0$ satisfies $u\xi_0\leq \frac{1}{1+\varepsilon}$.
\end{lemma}
Based on these preparations, we are in the position to prove Lemma \ref{lem1_xi}\eqref{eqn:xilim}.

\begin{proof}[of Lemma \ref{lem1_xi}\eqref{eqn:xilim}] 
By Lemma~\ref{lem_xin_xi1}, it is enough to prove the result in the case of $k=0$. Let $\tau_n = n^{-2/3 + \delta}$, and recall $\eta_n = n^{-1+\delta}.$ 
Then, we write
\begin{align*}
   \xi_{n,0} + \su_n^0 (\lambdahat + i \eta_n) 
    &= \su_n^0 (\lambdahat + i \eta_n) - \su_n^0(r_n) \\ 
    & = \int_0^{r_n} \frac{ \lambdahat + i \eta_n - r_n }{\lb  \lambda - r_n \rb \lb  \lambda - \lambdahat - i \eta_n \rb } d \underline{F}^{y_n,H_n} (\lambda) \\
    & =: \mathcal{J}_{1,n} + \mathcal{J}_{2,n} + \mathcal{J}_{3,n},
\end{align*}
where the terms $\mathcal{J}_{1,n}$, $\mathcal{J}_{2,n}$, $\mathcal{J}_{3,n}$ correspond to decomposing the integral over $[0,r_n]$ into the three intervals $[0,r_n-c]$, $[r_n-c,r_n-\tau_n]$, and $[r_n-\tau_n,r_n]$,
with $c>0$ being the constant from Lemma \ref{lem_sqrt_mp_density}.
For the first of these terms, the limit~\eqref{eqn:jointTWlimit} implies that there is some constant $c'>0$ such that the event 
$$
\Omega_{1,n} = \bigg\{ 
|\lambda - \lambdahat| \geq | r_n - c - \lambdahat| \geq  c' \textnormal{ for all } \lambda \in [0, r_n-c] \bigg\} $$ 
has probability approaching 1 as $n\to\infty$. This gives
\begin{align*}
    | \mathcal{J}_{1,n} | & = | \mathcal{J}_{1,n} | 1 \{ \Omega_{1,n} \}  + \op\\
    &\lesssim | \lambdahat + i \eta_n - r_n | + \op\\
    &= \op. 
\end{align*}
Regarding $\mathcal{J}_{2,n}$, we define the event
\begin{align*}
    \Omega_{2,n} = \{ | \lambda - \lambdahat | \geq | r_n - \tau_n - \lambdahat | \textnormal{ for all } \lambda \in [r_n - c, r_n -\tau_n] \}.
\end{align*}
By the definition of $\tau_n$ and the limit \eqref{eqn:jointTWlimit}, we see that $\lambdahat \geq r_n - \tau_n$ with probability converging to 1 and thus, $\PR \lb \Omega_{2,n} \rb = 1 - o(1).$  Moreover, we get as a consequence of the limit \eqref{eqn:jointTWlimit} and the definitions of $\tau_n$ and $\eta_n$ that
\begin{align} \label{eq_estimate_ratio}
      \lambdahat  - r_n  = o_{\PR}(\tau_n), \quad  \eta_n = o(\tau_n), \quad  \left| \tau_n\inv \lb r_n  - \lambdahat \rb  - 1 - \tau_n\inv i \eta_n  \right| = 1 + \op .
\end{align}
Consequently, we have
\begin{align*}
    | \mathcal{J}_{2,n} | & =  | \mathcal{J}_{2,n} | 1 \{ \Omega_{2,n} \} + \op \\ & 
    \lesssim \left| \frac{\lambdahat + i \eta_n - r_n}{  r_n - \tau_n - \lambdahat - i \eta_n } \right| \int_{r_n - c}^{r_n - \tau_n} \frac{ 1 }{ r_n - \lambda   } d \underline{F}^{y_n,H_n} (\lambda)
    + \op \\
    & \lesssim \left| \frac{\lambdahat + i \eta_n - r_n}{  r_n - \tau_n - \lambdahat - i \eta_n } \right| \int_{r_n - c}^{r_n - \tau_n} \frac{1}{\sqrt{r_n - \lambda}} d  \lambda  
    + \op \ \ \ \ \  \ \textup{(Lemma \ref{lem_sqrt_mp_density})} \\
    & \lesssim
     \frac{ \tau_n\inv | \lambdahat + i \eta_n - r_n | }{  | \tau_n\inv ( r_n  - \lambdahat )  - 1 - \tau_n\inv i \eta_n | }
    + \op\\
    &= \op,
\end{align*}
where we used \eqref{eq_estimate_ratio} for the last step, implying that the numerator is $\op$, whereas the denominator is $1+\op$. 

It is left to show that $\mathcal{J}_{3,n} = \op.$ Using Lemma \ref{lem_sqrt_mp_density}, we see that
\begin{align*}
    | \mathcal{J}_{3,n} | & \lesssim \left| \frac{ \lambdahat + i \eta_n - r_n}{\eta_n} \right| \int_{r_n - \tau_n}^{r_n } \frac{1 }{r_n - \lambda  } d \underline{F}^{y_n,H_n} (\lambda) \\
    & \lesssim \left| \frac{ \lambdahat + i \eta_n - r_n}{\eta_n} \right| \int_{r_n - \tau_n}^{r_n } \frac{1 }{ \sqrt{r_n - \lambda } } d \lambda \\ 
    & = \left| \frac{ \lambdahat + i \eta_n - r_n}{\eta_n} \right|\sqrt{\tau_n}
    = \mathcal{O}_{\PR} \lb n^{-2/3} \tau_n^{1/2} \eta_n^{-1} \rb = \mathcal{O}_{\PR} \lb n^{\delta/2 - \delta} \rb = \op,
\end{align*}
as needed.
\end{proof}

\begin{proof}[of Proposition \ref{thm_consistency_xi}\eqref{eqn:hatHthm}]

 For any probability distributions $\mu$ and $\nu$ on $\R$, define the bounded Lipschitz metric as
\begin{equation}
    d_{\textup{BL}}(\mu,\nu)=\sup_{f\in\textup{BL}}\Bigg|\int f(x)d\mu(x)-\int f(x)d\nu(x)\Bigg|, \label{def_BL_metric}
\end{equation}
where $\textup{BL}$ is the class of 1-Lipschitz functions $f:\R\to\R$ such that $\sup_{x\in\R}|f(x)|\leq 1$.
It is a well known fact that a sequence of probability distributions $\{\mu_n\}$ on $\R$ satisfies $\mu_n\xrightarrow{\D}\nu$ as $n\to\infty$ if and only if $d_{\textup{BL}}(\mu_n,\nu)=o(1)$. So, to prove Proposition~\ref{thm_consistency_xi}(b) it is enough to show that $d_{\textup{BL}}(\tilde H_n, H)=o_{\P}(1)$. Letting
$$\tilde H_{n,Q}(t)=\frac{1}{p}\sum_{j=1}^p 1\{\tilde \lambda_{j,Q}\leq t\}$$
denote the QuEST estimator of $H$,
it is known from~\citep[][p.382]{ledoit2015spectrum} that $d_{\textup{BL}}(\tilde H_{n,Q},H)=o_{\P}(1)$, and so the remainder of the proof will focus on showing that $d_{\textup{BL}}(\tilde H_n,\tilde H_{n,Q})=o_{\P}(1)$.

Letting $\hat b_n=1/(\hat\xi_n(1+\varepsilon))$ and letting $u$ denote the right endpoint of the support of $H$, Lemma~\ref{lem:xilim} and Proposition~\ref{thm_consistency_xi}(a) ensure there is some number $b_0\geq u$ such that $\hat b_n\to b_0$ in probability. Now, consider the basic inequality

\begin{align*}
   d_{\textup{BL}}(\tilde H_n,\tilde H_{n,Q}) 
    &\leq \sup_{f\in\textup{BL}}\bigg|\frac{1}{p}\sum_{j=1}^p f(\tilde \lambda_{j,Q}\wedge \hat b_n)-f(\tilde \lambda_{j,Q}\wedge b_0)\bigg| 
    +\sup_{f\in\textup{BL}}\bigg|\frac{1}{p}\sum_{j=1}^p f(\tilde \lambda_{j,Q}\wedge b_0)-f(\tilde\lambda_{j,Q})\bigg|\\
    &=: \sup_{f\in\textup{BL}}\textup{I}_n(f)+\sup_{f\in\textup{BL}}\textup{II}_n(f).
\end{align*}
For any fixed $\delta>0$ and $f\in\textup{BL}$, the quantity $\textup{I}_n(f)$ can be bounded as
\begin{align}
    \textup{I}_n(f)&=\textup{I}_n(f)\cdot 1\{|\hat b_n-b_0|\leq \delta\}+\textup{I}_n(f)\cdot 1\{|\hat b_n-b_0|>\delta\} \nonumber \\
    &\leq \delta + 2\cdot 1\{|\hat b_n-b_0|>\delta\}.  \nonumber
\end{align}
 Since $\hat b_n\to b_0$ in probability and $\delta$ may be taken arbitrarily small, we conclude that $\sup_{f\in\textup{BL}}\textup{I}_n(f)=o_{\P}(1)$.

To handle $\textup{II}_n(f)$, observe that $f(\tilde\lambda_{j,Q}\wedge b_0)$ and $f(\tilde\lambda_{j,Q})$ can only disagree when $\tilde\lambda_{j,Q}>b_0$, and so
\begin{equation}
    \textup{II}_n(f)=\bigg|\frac{1}{p}\sum_{j=1}^p (f(b_0)-f(\tilde\lambda_{j,Q}))\cdot 1\{\tilde\lambda_{j,Q}>b_0\}\bigg|. \label{eq_term_II_n}
\end{equation}
Next, let $\delta>0$ be fixed, and consider decomposing $1\{\tilde \lambda_{j,Q}>b_0\}$ into $1\{\tilde \lambda_{j,Q}>b_0+\delta\}+1\{b_0+\delta\geq \tilde \lambda_{j,Q}>b_0\}$. Thus, for any $f\in\textup{BL}$, we have
\begin{align*}
    \textup{II}_n(f)&\leq \Bigg(\frac{2}{p}\sum_{j=1}^p 1\{\tilde\lambda_{j,Q}>b_0+\delta\}\Bigg)+ \delta \nonumber \\
    &= 2[1-\tilde H_{n,Q}(b_0+\delta)]+\delta,
\end{align*}
where $\tilde H_{n,Q}$ is viewed as a distribution function.
Since $b_0\geq u$, it follows that $b_0+\delta$ is a continuity point of $H$ with $H(b_0+\delta)=1$. So, the consistency of QuEST implies $\tilde H_{n,Q}(b_0+\delta)\to 1$ in probability. Moreover, since $\delta$ is arbitrary, we have $\sup_{f\in\textup{BL}}\textup{II}_n(f)=o_{\P}(1)$. 
\end{proof}

 Next, we prove the assertions regarding the estimator of the scale parameter given in Proposition \ref{thm_consistency_xi}\eqref{eqn:hatsigmathm}. The proof relies crucially on the consistency of $\hat\xi_n$ and $\tilde H_n$ given in Proposition \ref{thm_consistency_xi}\eqref{eqn:hatxithm} and Proposition \ref{thm_consistency_xi}\eqref{eqn:hatHthm}.
 
 \begin{proof}[of Proposition \ref{thm_consistency_xi}\eqref{eqn:hatsigmathm}]

By part \eqref{eq_sigma_order_1} of Proposition~\ref{thm_tw_law}, we have $\sigma_n\asymp 1$, and so in order to show $\hat\sigma_n/\sigma_n - 1=o_{\P}(1)$, it is enough to show $\hat\sigma_n-\sigma_n=o_{\P}(1)$, which is in turn implied by $\hat\sigma_n^3-\sigma_n^3=o_{\P}(1)$ (using a first-order Taylor expansion of the cube root function at $\sigma_n^3$).
Moreover, since Lemma~\ref{lem:xilim} and Proposition \ref{thm_consistency_xi}\eqref{eqn:hatxithm} show that there is a positive value $\xi_0$ such that $\hat\xi_n=\xi_0+o_{\P}(1)$ and $\xi_{n,0} = \xi_0+o(1)$, the limit $\hat\sigma_n^3-\sigma_n^3=o_{\P}(1)$ reduces to showing that the quantity
\begin{equation}
  \hat\Delta_n:=  \int \Big(\ts\frac{\lambda \hat\xi_n}{1-\lambda\hat\xi_n}\Big)^3 d\tilde H_n(\lambda) - \displaystyle\int\Big(\ts\frac{\lambda \xi_{n,0}}{1-\lambda\xi_{n,0}}\Big)^3 dH_n(\lambda)
\end{equation}
satisfies $\hat\Delta_n=o_{\P}(1)$.

To proceed, define the event $E_n=\{\tilde\lambda_1\xi_0\leq \frac{1}{1+\varepsilon/2}\}$, and write $f(x)=(x/(1-x))^3$ as a temporary shorthand so that we may upper bound $|\hat\Delta_n|$ as
\begin{align*}
       |\hat\Delta_n| & \ \leq  \  1\{E_n^c\}\!\!\int f(\lambda\hat\xi_n)d\tilde H_n(\lambda)   \ + \  1\{E_n\}\bigg|\int \big(f(\lambda\hat\xi_n)-f(\lambda\xi_0)\big)d\tilde H_n(\lambda)\bigg|\\
       &  \ \ \ \ + \ \bigg| 1\{E_n\}\!\!\int f(\lambda\xi_0)d\tilde H_n(\lambda)
       -\int f(\lambda \xi_0)dH_n(\lambda)\bigg|\\
       &\quad+
\left|
\int \bigl(f(\lambda\xi_0)-f(\lambda\xi_{n,0})\bigr)
dH_n(\lambda)
\right|\\
&=:\textup{A}_n+\textup{B}_n+\textup{C}_n+\textup{D}_n.
\end{align*}
With regard to the term $\textup{A}_n$, note that $f(\tilde\lambda_j\hat\xi_n)$ is upper bounded by $f(\frac{1}{1+\varepsilon})$ for all $j$ and $n$. Also, the definition of $\tilde\lambda_1$ ensures that 
$\tilde\lambda_1
\leq \frac{1}{\hat\xi_n(1+\varepsilon)}$, and so
$$\P(E_n)\geq \P\Big(\ts\frac{1}{\hat\xi_n(1+\varepsilon)}\leq \frac{1}{\xi_0(1+\varepsilon/2)}\Big).$$
Furthermore, the probability on the right tends to 1 due to the limit $\hat\xi_n\xrightarrow{\P} \xi_0$, which implies $\P(E_n)\to 1$. Hence, we conclude that $\textup{A}_n=o_{\P}(1)$.

Turning to the term $\textup{B}_n$, note that $\tilde\lambda_j\hat\xi_n$ and $\tilde \lambda_j\xi_0$ both reside in the interval $[0,\frac{1}{1+\varepsilon/2}]$ for all $j=1,\dots,p$ when the event $E_n$
holds, and that $f$ has a $\mathcal{O}(1)$-Lipschitz constant on this interval. Thus,
\begin{equation}
    \textup{B}_n \ \lesssim \ 1\{E_n\}\!\int \lambda |\hat\xi_n-\xi_0| d\tilde H_n(\lambda) \ \leq \  \frac{|\hat\xi_n-\xi_0|}{\xi_0(1+\varepsilon/2)} \ = \ o_{\P}(1).
\end{equation} 

To handle the term $\textup{C}_n$, define the function $g(x)$ to agree with $f(x)$ for all $x\in [0,\frac{1}{1+\varepsilon/2}]$ and satisfy 
 $g(x)=f(\frac{1}{1+\varepsilon/2})$ for all $x>\frac{1}{1+\varepsilon/2}$. Hence, $g$ is a bounded continuous function on the non-negative real-line. Furthermore, the quantities $1\{E_n\}f(\lambda\xi_0)$ and $1\{E_n\}g(\lambda\xi_0)$ agree for all $\lambda$ in the support of $\tilde H_n$. So, by combining $1\{E_n\}\xrightarrow{\P}1$ with the limit $\tilde H_n\xrightarrow{\D}H$ in probability (by Proposition \ref{thm_consistency_xi}\eqref{eqn:hatHthm}), we have
\begin{equation}
    1\{E_n\}\!\int f(\lambda\xi_0)d\tilde H_n(\lambda) \xrightarrow{\P} \int g(\lambda\xi_0)dH(\lambda).
\end{equation}
 Then, Lemma~\ref{lem:xilim} and $\mathsf{H}_{0,n}$  show that $f(\lambda\xi_0)$ and $g(\lambda\xi_0)$ agree for all $\lambda$ in the support of $H_n$ when $n$ is sufficiently large, and so the weak limit $H_n\xrightarrow{\D}H$ under~\ref{ass_lsd} implies that $\int f(\lambda\xi_0)dH_n(\lambda)$ converges to $\int g(\lambda\xi_0)dH(\lambda)$ as $n\to\infty$. Therefore $\textup{C}_n=o_{\P}(1)$.

Finally, Lemma~\ref{lem:xilim}, $\mathsf{H}_{0,n}$,
and assumption~\ref{ass_lsd} ensure that, for all large $n$,
both $\lambda\xi_0$ and $\lambda\xi_{n,0}$ lie in
$[0,1/(1+\varepsilon/2)]$ for every
$\lambda\in\operatorname{supp}(H_n)$.
Since $f$ is $\mathcal{O}(1)$-Lipschitz on this interval, we have
\[
\textup{D}_n
\lesssim |\xi_{n,0}-\xi_0|\int\lambda\,dH_n(\lambda)
=o(1),
\]
which completes the proof.
\end{proof}

Next, we establish the validity of the confidence intervals in Corollary \ref{cor:CI}. 
\begin{proof}[of Corollary \ref{cor:CI}] The second assertion of the corollary is a direct consequence of the limit~\eqref{eqn:jointTWlimit} and Proposition~\ref{thm_consistency_xi}\eqref{eqn:hatsigmathm}.
Turning to the first assertion, Proposition~\ref{thm_consistency_xi} and $\mathsf{H}_{0,n}$ imply that the event \smash{$\{0<\hat\xi_n\lambda_1(\bfSigma)\leq 1/(1+\varepsilon/2)\}$} occurs with probability tending to 1. Define the random variable  $\breve \varepsilon>0$ to satisfy $1/ ( 1 + \breve\varepsilon) = \lambda_1(\bfSigma) \hat\xi_n$ when the mentioned event occurs, and otherwise put $\breve\varepsilon=1/2$. As a preparatory step, we show that
\begin{align} \label{aim3}
    T_n (\breve\varepsilon) \cond \zeta - \zeta', \quad n\to\infty. 
\end{align}
For each $j\in\{1,\dots,p\}$, let
\begin{align*}
    \breve\lambda_j = \tilde \lambda_{j,Q}  \wedge \frac{1}{\hat\xi_n ( 1 + \breve \varepsilon) }  = \tilde \lambda_{j,Q}  \wedge \lambda_1(\bfSigma)
\end{align*}
and define the associated spectral distribution function
\begin{align*}
    \breve H_n(t) = \frac{1}{p} \sum_{j=1}^p 1\{\breve\lambda_j\leq t\}.
\end{align*}
Using an argument analogous to the proof of Proposition~\ref{thm_consistency_xi}(b), it can be shown that as $n\to\infty$,
\begin{align} \label{consis_breve_Hn}
    \breve H_n \cond H \quad \textnormal{in probability}.
\end{align}

Next, if we let  
\begin{align*}
    \breve \sigma_n^3   =  \frac{1}{\hat\xi_n^3 } \lb 1 +  y_n \int \lb \frac{\lambda \hat\xi_n }{1 - \lambda\hat\xi_n } \rb^3 d\breve{H}_{n}(\lambda)\rb,
\end{align*}
then an argument similar to the proof of Proposition\ref{thm_consistency_xi}\eqref{eqn:hatsigmathm} can be used to show that
\begin{align}\label{eqn:sigmabrevlim}
    \breve \sigma_n - \sigma_n = \op.
\end{align}
Applying Proposition \ref{thm_tw_law} and using the limits~\eqref{eqn:jointTWlimit} and~\eqref{eqn:sigmabrevlim}, we obtain
\begin{align*}
    T_n(\breve\varepsilon) = \frac{n^{2/3}}{\breve\sigma_n} ( \lambdahat - \lambda_2(\hat\bfSigma) ) \cond \zeta - \zeta'
\end{align*}
as $n\to\infty$,
which proves the preparatory limit \eqref{aim3}.

To conclude the proof, note that
    \begin{align*}
      \PR \lb \lambda_1(\bfSigma) \leq \frac{1}{( 1 + \hat\varepsilon )  \hat\xi_n  } \rb
         &=  \PR (\hat\varepsilon \leq \breve \varepsilon )  \\
         &= \PR ( T_n\inv (q_{1-\alpha} ) \leq \breve\varepsilon ).
    \end{align*}
   Recal that $\hat\varepsilon=T_n^{-1}(q_{1-\alpha})=\inf\{\varepsilon>0:T_n(\varepsilon)\geq q_{1-\alpha}\}$. Since the event $T_n^{-1}(q_{1-\alpha})\leq \breve\varepsilon$ can only occur when $T_n^{-1}(q_{1-\alpha})$ is finite, we have
    \begin{align*}
      \PR \lb \lambda_1(\bfSigma) \leq \frac{1}{( 1 + \hat\varepsilon )  \hat\xi_n  } \rb
         &= \PR \Big(\big\{T_n\inv (q_{1-\alpha} ) \leq \breve\varepsilon\big\}\cap\big\{T_n\inv (q_{1-\alpha} )<\infty\big\}\Big).
    \end{align*}
    Moreover, since the function $T_n(\cdot)$ is continuous, we have $q_{1-\alpha}\leq T_n(T_n^{-1}(q_{1-\alpha}))$  when $T_n^{-1}(q_{1-\alpha})$ is finite. Combining this with the fact that $T_n(\cdot)$ is non-decreasing, it follows that 
        \begin{align*}
      \PR \lb \lambda_1(\bfSigma) \leq \frac{1}{( 1 + \hat\varepsilon )  \hat\xi_n  } \rb
         &\leq \PR \big(q_{1-\alpha}\leq T_n(\breve \varepsilon)\big)\\
         &=\alpha+o(1),
    \end{align*}
    where the limit~\eqref{aim3} is used in the last step. This completes the proof.
\end{proof}

\subsection{Proofs of Proposition \ref{thm_tw_law}, Lemma \ref{lem_xin_xi1} and Lemma \ref{lem:xilim}} \label{sec_proof_aux_results}

\begin{proof}[of Proposition \ref{thm_tw_law}] 
We first verify \eqref{eq_xi_order_1}. On one hand, assumption \ref{ass_lsd} implies
$$\xi_{n,1} < 1/\lambda_{2}(\bfSigma) \lesssim 1.$$ 
On the other hand, the inequality $\xi_{n,0} \gtrsim 1$ follows from $\int (\lambda\xi_{n,0}/(1-\lambda\xi_{n,0}))^2dH_n(\lambda)=1/y_n$ and $\lambda_1(\bfSigma)\lesssim 1$. Hence, it is enough to check the inequality $\xi_{n,1}>\xi_{n,0}$. For this purpose, define the functions
$f_j:[0,1/\lambda_j(\bfSigma))\to \R$ for $j=1,\dots,p$ according to
\begin{equation}\label{eqn:fjdef}
   f_j(x)=\Big(\frac{\lambda_j(\bfSigma)x}{1-\lambda_j(\bfSigma)x}\Big)^2. 
\end{equation}
From the definitions of $\xi_{n,0}$ and $\xi_{n,1}$, it is straightforward to check that they satisfy the equation 
$$\sum_{j=2}^p \big(f_j(\xin{1})-f_j(\xi_{n,0})\big) = f_1(\xi_{n,0}).$$
Since each function $f_j$ is monotone increasing and $f_1(\xi_{n,0})$ is positive, it follows that $\xin{1}>\xi_{n,0}$.  Thus, assertion \eqref{eq_xi_order_1} is true.

The first condition in \eqref{def_regular_edge} follows from \cite[(2.19)]{bao2013local}.  The second condition in \eqref{def_regular_edge} follows from $\mathsf{H}_{0,n}$ and the inequality $\xi_{n,1}>\xi_{n,0}$. Finally, \eqref{eq_sigma_order_1} and~\eqref{eqn:rnorder1}  follow from \eqref{eq_xi_order_1} under $\mathsf{H}_{0,n}$.
\end{proof}

\begin{proof}[of Lemma \ref{lem_xin_xi1}] The inequality $\xi_{n,1}>\xi_{n,0}$ was already shown in the proof of Proposition~\ref{thm_tw_law}, and so it is only necessary to show $\xi_{n,1}-\xi_{n,0}\lesssim 1/p$.
Let the functions $f_1,\dots,f_p$ be as defined as in~\eqref{eqn:fjdef}.
Due to the convexity of each $f_j$, we must have 
$$f_j(\xin{1})\geq f_j(\xi_{n,0})+f_j'(\xi_{n,0})(\xin{1}-\xi_{n,0})$$
and so 
$$ (\xin{1}-\xi_{n,0}) \sum_{j=2}^p f_j'(\xi_{n,0}) \ \leq \ f_1(\xi_{n,0}).$$
Furthermore, we have
$f_j'(\xi_{n,0})=\frac{2\lambda_j(\bfSigma)^2\xi_{n,0}}{[1-\lambda_j(\bfSigma)\xi_{n,0}]^3}\geq 2\lambda_j(\bfSigma)^2\xi_{n,0}$, and thus
$$ (\xin{1}-\xi_{n,0}) \xi_{n,0}\sum_{j=2}^p 2\lambda_j(\bfSigma)^2 \ \leq \ f_1(\xi_{n,0}).$$
By $\mathsf{H}_{0,n}$ and $\xi_{n,1}>\xi_{n,0}$, we have $f_1(\xi_{n,0})\lesssim 1$. 
Combining this with assumption \ref{ass_lsd} and $\xi_{n,0} \asymp 1$ from part \eqref{eq_xi_order_1} of Proposition~\ref{thm_tw_law}, we have  
$\xin{1}-\xi_{n,0} \ \lesssim  \ \textstyle \frac{1}{p},$
which completes the proof. 
\end{proof}

\begin{proof}[of Lemma \ref{lem:xilim}] By Lemma~\ref{lem_xin_xi1}, it is enough to prove the result in the case of $k=0$. By part~\eqref{eq_xi_order_1} of Proposition~\ref{thm_tw_law}, the sequence $\{\xi_{n,0}\}$ is bounded, and so to prove that a limit exists, it is enough to show that all convergent subsequences of $\{\xi_{n,0}\}$ converge to the same value. To this end, let $J\subset \N$ be a subsequence along which $\xi_{n,0}$ converges to some limit $\xi_0'$. It suffices to show that $\xi_0'$ resides in the interval $(0,1/u)$ and solves the equation
\begin{equation}\label{eqn:limiteqn}
    \int \Big(\ts\frac{\lambda \xi_0'}{1-\lambda\xi_0'}\Big)^2 dH(\lambda)=\frac{1}{y},
\end{equation}
because there is only one value of $\xi_0'\in (0,1/u)$ that can solve this equation.

To see that $\xi_0'$ must lie in $(0,1/u)$, note that~\ref{ass_lsd} implies 
\begin{align} \label{eq_u_leq_lambda1}
u\leq \limsup_{n\in J}\lambda_1(\bfSigma)
\end{align}
and so when $\mathsf{H}_{0,n}$ holds for all large $n$, we have
\begin{align}
u\xi_0'&\leq
\ts \limsup_{n\in J}\big(\lambda_1(\bfSigma)\xi_{n,0}\big)\nonumber \\
& \leq \ts\limsup_{n\in J}\big(\lambda_1(\bfSigma)\xi_{n,1}\big) \nonumber \\
&\leq\ts \frac{1}{1+\varepsilon}. \label{eqn:intermediate}
\end{align}
It is also useful to notice that this reasoning shows that $\lambda_1(\bfSigma)\xi_0'\leq \frac{1}{1+\varepsilon/2}$ holds for all large $n$ along $J$, which ensures that the integral $\int h(\lambda \xi_0')dH_n(\lambda)$ with $h(x)=(x/(1-x))^2$ is well defined for all large $n$ along $J$.

Next, observe that $h$ is a bounded continuous function on $[0,1/(1+\varepsilon/2)]$, and so the weak limit $H_n\xrightarrow{\D} H$ under~\ref{ass_lsd} implies $\int h(\lambda \xi_0')dH_n(\lambda)\to \int h(\lambda\xi_0')dH(\lambda)$ along $J$. Thus, to verify~\eqref{eqn:limiteqn}, it only remains to show that $\int h(\lambda\xi_0')dH_n(\lambda)\to 1/y$. Since the definition of $\xi_{n,0}$ gives $\int h(\lambda\xi_{n,0})dH_n(\lambda)=1/ y_n\to 1/y$, it is enough to check that the integrals of $h(\lambda \xi_0')$ and $h(\lambda \xi_{n,0})$ under $H_n$ approach the same limit. This follows from the fact that $h$ is Lipschitz on $[0,1/(1+\varepsilon/2)]$ and $\lambda_1(\bfSigma)\lesssim1$ under~\ref{ass_lsd}.
\end{proof}

\section{Proofs for Section \ref{sec_alternative}} 
\label{sec_proof_thm_alternative} 
For the proof of Proposition \ref{prop_sub_cond} below, note that the derivative $\psi'$ of the function $\psi$ defined in \eqref{eq_def_psi} is given by 
\begin{align}
    \label{eq_def_psi_deriv}
    \psi'(\beta) = 1 -  y  \int \frac{\lambda^2}{( \beta - \lambda ) ^2} dH(\lambda), \quad \beta\not\in \textup{supp}(H).
\end{align} 
\begin{proof}[of Proposition \ref{prop_sub_cond}]%
Denote the empirical distribution function of $\lambda_{K+1}(\bfSigma), \ldots, \lambda_p(\bfSigma)$ by 
\begin{align} \label{eq_def_H_nK}
    H_{n,K+1}(t) = \frac{1}{p-K} \sum_{j=K+1}^p 1 \{ \lambda_j(\bfSigma) \leq t\}, \quad t\in\R,
\end{align}
   and for any $\beta\not\in \textup{supp}(H_{n,K+1})$ and $n\in\N$, define the function
\begin{equation} \label{eq_def_psi_n}
    \psi_{n}(\beta) = \beta +  \frac{p - K}{n} \beta \int \frac{\lambda}{\beta - \lambda} dH_{n,K+1}(\lambda).
\end{equation}
By definition of $\xi_{n,K}$, we know that 
\begin{align} \label{w1}
    \psi_n'(1/\xi_{n,K}) =0 \quad \textnormal{for all } n\in\N. 
\end{align}
Similarly to the proof of Lemma \ref{lem:xilim}, one can show that there exists some $\xi_{K}$ such that $\lim_{n\to\infty} \xi_{n,K} = \xi_{K}. $ 
Also, arguments similar to those in the proof of Lemma \ref{lem:xilim} show that the derivative $\psi'$ is well defined at $1/\xi_K$ and $1/\xi_{n,K}$ for all large $n$, and that $\mathsf{H}_{1,n}(K)$ implies
\begin{align} \label{z4}
    \left| \psi_n' (1/\xi_{n,K}) - \psi_n' (1/\xi_K) \right|
    \lesssim \left| \xi_{n,K} - \xi_K \right| =o(1).
\end{align}
Moreover, from  $H_{n,K+1} \cond H$, we have
\begin{align} \label{z5}
    \psi_n' (1/\xi_K) - \psi'(1/\xi_K) = o(1). 
\end{align}
Combining \eqref{w1}, \eqref{z4} and \eqref{z5} gives $  \psi'(1/\xi_{K})=0$. 
Finally, since $\psi'(\beta)$ is strictly increasing for $\beta > 1/\xi_K$, and
$\lambda_j(\bfSigma) \geq 1/(\xi_{K} (1 - \varepsilon) )$ for $j=1,\dots,K$, we conclude that 
\begin{align*}
    \psi'(\lambda_j(\bfSigma)) > \psi'(1/\xi_{K}) =0
\end{align*}
as needed.
\end{proof}

To provide the necessary context for our following analysis, we restate a central limit theorem for the leading sample eigenvalues in the supercritical case. To describe their covariance structure, we need to introduce further notation. 
Let $\bfSigma=\mathbf{U}\mathbf{\Lambda}\mathbf{U}\ttop$ be a spectral decomposition of $\bfSigma$, where $\mathbf{\Lambda}=\textup{diag}(\lambda_1(\bfSigma),\dots,\lambda_p(\bfSigma))$ and $\mathbf{U}$ is an orthogonal matrix whose columns  are the eigenvectors of $\bfSigma$. For each $j=1,\dots,p$, write
$\mathbf{u}_j=(u_{j1},\dots,u_{jp})^\top$
for the $j$th column of $\mathbf{U}$.
If $K\geq 1$ denotes the number of supercritical eigenvalues, and $j,j'\in\{1,\dots,K\}$ then we define 
\begin{align*}
   \varsigma_{j,n}^2 & =  ( \E(x_{11}^4) - 3)  \lambda^2_j(\bfSigma) \left\{ \psi'(\lambda_j(\bfSigma) ) \right\}^2 \sum_{l=1}^p u_{jl}^4 + 2 \lambda_j^2(\bfSigma) \psi'(\lambda_j(\bfSigma) ) ,  \\
   \varsigma_{j,j',n} & =  ( \E(x_{11}^4) - 3) 
   \lambda_j(\bfSigma) \lambda_{j'}(\bfSigma) \psi'(\lambda_j(\bfSigma) ) \psi'(\lambda_{j'}(\bfSigma) ) \sum_{l=1}^p u_{jl}^2 u_{j'l}^2.
   \end{align*}
   Also, for $1 \leq j < K,$ we define
   \begin{align*}
   \gamma_{j,n}^2  & = \varsigma_{j,n}^2 + \varsigma_{j+1,n}^2 - 2 \varsigma_{j,j+1,n}.
\end{align*}
In the following, we provide a  central limit theorem for $\lambda_1(\bfSigmahat), \ldots, \lambda_K(\bfSigmahat).$ For this purpose, recall the definition of the function $\psi_n$ in \eqref{eq_def_psi_n}, and define the following shifted and rescaled version of the difference between two consecutive sample eigenvalues,
\begin{align*}
    D_{n,j} =  \frac{ \sqrt{n} }{\gamma_{j,n}} \lb \lambda_j(\bfSigmahat) - \psi_n( \lambda_j(\bfSigma) )  \rb - \frac{ \sqrt{n} }{\gamma_{j,n}} \lb \lambda_{j+1}(\bfSigmahat) - \psi_n( \lambda_{j+1}(\bfSigma) ) \rb.
\end{align*}
\begin{proposition}
    \label{thm_clt_sup}
    If the conditions of Theorem~\ref{thm_alternative} hold, then as $n\to\infty$, the following limit  holds for all $1 \leq j \leq K$
   \begin{align}\label{eqn:marginalspikelim}
   \frac{ \sqrt{n} }{ \varsigma_{j,n} } \lb \lambda_j(\bfSigmahat) - \psi_n(\lambda_j(\bfSigma)) \rb
   \cond \mathcal{N}(0,1).
   \end{align}
In addition, if $\mathsf{H}_{1,n}(K)$ holds with $K\geq 2$, and if $1 \leq j < K$, then as $n\to\infty$,
   \begin{align}\label{eqn:diffspikelim}
    D_{n,j} 
    \cond \mathcal{N}(0,1).
      \end{align}
\end{proposition}
\begin{proof}[of Proposition \ref{thm_clt_sup}]
This result follows from a combination of \cite[Theorem 2.2]{Zhang:2022} and its proof as well as Proposition \ref{prop_sub_cond}. 
Note that the paper \citep{Zhang:2022} assumes that $\varsigma_{j,n}$ and $\varsigma_{j,j+1,n}$ have limits as $n\to\infty$. This assumption, however, is not necessary for the weak convergence of the difference $D_{n,j}$ stated in \eqref{eqn:diffspikelim}, due to the normalization with $\gamma_{j,n}$. (See Remark 2 in~\citep{Zhang:2022}.) Likewise, the central limit theorem for an individual supercritical sample eigenvalue $\lambda_j(\bfSigmahat)$, $1 \leq j \leq K$, in \eqref{eqn:marginalspikelim} holds after normalization with $\varsigma_{j,n}$, even if $\varsigma_{j,n}$ is not guaranteed to have a limit.
\end{proof}

 If the alternative hypothesis $\mathsf{H}_{1,n}(K)$ holds, we denote by $r_{n,K+1}$ the rightmost endpoint of the support of the \MP distribution  $\underline{F}^{\frac{p-K}{n}, H_{n,K+1}}$, which is defined analogously to $\underline{F}^{y_n,H_n}$ as in \eqref{eq_mp_stieltjes}, but with $p/n$ replaced by $(p-K)/n$, and $H_n$ replaced by $H_{n,K+1}$ as defined in \eqref{eq_def_H_nK}. 
With these preparations, we are now in the position to prove Theorem \ref{thm_alternative}.

\begin{proof}[of Theorem \ref{thm_alternative}] We first prove the result when $K=1$, and then handle the case of $K>1$ later.
Define the statistic
    \begin{align} \label{eq_Z_n}
        Z_n = \frac{ \sqrt{n} }{ \varsigma_{1,n} } \lb \lambda_1(\bfSigmahat) - \psi_n(\lambda_1(\bfSigma)) \rb 
    - \frac{ \sqrt{n} }{ \varsigma_{1,n} } \lb \lambda_2(\bfSigmahat) - r_{n,2} \rb.
    \end{align}
    which allows $T_n$ to be represented as
\begin{align*}
    T_n 
    = 
    \frac{1}{\hat\sigma_n}\bigg(n^{1/6} \varsigma_{1,n} Z_n + n^{2/3}  \lb \psi_n(\lambda_1(\bfSigma)) - r_{n,2} \rb\bigg).
\end{align*}
It will be shown in Lemma~\ref{lem_misc_power} at the end of this subsection that $\varsigma_{1,n}\asymp 1$, $\hat\sigma_n=\mathcal{O}_{\P}(1)$, $Z_n \xrightarrow{\D} \mathcal{N}(0,1)$ as $n\to\infty$, as well as  $\psi_n(\lambda_1(\bfSigma)) - r_{n,2} \gtrsim 1$. 
Altogether, these facts imply that $\P(T_n> q_{1-\alpha})\to 1$ as $n\to\infty$, completing the proof in the case of $K=1$.

Turning to the case of $K>1$, observe that $T_n$ can be written as
\begin{align*}
    T_n = \frac{1}{\hat\sigma_n}\bigg(n^{1/6} \gamma_{1,n} D_{n,1} + n^{2/3}\left\{  \psi_n(\lambda_1(\bfSigma)) - \psi_n(\lambda_2(\bfSigma)) \right\}\bigg).
\end{align*}
Recall that Proposition \ref{thm_clt_sup} ensures $ D_{n,1} \cond \mathcal{N}(0,1)$ as $n\to\infty$, and that Lemma~\ref{lem_misc_power} ensures $\hat\sigma_n=\mathcal{O}_{\P}(1)$. Also, it is simple to check that $\gamma_{1,n}\lesssim 1$. Furthermore, arguments similar to those in the proof of Proposition~\ref{prop_sub_cond} can be used to show that as $n\to\infty$
$$\psi_n(\lambda_1(\bfSigma))-\psi_n(\lambda_2(\bfSigma)) \ \to \ \psi(\lambda_1(\bfSigma))-\psi(\lambda_2(\bfSigma))>0.$$
Taken together, these last several observations imply that $\P(T_n> q_{1-\alpha})\to 1$ as $n\to\infty$ when $K>1$.
\end{proof}

The following result is a consequence of \cite[Theorem 3.2]{ding2021spiked}. 
\begin{lemma} \label{thm_alt_sub}
    If the conditions of Theorem~\ref{thm_alternative} hold, then for every fixed $\eta>0$ and fixed $j>K$, we have
    \begin{align*}
        n^{2/3 - \eta} \lb \lambda_{j}(\bfSigmahat) -r_{n,K+1} \rb = o_{\PR}(1). 
    \end{align*}
\end{lemma}

\begin{lemma}\label{lem_misc_power}
If the conditions of Theorem~\ref{thm_alternative} hold, then
\begin{itemize}
    \item[(i)] $\varsigma_{1,n}\asymp 1$.
    \item[(ii)] $   \psi_n(\lambda_K(\bfSigma)) - r_{n,K+1} \gtrsim 1$.
    \item[(iii)] $ \hat\sigma_n=\mathcal{O}_{\P}(1)$.
    \item[(iv)] When $K=1$, we have $Z_n \xrightarrow{\D} \mathcal{N}(0,1)$ as $ n\to\infty$.
\end{itemize}

\end{lemma}

\begin{proof}[of Lemma \ref{lem_misc_power}]
(i) Note that $\varsigma_{1,n}^2$ can be expressed as
$$\varsigma_{1,n}^2=\lb \lambda_1(\bfSigma) \rb ^2\psi'(\lambda_1(\bfSigma))\bigg\{2+ (\E(x_{11}^4)-3)\psi'(\lambda_1(\bfSigma))\sum_{l=1}^p u_{1l}^4\bigg\}.$$

From the formula~\eqref{eq_def_psi_deriv}, we have $\psi'(\lambda_1(\bfSigma))<1$, where we recall that the conditions of Theorem~\ref{thm_alternative} treat $\lambda_1(\bfSigma)$ as fixed with respect to $n$. Also, since $(u_{j1},\dots,u_{jp})$ is a unit vector and $\E(x_{11}^4)\geq (\E(x_{11}^2))^2=1$, it follows that the factor in brackets above is positive and bounded away from 0. Together with $\psi'(\lambda_1(\bfSigma))>0$ from Proposition \ref{prop_sub_cond}, this gives $\varsigma_{1,n}^2\gtrsim 1$. On the other hand, it is clear that $\varsigma_{1,n}^2\lesssim 1$.

(ii). It is known from~\citep[][Lemmas 2.5-2.6]{knowles2017anisotropic} that $r_{n,K+1} = \psi_n(1/\xi_{n,K})$, and so applying the definition of $\psi_n$ gives
$$\psi_n(\lambda_K(\bfSigma))-r_{n,K+1}\\
=
\left(
\lambda_K(\bfSigma)-\frac{1}{\xi_{n,K}}
\right)
\left\{
1-\frac{p-K}{n}
\int
\frac{\lambda^2}
{(\lambda_K(\bfSigma)-\lambda)
 (1/\xi_{n,K}-\lambda)}
\,dH_{n,K+1}(\lambda)
\right\}.$$
Note also that the definition of $\xi_{n,K}$ is equivalent to
\begin{align*}
    1 = \frac{p -K }{n} \int \frac{\lambda^2}{\lb 1/ \xi_{n,K} - \lambda \rb^2} dH_{n,K+1}(\lambda).
\end{align*}
Substituting this into the previous equation gives
$$
\psi_n(\lambda_K(\bfSigma))-r_{n,K+1}\\
=
\frac{p-K}{n}
\left(
\lambda_K(\bfSigma)-\frac{1}{\xi_{n,K}}
\right)^2
\int
\frac{\lambda^2}
{(1/\xi_{n,K}-\lambda)^2
 (\lambda_K(\bfSigma)-\lambda)}
\,dH_{n,K+1}(\lambda).$$
By $\mathsf{H}_{1,n}(K)$, we have $
\lambda_K(\bfSigma)-1/\xi_{n,K}
\gtrsim 1 $, $\lambda_K(\bfSigma)-\lambda_{K+1}(\bfSigma)\asymp 1$, and $0<1/\xi_{n,K}-\lambda_{K+1}(\bfSigma)\lesssim 1$. Combining these bounds with \ref{ass_lsd}
implies
\begin{align*}
\int
\frac{\lambda^2}
{(1/\xi_{n,K}-\lambda)^2
 (\lambda_K(\bfSigma)-\lambda)}
\,dH_{n,K+1}(\lambda)
\gtrsim 1,
\end{align*}
which proves (ii).

(iii). Recall that 
$$\hat\sigma_n^3=\frac{1}{\hat\xi_n^3}\bigg(1+y_n\int\Big(\frac{\lambda\hat\xi_n}{1-\lambda\hat\xi_n}\Big)^3d\tilde H_n(\lambda)\bigg).$$
It follows from the definition of $\tilde H_n$ that the integral is non-negative and $\mathcal{O}_{\P}(1)$, and so it remains to show that $1/\hat\xi_n=\mathcal{O}_{\P}(1)$. Due to the elementary inequality $\frac{1}{a-b}\leq -\frac{1}{b}$ for any numbers $0\leq a< b$, it follows that regardless of which case holds in the definition~\eqref{eqn:xihatdef} of $\hat\xi_n$, we have
$$\hat\xi_n\geq \frac{n-1}{n}\frac{1}{\lambda_1(\hat\bfSigma)+1}.$$
Thus, the limit for $\lambda_1(\hat\bfSigma)$ in~\eqref{eqn:marginalspikelim} implies $1/\hat\xi_n=\mathcal{O}_{\P}(1)$, as needed.\\

(iv) This follows from~\eqref{eqn:marginalspikelim} in Proposition~\ref{thm_clt_sup}, Lemma~\ref{thm_alt_sub}, and $\varsigma_{1,n}\asymp 1$ from part (i).
\end{proof}

\subsection{Proof of Theorem~\ref{thm_est_spike}}\label{sec_est_spike}
The central theoretical result of this section is Corollary \ref{cor_H_1K}, which facilitates the proof of Theorem \ref{thm_est_spike}. Corollary \ref{cor_H_1K} addresses the asymptotic behavior of $T_{n,1},\dots,T_{n,K+1}$ under $\mathsf{H}_{1,n}(K)$, as well as properties of the estimates defined earlier in this section. The proof of Corollary \ref{cor_H_1K} relies on several additional auxiliary results, which are stated and proven first.

We begin this section by defining the scale estimator $\hat\sigma_{n,k}$, which requires the construction of additional estimators. 
For any $k\in \{1, \ldots, (p\wedge n) -1\}$, define
\begin{align*}
    \tilde{\underline{s}}_{n,k}(z)
    &:=
    \frac{1}{n}
    \sum_{j=k+1}^{n}
    \frac{1}
    {\lambda_j(\underline{\bfSigmahat})-z},
    \qquad
    z\in
    \C
    \setminus
    [
        \lambda_n(\underline{\bfSigmahat}),
        \lambda_{k+1}(\underline{\bfSigmahat})
    ].
\end{align*}
The estimator $\hat\xi_{n,k}$ is defined as
\begin{align*}
    \hat\xi_{n,k}
    &=
    \begin{cases}
    -\tilde{\underline{s}}_{n,k}
        (
            \lambda_{k}(\underline{\bfSigmahat})
        ),
    &
    \lambda_{k}(\underline{\bfSigmahat})
    >
    \lambda_{k+1}(\underline{\bfSigmahat}),
    \\[1mm]
    -\tilde{\underline{s}}_{n,k}
        (
            \lambda_{k}(\underline{\bfSigmahat})+1
        ),
    &
    \lambda_{k}(\underline{\bfSigmahat})
    =
    \lambda_{k+1}(\underline{\bfSigmahat}).
    \end{cases}
\end{align*}

Note that $\tilde\lambda_{1,Q} \geq \ldots \geq \tilde\lambda_{p,Q}$ by definition, and define for $j=k,\ldots,p$
\begin{align*}
    \tilde\lambda_{j,k}
    &=
    \tilde\lambda_{j,Q}
    \wedge
    \frac{1}
    {(1+\varepsilon)\hat\xi_{n,k}},
\end{align*}
and let
\begin{align*}
    \tilde H_{n,k}(t)
    &=
    \frac{1}{p-k+1}
    \sum_{j=k}^{p} 1\left\{
        \tilde\lambda_{j,k}\leq t
    \right\},
    \qquad t\in\R.
\end{align*}
Finally, the estimate $\hat\sigma_{n,k}$ is defined by
\begin{align}\label{eqn_hatsigmak}
    \hat\sigma_{n,k}^{3}
    &=
    \frac{1}{\hat\xi_{n,k}^{3}}
    \left\{
        1+
        \frac{p-k+1}{n}
        \int
        \left(
            \frac{\lambda\hat\xi_{n,k}}
                 {1-\lambda\hat\xi_{n,k}}
        \right)^{3}
        d\tilde H_{n,k}(\lambda)
    \right\}.
\end{align}

 Recall from
\eqref{eq_def_H_nK} that $H_{n,K+1}$ denotes the empirical distribution
of $\lambda_{K+1}(\bfSigma),\ldots,\lambda_p(\bfSigma)$. We define the
right edge associated with the subcritical population
eigenvalues by
\begin{align*}
    r_{n,K+1}
    &=
    \frac{1}{\xi_{n,K}}
    \left\{
        1+
        \frac{p-K}{n}
        \int
        \frac{\lambda\xi_{n,K}}
             {1-\lambda\xi_{n,K}}
        \,dH_{n,K+1}(\lambda)
    \right\}.
\end{align*}
The corresponding edge scaling parameter is defined by
\begin{align*}
    \sigma_{n,K+1}^{3}
    &:=
    \frac{1}{\xi_{n,K}^{3}}
    \left\{
        1+
        \frac{p-K}{n}
        \int
        \left(
            \frac{\lambda\xi_{n,K}}
                 {1-\lambda\xi_{n,K}}
        \right)^{3}
        dH_{n,K+1}(\lambda)
    \right\}.
\end{align*}

Put $ y_{n,K}
    =
    \frac{p-K}{n}$ and let
\begin{align*}
    \underline{s}_{n,K}^{0}
    =
    s_{\underline{F}_{y_{n,K},H_{n,K+1}}}
\end{align*}
denote the Stieltjes transform of the companion \MP distribution associated with the subcritical
population eigenvalues. 

Next, define $\bfSigma^\circ$ to be the $p\times p$ matrix with the same eigenvectors as $\bfSigma$, but whose eigenvalues are given by
\begin{align*}
    \lambda_j(\bfSigma^\circ)
    &=
    \begin{cases}
        \lambda_{K+1}(\bfSigma),
            & 1\leq j\leq K+1,\\
        \lambda_j(\bfSigma),
            & K+1< j\leq p.
    \end{cases}
\end{align*}
In other words, $\bfSigma^{\circ}$ is the same as $\bfSigma$, except that the leading $K+1$ eigenvalues are all equal to $\lambda_{K+1}(\bfSigma)$. We also define the triplet of positive parameters 
$$(\xi_n^\circ,r_n^\circ,\sigma_n^\circ)$$ 
in the same manner as $(\xi_{n,0},r_n,\sigma_n)$, but with the eigenvalues of $\bfSigma$ being replaced by those of $\bfSigma^{\circ}$. The following result shows that $(\xi_n^\circ,r_n^\circ,\sigma_n^\circ)$ is asymptotically close to $(\xi_{n,K},r_{n,K+1},\sigma_{n,K+1})$, and that all of these parameters are of order 1.

\begin{lemma} \label{lem:lem_reference_comparison}
If the Assumptions
\ref{ass_mp_regime}--\ref{ass_lsd} are satisfied and 
$\mathsf{H}_{1,n}(K)$ holds for all large $n$, then
\begin{equation}
\begin{split}
   1 &\asymp \xi_{n,K}
    =
     \xi_n^\circ+\mathcal O(n^{-1}),\\
   \quad 
    1&\asymp r_{n,K+1}
    =
    r_n^\circ+\mathcal O(n^{-1}),\\
  \quad 
    1&\asymp \sigma_{n,K+1}
    =
    \sigma_n^\circ+O(n^{-1}).
\label{eq_reference_comparison}
\end{split}
\end{equation}
\end{lemma}
\begin{proof}[of Lemma \ref{lem:lem_reference_comparison}] Regarding assertion that $(\xi_{n,K},r_{n,K+1},\sigma_{n,K+1})$ are all of order 1, the argument used to show $\xi_{n,0}\asymp 1$ in the proof of Proposition~\ref{thm_tw_law} can be applied in an analogous manner to show $\xi_{n,K}\asymp 1$. Based on this, the assertions $r_{n,K+1}\asymp 1$ and $\sigma_{n,K+1}\asymp 1$  follow directly from the definitions of these parameters and $\mathsf{H}_{1,n}(K)$.

To show that $\xi_{n,K}=\xi_n^{\circ}+\mathcal{O}(n^{-1})$, consider the function $g:[0,1/\lambda_1(\bfSigma^{\circ}))\to \R$ defined by
$$g_n(x)=\frac{1}{n} \sum_{j=1}^p \Big( \frac{\lambda_j(\bfSigma^\circ)x}{1 - \lambda_j(\bfSigma^\circ) x}\Big)^2.$$
Observe that the definitions of $\xi_n^{\circ}$ and $\xi_{n,K}$ imply $g_n(\xi_n^{\circ})=1$, and
\begin{equation}\label{eqn:xidiff}
g_n(\xi_{n,K})-g_n(\xi_n^{\circ})=   \frac{K}{n}
    \left(
        \frac{\lambda_{K+1}(\bfSigma)\xi_{n,K}}
             {1-\lambda_{K+1}(\bfSigma)\xi_{n,K}}
    \right)^2\lesssim \frac{1}{n},
\end{equation}
where $\mathsf{H}_{1,n}(K)$ has been used in the last step.
Since the middle expression is non-negative and $g_n$ is increasing, we must have $\xi_{n,K}\geq \xi_n^{\circ}$. Moreover, since $g_n$ is convex, we have
$$0\geq \xi_n^{\circ}-\xi_{n,K}\geq \frac{g_n(\xi_n^{\circ})-g_n(\xi_{n,K})}{g_n'(\xi_{n}^\circ)}.$$
It is straightforward to check that $g_n'(\xi_{n}^\circ)\geq (2/\xi_{n}^\circ)g_n(\xi_{n}^\circ)\gtrsim 1$, and so combining with~\eqref{eqn:xidiff} implies $\xi_{n,K}=\xi_n^{\circ}+\mathcal{O}(n^{-1})$. Next, the relation $r_{n,K+1}=r_n^{\circ}+\mathcal{O}(n^{-1})$ can be shown by observing that the definitions of $r_{n,K+1}$ and $r_n^{\circ}$ imply the following equation
\begin{align} \label{eq:diff_edge}
    r_n^\circ-r_{n,K+1}
    &=
    \left(
        \frac{1}{\xi_n^\circ}
        -
        \frac{1}{\xi_{n,K}}
    \right)
    +
    \frac{K}{n}
    \frac{\lambda_{K+1}(\bfSigma)}
         {(1-\lambda_{K+1}(\bfSigma)\xi_n^\circ)}
   +
    \frac{1}{n}\sum_{j=K+1}^p
    \left\{
        \frac{\lambda_j(\bfSigma)^2(\xi_n^{\circ}-\xi_{n,K})}
             {(1-\lambda_j(\bfSigma)\xi_n^\circ)(1-\lambda_j(\bfSigma)\xi_{n,K})}
    \right\},
\end{align}
and then applying $1\asymp \xi_{n,K}=\xi_n^{\circ}+\mathcal{O}(n^{-1})$ and $\mathsf{H}_{1,n}(K)$. Finally, similar reasoning can be used to show that $\sigma_{n,K+1}^3 - \lb \sigma_n^\circ \rb ^3 = \mathcal{O}(n^{-1}).$ Since $\sigma_{n,K+1} \asymp 1$, a first-order Taylor expansion of the cube root function implies $  \sigma_{n,K+1} = \sigma_n^\circ  + \mathcal{O}(n^{-1}).$
This concludes the proof.
\end{proof}
 
We denote the Stieltjes transform of the companion \MP distribution associated with the reference matrix $\bfSigma^\circ$ by 
\begin{align*}
    \su_n^\circ (z) = s_{\underline{F}_{y_n, F^{\bfSigma^\circ}}} (z). 
\end{align*}
Furthermore, let
\begin{align*}
    \hat N_{n,K}
    &:=
    \sum_{j=K+1}^{n}
    1 \left\{
        \lambda_j(\underline{\bfSigmahat})
        >
        \lambda_{K+1}(\underline{\bfSigmahat})-n^{-1/3}
    \right\}.
\end{align*}
For fixed constants $\tau,\tau'>0$, define the domain
\begin{align*}
    \mathbf{D}_{n,K}(\tau,\tau')
    & =
    \left\{
        z\in\C^{+}:
        \left|\operatorname{Re}(z)-r_{n,K+1}\right|
        \leq\tau',
        \quad
        n^{-1+\tau}\leq\operatorname{Im}(z)\leq 1
    \right\}.
\end{align*}

The following lemma collects some necessary random matrix theory results.

\begin{lemma}\label{lem_aux_ext_xi}
Suppose that Assumptions
\ref{ass_mp_regime}--\ref{ass_lsd} are satisfied and that
$\mathsf{H}_{1,n}(K)$ holds for all large $n$. Then, the
following statements hold.

\begin{enumerate}

\item For every fixed $d\in\N$, we have
\begin{align}
    \frac{n^{2/3}}{\sigma_{n,K+1}}
    \lb
        \lambda_{K+j}(\bfSigmahat)-r_{n,K+1}
    \rb_{1\leq j\leq d}
    \cond
    (\zeta_1,\ldots,\zeta_d)
    \label{eqn:jointTWlimit:general}
\end{align}
as $n\to\infty$, where $(\zeta_1,\ldots,\zeta_d)$ has a $d$-dimensional
Tracy--Widom distribution of type 1.

\item It holds that
\begin{align}
    \hat N_{n,K}
    =
    \mathcal O_{\PR}\lb n^{1/2}\rb.
    \label{eqn:eigenvalue_count:general}
\end{align}

\item For every fixed $\tau\in(0,1)$, there exists a constant
$\tau'>0$ such that as $n\to\infty$,
\begin{align}
        \sup_{z\in\mathbf D_{n,K}(\tau,\tau')}
        \left|
            \tilde{\underline{s}}_{n,K+1}(z)-\underline{s}_{n}^{\circ}(z)
        \right|
    =
    o_{\P}(1).
    \label{eqn:local_law:general}
\end{align}
\end{enumerate}
\end{lemma}

\begin{proof}[of Lemma \ref{lem_aux_ext_xi}] 
The sample covariance matrix corresponding to the non-spiked reference matrix is defined as
\begin{align*}
    \bfSigmahat^\circ
    &=
    \frac{1}{n}
    (\bfSigma^\circ)^{1/2}
    \bfX^\top \bfX
    (\bfSigma^\circ)^{1/2}.
\end{align*}
Due to $\mathsf H_{1,n}(K)$ and \eqref{eq_reference_comparison}, $\bfSigma^\circ$ satisfies the regular edge condition \eqref{def_regular_edge}.  
 Then, 
the eigenvalue sticking result of \cite[equation (3.14) and
Corollary 3.19]{ding:yang:2021} implies that for every fixed $d\in \N$ 
\begin{align} \label{eq_ev_sticking}
    \max_{1\leq j\leq d}
    \left|
        \lambda_{K+j}(\bfSigmahat)-\lambda_j(\bfSigmahat^\circ)
    \right|
    &
    =
    o_{\PR}(n^{-2/3}).
\end{align}
Note that under $\mathsf{H}_{1,n}(K)$, the eigenvalue $\lambda_{K+1}(\bfSigma)$ is sub-critical, and so the limit~\eqref{eqn:jointTWlimit} implies a $d$-dimensional Tracy-Widom limit for $(\lambda_1(\bfSigmahat^\circ),\ldots,\lambda_d(\bfSigmahat^\circ))$. Combining this with $1\asymp\sigma_{n,K+1}=\sigma_n^{\circ}+\mathcal{O}(n^{-1})$ from
Lemma~\ref{lem:lem_reference_comparison} proves the first assertion.

For the second assertion \eqref{eqn:eigenvalue_count:general}, define the companion matrix associated with
$\bfSigmahat^\circ$ by
\begin{align*}
    \underline{\bfSigmahat}^{\circ}
    &:=
    \frac{1}{n}
    \bfX\bfSigma^\circ\bfX^\top,
\end{align*}
and let
\begin{align*}
    \hat N_n^\circ
    &:=
    \sum_{j=1}^{n} 1\left\{
        \lambda_j(\underline{\bfSigmahat}^{\circ})
        >
        \lambda_1(\underline{\bfSigmahat}^{\circ})
        -
        2 n^{-1/3}
    \right\} .
\end{align*}
The arguments used in the proof
of Proposition \ref{prop_order_hat_N_n} apply to $\underline{\bfSigmahat}^{\circ}$, yielding 
\begin{align} \label{eq_order_N_hat_circ}
    \hat N_n^\circ
    &=
    \mathcal O_{\PR}(n^{1/2}).
\end{align}
Moreover,
\begin{align} \label{eq_rank_ineq}
    \operatorname{rank}
    (
        \underline{\bfSigmahat}
        -
        \underline{\bfSigmahat}^{\circ}
    )
    =\operatorname{rank}(\textstyle\frac{1}{n}\bfX(\bfSigma-\bfSigma^{\circ})\bfX\ttop)&\leq
    \operatorname{rank}
    \lb
        \bfSigma-\bfSigma^\circ
    \rb
    \leq K.
\end{align}
By \eqref{eq_ev_sticking}, the event
\begin{align*}
    \lambda_{K+1}(\underline{\bfSigmahat})
    -
    n^{-1/3}
    &\geq
    \lambda_1(\underline{\bfSigmahat}^{\circ})
    -
    2n^{-1/3}
\end{align*}
has probability approaching one. 
On this event, we have 
\begin{align*}
    \hat N_{n,K} - \hat N_n^\circ
    &\leq 
     \sum_{j=1}^{n} 1\left\{
\lambda_j(\underline{\bfSigmahat})
        >
\lambda_1(\underline{\bfSigmahat}^{\circ})
        -
        2 n^{-1/3}
    \right\} 
    - \hat N_n^\circ\\
    &\leq \sup_{t\in \R}n|F^{\underline{\bfSigmahat}}(t)-F^{\underline{\bfSigmahat}^{\circ}}(t)|\\
    &\leq \operatorname{rank} (\underline{\bfSigmahat} - \underline{\bfSigmahat}^{\circ}) \\
    &\leq K,
\end{align*}
where~\citep[][Theorem A.43]{bai2004} has been used in the third step.
Combining the last bound with \eqref{eq_order_N_hat_circ} proves the second assertion.

For the third assertion \eqref{eqn:local_law:general}, an analogous version of Lemma \ref{thm_local_law} holds when $\hat\bfSigma$ is replaced with $\underline{\bfSigmahat}^{\circ}$. Also, note that since $r_n^\circ-r_{n,K+1}=\mathcal O(n^{-1})$ from \eqref{eq_reference_comparison}, we may choose
$\tau'>0$ sufficiently small such that, for all  large
$n$, the domain $\mathbf D_{n,K}(\tau,\tau')$ is contained in the domain for which Lemma \ref{thm_local_law} can be adapted to
$\underline{\bfSigmahat}^\circ$. Consequently, there exists a constant $\tau'>0$ not depending on $n$ such that
\begin{align} \label{n0}
    \sup_{z\in\mathbf{D}_{n,K}(\tau,\tau')}
    \left|
        s_{F^{\underline{\bfSigmahat}^{\circ}}}(z)
        -
        \underline{s}_n^\circ(z)
    \right|
    &=
    o_{\PR}(1).
\end{align}

By \eqref{eq_rank_ineq} and a standard perturbation argument for resolvents, the following bound holds for any $z\in\C^+$,
\begin{align} \label{n1}
    \left|
        \underline{s}_{F^{\underline{\bfSigmahat}}}(z)
        -
        s_{F^{\underline{\bfSigmahat}^{\circ}}}(z)
    \right|
    \leq \frac{\operatorname{rank}
    (
        \underline{\bfSigmahat}
        -
        \underline{\bfSigmahat}^{\circ}
    )}{n \operatorname{Im}(z)}
    &\leq
    \frac{K}{n\operatorname{Im}(z)}.
\end{align}
Furthermore, the definition of $\tilde{\underline{s}}_{n,K+1}$ ensures that for any $z\in\C^+$, we have
\begin{align} \label{n2}
    \left|
        \tilde{\underline{s}}_{n,K+1}(z)
        -
        s_{F^{\underline{\bfSigmahat}}}(z)
    \right|
    &=
    \left|
        \frac{1}{n}
        \sum_{j=1}^{K+1}
        \frac{1}
        {\lambda_j(\underline{\bfSigmahat})-z}
    \right|
    \leq
    \frac{K+1}{n\operatorname{Im}(z)}.
\end{align}
Since $\operatorname{Im}(z)\geq n^{-1+\tau}$ holds for
$z\in\mathbf{D}_{n,K}(\tau,\tau')$, combining \eqref{n1} and \eqref{n2} leads to
\begin{align} \label{n4}
    \sup_{z\in\mathbf{D}_{n,K}(\tau,\tau')}
    \left|
        \tilde{\underline{s}}_{n,K+1}(z)
        -
        s_{F^{\underline{\bfSigmahat}^{\circ}}}(z)
    \right|
    &=
    o_{\PR}(1).
\end{align}
Combining the bounds \eqref{n0} and \eqref{n4}   completes the proof.
\end{proof}

\begin{corollary}\label{cor_H_1K}
Suppose that Assumptions
\ref{ass_mp_regime}-\ref{ass_lsd} are satisfied and that
$\mathsf H_{1,n}(K)$ holds for all  large $n$. Then, the
following statements hold.

\begin{enumerate}

\item As $n\to\infty$, 
\begin{align}
    \hat\xi_{n,K+1}-\xi_{n,K+j}
    &\xrightarrow{\P}0,
    \qquad j\in\{0,1\},
    \tag{a}
    \label{eqn:hatxithm:K}\\
    \tilde H_{n,K+1}
    &\cond H
    \quad\textnormal{in probability},
    \tag{b}
    \label{eqn:hatHthm:K}\\
    \frac{\hat\sigma_{n,K+1}}
         {\sigma_{n,K+1}}
    &\xrightarrow{\P}1.
    \tag{c}
    \label{eqn:hatsigmathm:K}
\end{align}

\item As $n\to\infty$,
\begin{align*}
    T_{n,K+1}
    &\cond
    \zeta-\zeta',
\end{align*}
where $(\zeta,\zeta')$ has a two-dimensional Tracy--Widom
distribution of type 1.

\item Suppose additionally that $\lambda_1(\bfSigma)>\dots >\lambda_K(\bfSigma)$ are fixed with respect to $n$.
Then, for any fixed $1 \leq k \leq K$, the following limit holds as $n\to\infty$
\begin{align*}
    T_{n,k} \conp \infty.
\end{align*}

\end{enumerate}
\end{corollary}

\begin{proof}[of Corollary \ref{cor_H_1K}] Part 1:
By \eqref{eqn:jointTWlimit:general} in Lemma \ref{lem_aux_ext_xi}, we have
\begin{align*}
    \PR\left(
        \lambda_{K+1}(\underline{\bfSigmahat})
        >
        \lambda_{K+2}(\underline{\bfSigmahat})
    \right)
    &=
    1-o(1).
\end{align*}
Moreover, the first two assertions of
Lemma~\ref{lem_aux_ext_xi} imply
\begin{align*}
    n^{-4/3}
    &=
    \mathcal O_{\PR}
    \left(
        \left\{
            \lambda_{K+1}(\underline{\bfSigmahat})
            -
            \lambda_{K+2}(\underline{\bfSigmahat})
        \right\}^{2}
    \right)
\end{align*}
and $
    \hat N_{n,K}
    =
    \mathcal O_{\PR}\lb n^{1/2}\rb. $
Using these bounds, arguments similar to those in the proof of
Lemma \ref{lem1_xi}\eqref{eqn:hatxlim} with
$\lambda_1(\underline{\bfSigmahat})$ replaced by
$\lambda_{K+1}(\underline{\bfSigmahat})$ show there is a fixed
$\delta\in(0,1/6)$ such that the following limit holds with $\eta_n=n^{-1+\delta}$,
\begin{align*}
\hat{\xi}_{n,K+1}
+
\tilde{\underline{s}}_{n,K+1}
\bigl(\lambda_{K+1}(\underline{\bfSigmahat})+i\eta_n\bigr)
=o_{\PR}(1).
\end{align*}
Combining this with the local law \eqref{eqn:local_law:general} yields
\begin{align}
    \hat\xi_{n,K+1}
    +
    \underline{s}_{n}^{\circ}
    \lb
        \lambda_{K+1}(\underline{\bfSigmahat})
        +
        i\eta_n
    \rb
    &=
    o_{\PR}(1).
    \label{eqn:hat_xi_stieltjes_K}
\end{align}

Furthermore, \eqref{eqn:jointTWlimit:general} and
\eqref{eq_reference_comparison} imply that 
\begin{align} \label{eq:r1}
    \lambda_{K+1}(\underline{\bfSigmahat})
    -
    r_n^\circ
    &=
    \mathcal O_{\PR}\lb n^{-2/3}\rb.
\end{align}
The argument used in the proof of Lemma \ref{lem1_xi}\eqref{eqn:xilim}, with
$(r_n,\xi_{n,0},\underline{s}_n^0)$ replaced by
$(r_{n}^\circ,\xi_{n}^\circ,\underline{s}_{n}^\circ)$ and using \eqref{eq:r1},
gives
\begin{align}
    \xi_{n}^\circ
    +
    \underline{s}_{n}^{\circ}
    \lb
        \lambda_{K+1}(\underline{\bfSigmahat})
        +
        i\eta_n
    \rb
    &=
    o_{\PR}(1).
    \label{eqn:xi_stieltjes_K}
\end{align}
Combining \eqref{eqn:hat_xi_stieltjes_K} and
\eqref{eqn:xi_stieltjes_K} yields
\begin{align*}
    \hat\xi_{n,K+1}-\xi_{n}^\circ
    &=
    o_{\PR}(1),
\end{align*}
and by \eqref{eq_reference_comparison}, this implies 
\begin{align*}
    \hat\xi_{n,K+1} - \xi_{n,K} = \op.
\end{align*}
Furthermore, the same convexity argument used in the proof of
Lemma \ref{lem_xin_xi1} gives
\begin{align*}
    0
    <
    \xi_{n,K+1}-\xi_{n,K}
    &\lesssim
    p^{-1}.
\end{align*}
It follows that
\begin{align*}
    \hat\xi_{n,K+1}-\xi_{n,K+j}
    &=
    o_{\PR}(1),
    \qquad j\in\{0,1\},
\end{align*}
which proves \eqref{eqn:hatxithm:K}.
The proofs of \eqref{eqn:hatHthm:K} and
\eqref{eqn:hatsigmathm:K} are similar to the proofs of
Proposition \ref{thm_consistency_xi}\eqref{eqn:hatHthm} and
\eqref{eqn:hatsigmathm}, respectively, and are thus omitted
for the sake of brevity.

Part 2: The limit $T_{n,K+1}\xrightarrow{\mathcal{D}}\zeta-\zeta'$ as $n\to\infty$ is a consequence of \eqref{eqn:jointTWlimit:general}  and  part 1 \eqref{eqn:hatsigmathm:K} above.

Part 3: The argument for showing $T_{n,k}\xrightarrow{\P}\infty$ as $n\to\infty$ is similar to the proof of Theorem \ref{thm_alternative}, after replacing
($\lambda_1(\bfSigmahat),\lambda_2(\bfSigmahat),\hat\sigma_n)$ with $(\lambda_{k}(\bfSigmahat),\lambda_{k+1}(\bfSigmahat),\hat\sigma_{n,k})$, and so the details are omitted.
\end{proof}

We conclude this section with the proof of Theorem \ref{thm_est_spike}.
\begin{proof}[of Theorem \ref{thm_est_spike}]
By part 3 of Corollary \ref{cor_H_1K},  since $q_{1-\alpha}$ is fixed,
\begin{align*}
    \P\big(T_{n,k}\leq q_{1-\alpha}\big)\to 0,
    \qquad k=1,\ldots,K.
\end{align*}
Since $K$ is fixed, it follows from the definition of $\hat K_n$ that as $n\to\infty$,
\begin{align*}
    \P(\hat K_n<K)
    &\leq
    \sum_{k=1}^{K}
    \P\big(T_{n,k}\leq q_{1-\alpha}\big)
    \to 0.
\end{align*}
Also, part 2 of Corollary \ref{cor_H_1K} gives
\begin{align} \label{n5}
    \P\big(T_{n,K+1}\leq q_{1-\alpha}\big)
    \to 1-\alpha
\end{align}
as $n\to\infty$. Moreover, the inclusions 
$$\{\hat K_n=K\}\subset \{T_{n,K+1}\leq q_{1-\alpha}\}\subset \{\hat K_n\leq K\}$$
imply
\begin{align} \label{n6}
    \left|
    \P(\hat K_n=K)
    -
    \P\big(T_{n,K+1}\leq q_{1-\alpha}\big)
    \right|
    &\leq
    \P(\hat K_n<K)
    \to 0
\end{align}
as $n\to\infty$. Combining \eqref{n5} and \eqref{n6} completes the proof.
\end{proof}

\section{Proofs for Section \ref{sec_stat_appl}}\label{sec_proof_boot}
For the proof of Theorem \ref{thm_bootstrap_cont_map}, we need two preparatory lemmas. If $G$ is compactly supported probability measure on $[0,\infty)$ and $c_G>0$ denotes the rightmost endpoint of its support, then for any $t\in (0,1/c_G)$, we define
  \begin{align}\label{eqn:Ifunction}
 	I (t,G) = \int \lb \frac{\lambda t}{1 - \lambda t} \rb^2 d G(\lambda).
 \end{align}

\begin{lemma}\label{lem:Ifunction} In the notation above, if $a$ and $b$ are two numbers in the interval $(0,1/c_G)$, then
$$|I(b,G)-I(a,G)|\geq \frac{I(b,G)}{b^2}(b-a)^2.$$
\end{lemma}
\begin{proof}
    First consider the case that $a\leq b$. Since $I(\cdot, G)$ is increasing, we have
    \begin{equation}
        \begin{split}
            |I(b,G)-I(a,G)|&=I(b,G)-I(a,G)\\
            &=\int \frac{\lambda^2 (b-a)(b+a-2\lambda ab)}{(1-a\lambda)^2(1-b\lambda)^2}dG(\lambda)\\
            &\geq \frac{(b-a)^2}{b^2}\int \frac{(b\lambda)^2}{(1-b\lambda)^2}dG(\lambda)\\
            &=\frac{(b-a)^2}{b^2} I(b,G),
        \end{split}
    \end{equation}
    where in the third step we multiplied and divided by $b^2$ while also using that the inequalities $(1-a\lambda)\leq 1$ and $b+a-2\lambda ab\geq b-a$ hold for all $\lambda$ in the support of $G$.
    On the other hand, if $b<a$, then the difference $I(a,G)-I(b,G)$ can be lower bounded in a similar manner, but instead using $(b+a-2\lambda ab)\geq a-b$.
\end{proof}

\begin{lemma}\label{lem_tilde_xi}
    If the conditions of Theorem~\ref{thm_bootstrap_cont_map} hold, then as $n\to\infty$
   \begin{equation}
       \tilde\xi_{n,0}/\xi_{n,0}=1+o_{\P}(1).
   \end{equation}
\end{lemma}
\begin{proof} Let $\Omega_n=\{\tilde\lambda_1\xi_{n,0}\leq 1/(1+\varepsilon/2)\}$ and note that $\P(\Omega_n)=1-o(1)$ because of Proposition~\ref{thm_consistency_xi}(a) and because $\mathsf{H}_{0,n}$ holds for all large $n$. Recalling the function $I(\cdot,\cdot)$ introduced in~\eqref{eqn:Ifunction}, note that $I(\tilde\xi_{n,0},\tilde H_n)=1/y_n$ holds by the definition of $\tilde\xi_{n,0}$, and that the expression $I(\xi_{n,0},\tilde H_n)$ is well defined when the event $\Omega_n$ holds. So, letting $b=\tilde\xi_{n,0}$ and $a=\xi_{n,0}$ in Lemma~\ref{lem:Ifunction} gives
\begin{equation}
    \begin{split}
        \Big(\frac{\tilde\xi_{n,0}}{\xi_{n,0}}-1\Big)^2 
        &=\Big(\frac{\tilde\xi_{n,0}}{\xi_{n,0}}-1\Big)^2 1\{\Omega_n\}+o_{\P}(1)\\
        &\leq y_n \frac{\tilde\xi_{n,0}^2}{\xi_{n,0}^2} \big|1/y_n-I(\xi_{n,0},\tilde H_n)\big|1\{\Omega_n\} + o_{\P}(1).
    \end{split}
\end{equation}

Using $\tilde\xi_{n,0}<1/\tilde\lambda_1$, Proposition~\ref{thm_tw_law}\eqref{eq_xi_order_1} and Proposition~\ref{thm_consistency_xi}\eqref{eqn:hatHthm} ensure that $\frac{\tilde\xi_{n,0}^2}{\xi_{n,0}^2} =\mathcal{O}_{\P}(1).$
Thus, it suffices to show that $\big|1/y_n-I(\xi_{n,0},\tilde H_n)\big|1\{\Omega_n\}=o_{\P}(1)$. To this end, observe that $I(\xi_{n,0},H_n)=1/y_n$ holds by the definition of $\xi_{n,0}$, and so
\begin{equation}
\begin{split}
|1/y_n-I(\xi_{n,0},\tilde H_n)\big|1\{\Omega_n\}
&=
     \big|I (\xi_{n,0}, H_n )-I( \xi_{n,0}, \tilde H_n)\big|1\{\Omega_n\}\\ 
& \ \leq \
     \frac{1}{p}\sum_{j=1}^p \bigg|\Big(\ts\frac{\lambda_j(\bfSigma) \xi_{n,0}}{1-\lambda_j(\bfSigma) \xi_{n,0}}\Big)^2-\Big(\frac{\tilde\lambda_j\xi_{n,0}}{1-\tilde\lambda_j\xi_{n,0}}\Big)^2\bigg|1\{\Omega_n\}\\
     &\lesssim \ \frac{\xi_{n,0}}{p}\sum_{j=1}^p |\tilde\lambda_j-\lambda_j(\bfSigma)|\ \\
     &= \ o_{\P}(1),
 \end{split}
\end{equation}
where the third step uses the fact that the function $x\mapsto (x/(1-x))^2$ is $\mathcal{O}(1)$-Lipschitz on $[0,1/(1+\varepsilon/2)]$, and the fourth step uses Proposition~\ref{thm_consistency_xi}\eqref{eqn:hatxithm}, Proposition~\ref{thm_tw_law}\eqref{eq_xi_order_1}  and~\citep[Corollary A.1]{ledoit2015spectrum}. 
\end{proof}

\begin{proof}[of Theorem \ref{thm_bootstrap_cont_map}]
Let the function $g$ be as in the definition of $V_n$ and $V_n^{\star}$, and let $(\zeta_1,\dots,\zeta_k)$ have a $k$-dimensional Tracy-Widom distribution of type 1. Also, recall that Proposition~\ref{thm_tw_law} shows that the regular edge~\eqref{def_regular_edge} condition defined in Section~\ref{sec_proof_thm_conv_T_null} holds under the conditions of Theorem~\ref{thm_bootstrap_cont_map}. It is known from~\cite[Corollary 3.19]{knowles2017anisotropic} that the limit 
\begin{equation}
    \lim_{n\to\infty}d(\mathcal{L}(V_n),\mathcal{L}(g(\zeta_1,\dots,\zeta_k))= 0
\end{equation}
is implied by~\eqref{def_regular_edge}, as well as assumptions \ref{ass_mp_regime}, \ref{ass_mom}, and the following weaker version of assumption \ref{ass_lsd}:
    \begin{enumerate}[label={A3'}, ref=A3']
        \item \label{label:A3prime} As $n\to\infty$, the distribution $H_n$ has a weak limit $H$, and the support of $H$ is a \smash{finite} union of compact intervals in $(0,\infty)$.
    \end{enumerate}
If we can show that the ``bootstrap-world'' counterparts of~\eqref{def_regular_edge},~\ref{ass_mp_regime},~\ref{ass_mom}, and \eqref{label:A3prime} hold with probability tending to 1 as $n\to\infty$, then it will follow from a standard subsequence argument that
\begin{equation}
    d(\mathcal{L}(V_n^{\star}|\mathbf{Y}),\mathcal{L}(g(\zeta_1,\dots,\zeta_k))\xrightarrow{\P}0
\end{equation}    
as $n\to\infty$, and the proof will be complete.

It is clear that~\ref{ass_mp_regime} and \ref{ass_mom} always hold in the bootstrap world. With regard to A3', Proposition \ref{thm_consistency_xi}\eqref{eqn:hatHthm} ensures that $\tilde H_n$ converges weakly in probability to $H$, whose support is a finite union of compact intervals. 

Turning to \eqref{def_regular_edge}, let $[\tilde l_n,\tilde r_n]$ denote the rightmost interval in $\textup{supp}(\underline{F}^{y_n,\tilde H_n})$, which consists of a finite union of compact intervals by \citep[Lemma 2.6]{knowles2017anisotropic}. To show that \eqref{def_regular_edge} holds with probability tending to 1, it is enough to verify that there is a constant $\eta$ fixed with respect to $n$ such that the following two inequalities hold with probability tending to 1,
\begin{align}
    \tilde r_n - \tilde l_n &\geq \eta,\label{def_regular_edge_boot1}\\
    \tilde \lambda_1(\tilde \xi_{n,0}+\eta)&\leq 1.\label{def_regular_edge_boot2}
\end{align}
By Lemma \ref{lem_tilde_xi}, Proposition~\ref{thm_consistency_xi}\eqref{eqn:hatxithm}, and Proposition~\ref{thm_tw_law}\eqref{eq_xi_order_1}, we know that $\tilde\xi_{n,0}/\xi_{n,0}=1+o_{\P}(1)$, and that the inequality $\tilde\lambda_1\xi_{n,0}\leq 1/(1+\varepsilon/2)$ holds with probability tending to 1. It follows that there is a choice of $\eta>0$ such that \eqref{def_regular_edge_boot2} holds with probability tending to 1. Finally, since we know that~\eqref{def_regular_edge_boot2} and the bootstrap-world counterparts of \ref{ass_mp_regime} and A3' hold with probability tending to 1, the argument used to establish \cite[eqn.~(2.19)]{bao2013local} can be applied to show that~\eqref{def_regular_edge_boot1} also holds with probability tending to 1.
\end{proof}

\section{Additional numerical results} \label{sec_supp_sim}
\rev{Below, Section~\ref{sec_sim_bootstrap} investigates the performance of the proposed bootstrap method for distributional approximation and bias estimation. Section~\ref{sec_stock} applies the proposed test to stock market data. Sections~\ref{sec_epsilon_hat} and~\ref{sec_sim_epsilon} respectively examine the finite-sample behavior of $\hat\varepsilon$ and the sensitivity of the test to the choice of $\varepsilon$. Lastly, Section~\ref{sec_robustness} presents robustness checks.}

\subsection{Bootstrap}  \label{sec_sim_bootstrap} 

This subsection illustrates the performance of the proposed bootstrap method in two tasks: distributional approximation and bias estimation.\\

\noindent \emph{Distributional approximation.} Here, we look at how well the bootstrap can approximate the distributions of the statistics
\begin{align}
    L_n &=\frac{n^{2/3}}{\sigma_n} \big( \lambda_1(\bfSigmahat) - r_n \big)\\
    G_n &= \frac{n^{2/3}}{\sigma_n} \big( \lambda_1(\bfSigmahat) - \lambda_2(\bfSigmahat) \big),
\end{align}
with respect to their means, standard deviations, and 0.95-quantiles.

Synthetic $n\times p$ data matrices of the form $\mathbf{Y}=\mathbf{X}\bfSigma^{1/2}$ were generated with $(n,p)=(500,300)$, and $\mathbf{X}$ having i.i.d.~entries $x_{11}\sim \mathcal{N}(0,1)$ or $x_{11}\sim t_{10}/\sqrt{\var(t_{10})}$. The matrix $\bfSigma$ was chosen by considering a spectral decomposition $\bfSigma=\mathbf{Q}\mathbf{\Lambda}\mathbf{Q}^{\top}$, where $\mathbf{Q}$ was drawn uniformly at random from the set of $p\times p$ orthogonal matrices, and the matrix of eigenvalues $\mathbf{\Lambda}$ was taken to have the following three forms
\begin{enumerate}[label=(\alph*)]
    \item \label{sigma_identity}
     $\mathbf{\Lambda} = \bfI$,
    \item \label{sigma_two_spikes} 
    $\mathbf{\Lambda} = \operatorname{diag}(1.4,1.2,1, \ldots, 1),$
     \item \label{sigma_five_spikes} 
    $\mathbf{\Lambda} = \operatorname{diag}(1.3, 1.3, 1.3, 1.3, 1.3 ,1, \ldots, 1)$.
\end{enumerate}
For each choice of $\bfSigma$ and the distribution of $x_{11}$, we generated 30000 realizations of $\mathbf{Y}$, and computed the corresponding values of $L_n$ and $G_n$. The empirical means, standard deviations, and 0.95-quantiles associated with these samples were then treated as ground truth for the distributions of $L_n$ and $G_n$. For the first 600 realizations of $\mathbf{Y}$, we applied Algorithm~\ref{alg_bootstrap} to generate $B=500$ bootstrap samples of the form $L_n^{\star}= \frac{n^{2/3}}{\tilde\sigma_n} \big( \lambda_1 (\bfSigmahat^\star) - \tilde r_n \big)$ 
and $G_n^{\star}=\frac{n^{2/3}}{\tilde\sigma_n} \big( \lambda_1 (\bfSigmahat^\star) - \lambda_2 (\bfSigmahat^\star) \big)$. The empirical mean, empirical standard deviation, and empirical 0.95-quantile of each set of 500 bootstrap samples were then treated as bootstrap estimates of the corresponding quantities for $L_n$ and $G_n$. So altogether, this produced 600 realizations of each type of bootstrap estimate.

The results for $L_n$ and $G_n$ are shown in Tables~\ref{table1_bootstrap_ev} and~\ref{table1_bootstrap_gap} respectively. In each table, there are two rows associated with each model configuration, with the upper row displaying the ground truth values, and the lower row summarizing the bootstrap estimates. Specifically, the lower rows contain the means of the bootstrap estimates (with standard deviation in parentheses) for $\E(L_n)$, $\E(G_n)$, $\sqrt{\var(L_n)}$, and $\sqrt{\var(G_n)}$, as well as the coverage probabilities $\P(L_n\leq \hat q_{0.95})$ and $\P(G_n\leq \hat q_{0.95})$, where $\hat q_{0.95}$ denotes the corresponding bootstrap quantile estimate.
\begin{table}
    \centering
    \begin{tabular}{llllc}
         \multicolumn{5}{c}{
         }  \\[0.2cm]
       $x_{11}$  & $\bfSigma$  & $\E(L_n)$ & $\sqrt{\textup{var}(L_n)}$ & $\P(L_n\leq \hat q_{0.95})$ \\[0.cm] 
       \hline\\[-0.3cm] 
       $\mathcal{N}(0,1)$ & \ref{sigma_identity}  & -1.26 & 1.23 & - \\
        & & -1.31 (0.06) & 1.25 (0.05) & 0.95 \\
         & \ref{sigma_two_spikes} & -1.29 & 1.24 & - \\ 
        & & -1.31 (0.06) & 1.25 (0.05) & 0.95 \\ 
           & \ref{sigma_five_spikes}  & -1.29 & 1.32 & - \\ 
        & & -1.31 (0.05) & 1.25 (0.05) & 0.93\\[0.cm] 
        \hline\\[-0.3cm]
        
       $t_{10}/\sqrt{\var}(t_{10})$ & \ref{sigma_identity}  & -1.12 & 1.31 & - \\
        & & -1.31 (0.06) & 1.24 (0.05) & 0.93 \\
        & \ref{sigma_two_spikes}  & 
       -1.08 & 1.31 & - \\
        & & -1.31 (0.06) & 1.24 (0.05) & 0.92 \\
        & \ref{sigma_five_spikes} & 
       -1.07 & 1.28 & - \\
       & & -1.32 (0.06) & 1.24 (0.06) & 0.93\\[0.2cm]
       
    \end{tabular}
    \caption{
    Performance of bootstrap estimates for the mean, standard deviation and 0.95-quantile of $L_n$.
    }
    \label{table1_bootstrap_ev}
\end{table}

Overall, the bootstrap estimates are accurate for all three estimation tasks. Indeed, in the estimation of $\E(L_n)$, $\E(G_n)$, $\sqrt{\var(L_n)}$, and $\sqrt{\var(G_n)}$, both the bias and standard deviation of the bootstrap estimates are a small percentage of the target parameter value. 
Furthermore, with regard to quantile estimation, the coverage probabilities are typically within 2\% of the desired nominal coverage of 95\%.\\

\noindent \emph{Bias estimation.} Due to the fact that the bias of $\lambda_1(\hat\bfSigma)$ is a well-known source of difficulty in high-dimensional inference, it is natural to ask if the bootstrap can be used to estimate this bias. That is, we seek to estimate the parameter $\E(\lambda_1(\hat\bfSigma))-\lambda_1(\bfSigma)$. In Table~\ref{table_bootstrap_bias2}, we present results for this task under conditions similar to those used earlier in this subsection. The only differences were that we put $(n,p)=(500,600)$ and took the matrix $\mathbf{\Lambda}$ to be of the form $\mathbf{\Lambda}=\textup{diag}(\lambda_1(\bfSigma),\lambda_2(\bfSigma),1,\ldots,1)$ using several choices of $(\lambda_1(\bfSigma),\lambda_2(\bfSigma))$ that are given in the second column of Table~\ref{table_bootstrap_bias2}. The results are presented in the same format as in Tables~\ref{table1_bootstrap_ev} and~\ref{table1_bootstrap_gap}, demonstrating that the errors of the bootstrap estimates are typically only a few percent of the true parameter value.

\begin{table}
    \centering
    \begin{tabular}{llllc}
         \multicolumn{5}{c}{ 
 }  \\[0.2cm]
       $x_{11}$ & $\bfSigma$   & $\E(G_n)$ & $\sqrt{\textup{var}(G_n)}$ & $\P(G_n\leq \hat q_{0.95})$ \\[0.cm] 
       \hline\\[-0.3cm] 
       $\mathcal{N}(0,1)$ & \ref{sigma_identity}  & 2.03 & 1.10  & - \\
       & & 2.02 (0.06) & 1.11 (0.04) & 0.95 \\ 
       & \ref{sigma_two_spikes}  & 2.04 & 1.16  & - \\
       & & 2.02 (0.06) & 1.11 (0.04) & 0.94 \\
        & \ref{sigma_five_spikes}  & 2.01 & 1.07  & - \\
       & & 2.01 (0.08) & 1.10 (0.05) & 0.95 \\[0.cm]
        \hline\\[-0.3cm] 
       $t_{10}/\sqrt{\var}(t_{10})$ & \ref{sigma_identity}  & 2.05 & 1.08 & - \\
       & & 2.01 (0.07) & 1.10 (0.05) & 0.95 \\
       & \ref{sigma_two_spikes} & 
        2.02 & 1.11 & - \\
       & & 2.01 (0.08) & 1.10 (0.05) & 0.96 \\ 
        & \ref{sigma_five_spikes}  & 2.06 & 1.14 & - \\
       & & 2.01 (0.08) & 1.10 (0.05) &  0.94\\[0.2cm]  
    \end{tabular}
    \caption{  Performance of bootstrap estimates for the mean, standard deviation and 0.95-quantile of $G_n$.}
    \label{table1_bootstrap_gap}
\end{table}

\begin{table}
    \centering
    \begin{tabular}{cll}{}
        $x_{11} $ & $(\lambda_1(\bfSigma),\lambda_2(\bfSigma))$ & $\mathbb{E} ( \lambda_1(\hat\bfSigma)) - \lambda_1(\bfSigma)  $  \\[0.1cm] 
        \hline\\[-0.3cm]
         $ \mathcal{N}(0,1) $ & $(1,1)$ & 3.34 \\
          & & 3.17 (0.19) \\
        & $(1.1,1)$ & 3.24 \\ 
        & & 3.18 (0.18)\\
        & $(1.2,1)$ & 3.14  \\
        & & 3.16 (0.20)\\
        & $(1.2,1.1)$  &  3.14   \\ 
        & & 3.17 (0.18)\\
        \hline\\[-0.3cm] 
        $t_{10}/ \sqrt{\operatorname{var}}(t_{10})$ &  $(1,1)$ & 3.35\\
        & & 3.10 (0.22)\\
       & $(1.1,1)$ & 3.25  \\ 
       & & 3.10 (0.23)\\
        & $(1.2,1)$ & 3.15   \\ 
             & & 3.10 (0.23)\\
         & $(1.2,1.1)$  &  3.15   \\ 
         & & 3.11 (0.22)
    \end{tabular}
\caption{ 
Performance of bootstrap estimate for the bias of $\lambda_1(\hat\bfSigma)$.}
\label{table_bootstrap_bias2}
\end{table}

\subsection{Stock market data}\label{sec_stock}
In the context of finance, the largest eigenvalues of sample covariance matrices play a well-established role in capturing the primary sources of variance in asset returns~\citep{fabozzi2007robust,laloux2000random, ruppert2011statistics}.
Likewise, it is of interest for practitioners to consider testing $\mathsf{H}_{0,n}$ versus $\mathsf{H}_{1,n}$ in order to understand whether or not any of the population eigenvalues are supercritical.

As an illustration, we study this testing problem with two high-dimensional financial datasets derived from log returns of stocks in the S\&P 500 index. 
Specifically, we used the \texttt{quantmod} R package to  collect quarterly log-returns over 53 quarters between 01/08/2010-01/09/2023, as well as over 93 quarters between 01/08/2000-01/09/2023 from Yahoo Finance. After excluding stocks with missing values, this process produced two $n\times p$ data matrices with \smash{$(n,p) \in \{(53, 440), (93, 364)\}$,} where $p$ refers to the number of stocks.

Table \ref{table_p_values_stock_untruncated} reports the p-values that arise from the two datasets when using the proposed statistic $T_n$ and Onatski's statistic $R_n(\kappa)$ with $\kappa\in\{1,10\}$. Up to two decimal places, the proposed statistic gives p-values that are 0, whereas the p-values for $R_n(1)$ and $R_n(10)$ are non-zero. Furthermore, the p-values for $R_n(10)$ are larger than those for $R_n(1)$, which agrees with our earlier discussion of the role of the parameter $\kappa$. Note too that in the particular case of a 5\% nominal level, the statistic $R_n(1)$ is only able to reject the null for one of the datasets, while $R_n(10)$ is not able to reject the null for either dataset. Overall, these results for the p-values of the three statistics are well aligned with the power comparisons based on synthetic data in Section~\ref{sec:expt:test}. \rev{Lastly, we computed the proposed estimate $\hat K_n$ for the number of supercritical eigenvalues from Section~\ref{sec_method} and we observed $\hat K_n=2$ for both datasets.}

\begin{table}[h]
\centering
\begin{tabular}{cccc}
& \multicolumn{3}{c}{   p-values from stock data } 
 \\[0.1cm] \hline 
$(n,p)$ & $T_n$ & $R_n(1)$ & $R_n(10)$ \\ \hline 
\begin{tabular}[c]{@{}l@{}}$(53, 440)$  
\end{tabular} & 0.00 & 0.02 & 0.11 \\
\begin{tabular}[c]{@{}l@{}}$(93, 364)$ 
\end{tabular} & 0.00 & 0.05 & 0.12 \\[0.2cm]
\end{tabular}
\caption{p-values computed from the statistics $T_n$, $R_n(1)$ and $R_n(10)$ in the stock data example.}
\label{table_p_values_stock_untruncated}
\end{table}

\subsection{Finite-sample behavior of $\hat{\varepsilon}$} \label{sec_epsilon_hat}
\rev{Recall the random variable $\hat \varepsilon$ that was defined before Corollary 1 as the infimum over $\varepsilon>0$ for which the statistic $T_n$ can reject the null hypothesis $\mathsf{H}_{0,n}: \lambda_1(\bfSigma)\leq \frac{1}{(1+\ve)\xin{1}}$.  In this subsection, we present simulation results illustrating how the behavior of $\hat\varepsilon$ is influenced by the value of $\lambda_1(\bfSigma)$. The results are based on the model $\bfSigma=\operatorname{diag}(\lambda_1(\bfSigma),1,\ldots,1)$, and for each of several values of $\lambda_1(\bfSigma)$ in a grid contained in the interval $[1,3]$, we generated 500 datasets of size $(n,p)=(600,400)$ from $\mathcal{N}(0,\bfSigma)$. From each dataset, we obtained one realization of $\hat\varepsilon$, and the median of the 500 realizations at each grid value of $\lambda_1(\bfSigma)$ is plotted in Figure \ref{fig:epsilon_hat_lambda}. Note also that in this model, the critical threshold satisfies $1/\xi_{n,1}\approx 1+\sqrt{p/n}\approx 1.82$, which is marked with a dashed vertical line.}

\rev{To understand the plot, observe that when $\lambda_1(\bfSigma)$ becomes larger, the null hypothesis $\mathsf{H}_{0,n}$ can be violated with smaller values of $\varepsilon$. Likewise, the median of $\hat \varepsilon$ is expected to decrease as $\lambda_1(\bfSigma)$ becomes larger, and this is confirmed in Figure \ref{fig:epsilon_hat_lambda}. Another way of interpreting this behavior is that smaller values of $\hat\varepsilon$ can be regarded as stronger empirical evidence against the null hypothesis, as was discussed in Section~\ref{sec_formulation}. Lastly, it should be noted that in cases where $T_n$ cannot reject $\mathsf{H}_{0,n}$ for any $\varepsilon$, the formal definition of $\hat\varepsilon$ implies that it is equal to $\infty$. Hence, when this occurs for more than 50\% of the realizations of $\hat\varepsilon$ at a particular value of $\lambda_1(\bfSigma)$, the curve in Figure~\ref{fig:epsilon_hat_lambda} leaves the plotting range.
}

\begin{figure}[h]
  \centering
 \begin{overpic}[width=0.55\linewidth]{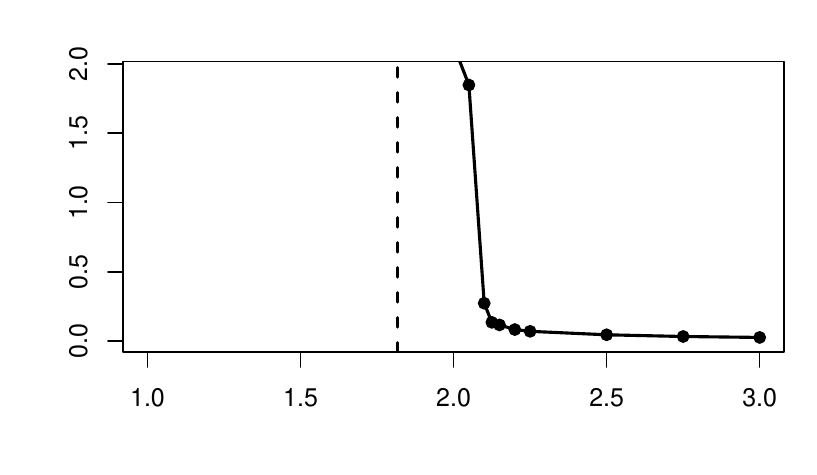}
        \put(55,1){\makebox(0,0){$\lambda_1(\boldsymbol{\Sigma})$}}
        \put(1,32){\rotatebox{90}{\makebox(0,0){median of $\hat{\varepsilon}$}}}
    \end{overpic}
    \caption{ Median of $\hat{\varepsilon}$ as a function of $\lambda_1(\bfSigma)$ for $(p,n)=(400,600)$ based on 500 repetitions, $\bfSigma=\operatorname{diag}(\lambda_1(\bfSigma),1,\ldots,1)$, and $x_{11}\sim\mathcal{N}(0,1)$. The dashed vertical line indicates the critical threshold $1/\xi_{n,1}\approx 1.82$. The median curve leaves the plotting range whenever more than $50\%$ of the realizations yield $\hat{\varepsilon}=\infty$. 
     }

    \label{fig:epsilon_hat_lambda}
\end{figure}

\subsection{Sensitivity analysis with respect to $\varepsilon$} \label{sec_sim_epsilon}

\rev{ Throughout our earlier experiments, we set $\varepsilon=0.2$ across all numerical examples. 
Here, we look at how changing the value of $\varepsilon$ affects performance of the proposed test. In Tables \ref{tab_sen_one_spikes} and \ref{tab_sen_three_spikes}, we show how the empirical rejection rates vary with $\varepsilon$ in the scenarios underlying earlier results presented in Figure \ref{fig1} and Table  \ref{table_three_spikes}, respectively. We observe that  the performance under $\mathsf{H}_{0,n}$ is essentially unchanged by the choice of $\varepsilon$.  Under $\mathsf{H}_{1,n}$, the power does decrease as $\varepsilon$ decreases, and this is expected, because the testing problem becomes more difficult when there is less separation between the null and alternative. Nevertheless, the reduction in power mostly occurs gradually. Also, larger values of $\lambda_1(\bfSigma)$ reduce the sensitivity to $\varepsilon$ under the alternative.
}

\begin{table}[h]
\centering
\begin{tabular}{c|cccc}
  \hline
$\varepsilon$ & \multicolumn{4}{c}{$\lambda_1(\bfSigma)$} \\ 
~ & 1.25 & 2.3 &  2.5 & 3 \\ 
  \hline
 0.04 &   0.04 &    0.16 &   0.31 &   0.97\\ 
  0.06 &   0.04 &   0.43 &   0.74 &   1.00 \\ 
  0.08 &   0.04 &   0.64 &   0.88 &   1.00 \\ 
  0.10 &   0.04 &   0.74 &   0.93 &   1.00 \\ 
  0.12 &   0.04 &   0.79 &   0.95 &   1.00  \\ 
  0.14 &   0.04 &   0.80 &   0.96 &   1.00 \\ 
  0.16 &   0.04 &   0.81 &   0.96 &   1.00 \\ 
  0.18 &   0.04 &   0.81 &   0.96 &   1.00  \\ 
  0.20 &   0.04 &  0.82 &   0.96 &   1.00  \\ 
   \hline 
\end{tabular}
\caption{Rejection rates for $T_n$ as a function of $\varepsilon$ for $(n,p)=(600,400)$, $x_{11}\sim \mathcal{N}(0,1)$, and $\bfSigma=\operatorname{diag}(\lambda_1(\bfSigma),1, \ldots, 1)$. 
The first column corresponds to models under $\mathsf{H}_{0,n}$, and the other columns correspond to $\mathsf{H}_{1,n}.$
}
\label{tab_sen_one_spikes}
\end{table}

\begin{table}[h]
\centering
\begin{tabular}{c|cccc}
  \hline
$\varepsilon$ & \multicolumn{4}{c}{$(\lambda_1(\bfSigma), \lambda_2(\bfSigma), \lambda_3(\bfSigma))$} \\ 
~ & (1.25, 1.25, 1.25) & (2.3, 1.5, 1.3) & (2.5, 1.5, 1.3)  & (2.7, 1.5, 1.3)  \\ 
  \hline 
  0.04 &   0.03 &   0.13 &   0.26 &   0.62\\ 
  0.06 &   0.03 &   0.39 &   0.75 &   0.93\\ 
  0.08 &   0.03 &   0.57 &   0.88 &   0.98 \\ 
  0.10 &   0.03 &   0.71 &   0.93&   0.99\\ 
  0.12 &   0.03 &   0.76  &   0.94&   1.00\\ 
  0.14 &   0.03 &   0.78&   0.95 &   1.00\\ 
  0.16 &   0.03 &   0.79&   0.95 &   1.00\\ 
  0.18 &   0.03 &   0.80 &   0.96 &   1.00\\ 
  0.20 &   0.03 &   0.81 &   0.96 &   1.00\\ 
   \hline
\end{tabular}
\caption{Rejection rates for $T_n$ as a function of $\varepsilon$ in the same setting as Table \ref{table_three_spikes}. The first column corresponds to models under $\mathsf{H}_{0,n}$, and the other columns correspond to $\mathsf{H}_{1,n}.$}
\label{tab_sen_three_spikes}
\end{table}

\newpage

\newpage

\subsection{Robustness checks}\label{sec_robustness}
\rev{In this subsection, we provide numerical results for population eigenvalue configurations not covered by our earlier experiments, and also, we look at the potential effects of heavy tails and temporal dependence in the stock data example. All the results with synthetic data in this subsection are based on $\varepsilon=0.2$, $(n,p)=(600,400)$, and  500 simulation trials with the $\{x_{ij}\}$ variables drawn from either $\mathcal{N}(0,1)$ or a standardized $t_{10}$ distribution.}\\

\rev{\noindent \textit{Bulk population spectrum drawn from Beta$(2,2)$.}
We drew the eigenvalues of $\bfSigma$ in an i.i.d.~manner from the Beta$(2,2)$ distribution, and added a shift $\delta>0$ to the largest eigenvalue to construct models under $\mathsf{H}_{0,n}$ and $\mathsf{H}_{1,n}$. (Note that the density of Beta(2,2) is zero at both endpoints of [0,1], with its mode at 1/2.) The matrix $\bfSigma$ was taken to be diagonal, and we held it fixed over all simulation trials.
In contrast with many of the previous settings, the bulk population eigenvalues are all distinct. In Table \ref{table_beta}, we report the empirical rejection rates of the three tests $T_n, R_n(1), R_n(10)$ as a function of $\lambda_1(\bfSigma).$ The results show that the proposed test performs favorably.\\ }

\rev{\noindent \textit{Zero gap between leading eigenvalues.}
In Table \ref{table_two_spikes}, we consider models with two
equal leading population eigenvalues,
$\lambda_1(\bfSigma)=\lambda_2(\bfSigma)>1
=\lambda_3(\bfSigma)=\cdots=\lambda_p(\bfSigma)$, violating the conditions of
Theorem \ref{thm_alternative} on power consistency of $T_n$.  In comparison with the corresponding setting in which $\lambda_1(\bfSigma)$ and $\lambda_2(\bfSigma)$ are distinct (shown in the upper two panels of Figure~\ref{fig1}), all three statistics $T_n$, $R_n(1)$, and $R_n(10)$ still maintain good agreement with the nominal level under $\mathsf{H}_{0,n}$, but all of their rejection rates are noticeably reduced when $\mathsf{H}_{0,n}$ does not hold. Despite this overall reduction when $\lambda_1(\bfSigma)=\lambda_2(\bfSigma)$, the proposed statistic $T_n$ still has a higher rejection rate than $R_n(1)$ and $R_n(10)$ when $\mathsf{H}_{0,n}$ does not hold.}\\

\rev{\emph{Tail effects in stock data}. Financial data may exhibit heavy tails, and it is therefore important to understand how the proposed method is affected by them. To address this issue, we applied a truncation procedure to the stock market data and compared the resulting p-values for testing $\mathsf{H}_{0,n}$ versus $\mathsf{H}_{1,n}$ to those obtained from the original data. The truncated data matrices were constructed in the following way. For each entry $Y_{ij}$ of the original data matrices, with $i=1,\dots,n$ and $j=1,\dots,p$, we applied the truncation function $Y_{ij}\mapsto \textup{sgn}(Y_{ij})(|Y_{ij}|\wedge (3\hat\sigma_j))$, where $\hat\sigma_j$ is the empirical standard deviation of $Y_{1j},\dots,Y_{nj}$, and then the columns of the resulting matrix were centered by subtracting empirical means.}

\rev{In Table \ref{table_p_values_stock}, we report the p-values for the truncated data. Comparing with the results for the original data in Table~\ref{table_p_values_stock_untruncated}, the p-values for $T_n$ are unchanged, indicating that the tails of the stock data have limited influence on $T_n$. Meanwhile, the p-values for $R_n(1)$ and $R_n(10)$ are affected somewhat more by the truncation, but are still fairly stable as well. }\\

\rev{\emph{Temporal dependence in stock data}. Whereas our theoretical results are based on i.i.d.~observations, stock data are generally subject to temporal dependence. To explore the effect of temporal dependence on $T_n$, $R_n(1)$ and $R_n(10)$, we generated 500 resampled versions of each of the original stock datasets by sampling  with replacement in an i.i.d.~manner, and for each realization, we computed the p-values of the three statistics. Also, each resampled data matrix was of the same size as its original counterpart. Table \ref{table_p_values_stock_resampling} shows that the medians of the 500 p-values for all three statistics are similar to those computed on the original data (Table~\ref{table_p_values_stock_untruncated}), and in the case of $T_n$ the results are unchanged. }

\rev{The underlying idea of this experiment is that the i.i.d.~resampling mechanism may be interpreted as a rough proxy for sampling from the underlying population in a way that intentionally removes all temporal dependence structure. Hence, the similarity of the p-values for the original and resampled data suggests that temporal dependence plays a limited role in this particular example. This is also plausible because the stock data are log returns computed over disjoint three-month periods.}

    \begin{table}[h]
    \centering 
\begin{tabular}{cccccc}
& & & \multicolumn{3}{c}{\underline{ \ \ \ rejection probabilities \ \ \ }}\\[0.1cm]
$x_{11}$ & $\delta$ & hypothesis &  $T_n$    & $R_n(1)$    & $R_n(10)$   \\ \hline 
\multirow{8}{*}{$\mathcal{N}(0,1)$} & 0 & $\mathsf{H}_{0,n}$ & 0.07 & 0.06 & 0.05 \\ 
 & 0.2 & neither & 0.11 & 0.07 & 0.05 \\ 
   & 0.4 & neither & 0.54 & 0.11 & 0.06 \\ 
   & 0.6 & $\mathsf{H}_{1,n}$ & 0.94 & 0.31 & 0.09 \\ 
   & 0.8 & $\mathsf{H}_{1,n}$ & 0.99 & 0.58 & 0.19 \\ 
   & 1 & $\mathsf{H}_{1,n}$ & 1 & 0.79 & 0.29 \\ 
   & 1.2 & $\mathsf{H}_{1,n}$ & 1 & 0.91 & 0.39 \\ 
   & 1.4 & $\mathsf{H}_{1,n}$ & 1 & 0.98 & 0.5 \\ 
\hline 
\multirow{8}{*}{$t_{10}/\sqrt{\Var(t_{10})}$} & 0 & $\mathsf{H}_{0,n}$ & 0.05 & 0.04 & 0.05 \\ 
  & 0.2 & neither & 0.11 & 0.07 & 0.05 \\ 
  & 0.4 & neither & 0.38 & 0.1 & 0.05 \\ 
  & 0.6 & $\mathsf{H}_{1,n}$ & 0.86 & 0.31 & 0.11 \\ 
  & 0.8 & $\mathsf{H}_{1,n}$ & 0.99 & 0.54 & 0.17 \\ 
  & 1 & $\mathsf{H}_{1,n}$ & 1 & 0.73 & 0.25 \\ 
  & 1.2 & $\mathsf{H}_{1,n}$ & 1 & 0.88 & 0.31 \\ 
  & 1.4 & $\mathsf{H}_{1,n}$ & 1 & 0.97 & 0.47 \\
\end{tabular}
\caption{
Rejection probabilities for the tests $T_n$, $R_n(1)$ and $R_n(10)$ under the null $\mathsf{H}_{0,n}$ and the alternative $\mathsf{H}_{1,n}$, when the spectrum is randomly generated from a Beta$(2,2)$. The nominal level is $\alpha=0.05$ and the value $\delta$ is added to the first eigenvalue.
}
\label{table_beta}
\end{table}

    \begin{table}
    \centering 
\begin{tabular}{cccccc}
& & & \multicolumn{3}{c}{\underline{ \ \ \ rejection probabilities \ \ \ }}\\[0.1cm]
$x_{11}$ & $\lambda_1(\bfSigma) = \lambda_2(\bfSigma)$ & hypothesis &  $T_n$    & $R_n(1)$    & $R_n(10)$   \\ \hline 
\multirow{1}{*}{} 
& 1 & $\mathsf{H}_{0,n}$ & 0.05 & 0.05 & 0.05 \\ 
   & 1.2 & $\mathsf{H}_{0,n}$ & 0.04 & 0.04 & 0.05 \\ 
   & 1.4 & $\mathsf{H}_{0,n}$ & 0.05 & 0.05 & 0.06 \\ 
   & 1.6 & neither & 0.05 & 0.06 & 0.03 \\ 
   & 1.8 & neither & 0.13 & 0.08 & 0.05 \\ 
 $\mathcal{N}(0,1)$  & 2 & neither & 0.25 & 0.07 & 0.06 \\ 
   & 2.2 & neither & 0.46 & 0.06 & 0.06 \\ 
   & 2.4 & neither & 0.61 & 0.03 & 0.06 \\ 
   & 2.6 & neither  
   & 0.68 & 0.00 & 0.08 \\ 
   & 2.8 & neither  
   & 0.74 & 0.00 & 0.12 \\ 
   & 3 & neither 
   & 0.73 & 0.00 & 0.17 \\
\\ 
\hline 
\multirow{1}{*}{} & 
 1 & $\mathsf{H}_{0,n}$ & 0.06 & 0.06 & 0.04 \\ 
  & 1.2 & $\mathsf{H}_{0,n}$ & 0.04 & 0.05 & 0.05 \\ 
  & 1.4 & $\mathsf{H}_{0,n}$ & 0.05 & 0.05 & 0.05 \\ 
  & 1.6 & neither & 0.07 & 0.05 & 0.04 \\ 
   & 1.8 & neither & 0.12 & 0.07 & 0.06 \\ 
  $t_{10}/\sqrt{\Var(t_{10})}$ & 2.0 & neither & 0.26 & 0.06 & 0.05 \\ 
   & 2.2 & neither & 0.52 & 0.07 & 0.07 \\ 
   & 2.4 & neither & 0.63 & 0.04 & 0.06 \\ 
   & 2.6 & neither 
   & 0.69 & 0.01 & 0.09 \\ 
   & 2.8 & neither 
   & 0.74 & 0.00 & 0.12 \\ 
   & 3 & neither 
   & 0.74 & 0.00 & 0.17 \\
\end{tabular}
\caption{
Rejection probabilities for the tests $T_n$, $R_n(1)$ and $R_n(10)$ in the presence of a zero leading eigengap, with $\bfSigma=\operatorname{diag}(\lambda_1(\bfSigma), \lambda_2(\bfSigma), 1, \ldots, 1)$ and $\lambda_1(\bfSigma)=\lambda_2(\bfSigma)$ . 
The nominal level is $\alpha=0.05$. 
}
\label{table_two_spikes}
\end{table}

\begin{table}
\centering
\begin{tabular}{cccc}
& \multicolumn{3}{c}{   p-values for truncated  data   }
 \\[0.1cm] \hline 
$(n,p)$ & $T_n$ & $R_n(1)$ & $R_n(10)$ \\ \hline 
\begin{tabular}[c]{@{}l@{}}$(53, 440)$  
\end{tabular} & 0.00 & 0.01 & 0.05 \\
\begin{tabular}[c]{@{}l@{}}$(93, 364)$ 
\end{tabular} & 0.00 & 0.03 & 0.19 \\[0.2cm]
\end{tabular}
\caption{ 
p-values computed from the statistics $T_n, R_n(1), R_n(10)$, based on the truncated stock data.
}
\label{table_p_values_stock}
\end{table}

\begin{table}
\centering

\begin{tabular}{c*{3}{>{\centering\arraybackslash}p{2cm}}}
& \multicolumn{3}{c}{   median p-values  (sampling with repl.) }
 \\[0.1cm] \hline 
$(n,p)$ & $T_n$ & $R_n(1)$ & $R_n(10)$  \\ \hline 
\begin{tabular}[c]{@{}l@{}}$(53, 440)$  
\end{tabular} & 0.00 & 0.01 & 0.07 \\
\begin{tabular}[c]{@{}l@{}}$(93, 364)$ 
\end{tabular} & 0.00 & 0.04 & 0.16 \\[0.2cm]
\end{tabular}
\caption{ 
Median p-values computed from the statistics $T_n, R_n(1), R_n(10)$, based on 500 samples drawn with replacement from the original stock data.
}
\label{table_p_values_stock_resampling}
\end{table}

\newpage

\section{Connections with factor models}\label{sec_factor}

\rev{Let $\tilde{\mathbf y}_1,\dots,\tilde{\mathbf y}_n$ be a set of observations generated by the approximate factor model \citep[e.g.][]{BaiNg2002, onatski2009testing}
\begin{align} \label{eq_factor_model}
  \tilde{\mathbf{y}}_i = \boldsymbol{L} \mathbf{f}_i + \mathbf{e}_i  , \quad  1\leq i \leq n,
\end{align}
where $\boldsymbol{L}\in \R^{p \times \ell}$ is a non-random matrix of factor loadings, $\mathbf{f}_i$ is a random $\ell$-dimensional vector of factors, and $\mathbf{e}_i$ is a random $p$-dimensional idiosyncratic error vector. To align this model with the theory developed in the main text, we consider the following assumptions.
\begin{enumerate}[label=B\arabic*]
  \item \label{ass_b1} (Dimensionality and number of factors). The number $\ell$ is fixed with respect to $n$, and there is a number  $y \in (0,\infty)\setminus \{1\} $ such that $y_n = p/n$ satisfies $\lim_{n\to\infty} y_n = y$. 
  %
  %
  \item\label{ass_b2} (Noise distribution). The random vectors $\mathbf{e}_1,\dots,\mathbf{e}_n$ are i.i.d.~samples from $  \mathcal{N}(\mathbf{0}, \boldsymbol{\Psi})$ for some positive definite matrix $\boldsymbol \Psi\in\R^{p\times p}$. Also, as $n\to\infty$, the empirical spectral distribution function $F^{\boldsymbol{\Psi}}$ has a weak limit whose support is a finite union of compact intervals in $(0,\infty)$, and there is a fixed compact interval in $(0,\infty)$ containing the support of $F^{\boldsymbol \Psi}$ for all large $n$. 
  \item\label{ass_b3} (Factors and loadings). The random vectors $\mathbf{f}_1, \ldots, \mathbf{f}_n$ are  i.i.d.~samples from $\mathcal{N}(\mathbf{0}, \bfI)$ and $ \lambda_1(\boldsymbol L\boldsymbol L\ttop)\lesssim 1$.
  \item \label{ass_b4} (Independence between factors and noise). The random vectors  $\mathbf{e}_1, \ldots, \mathbf{e}_n $  and $\mathbf{f}_1, \ldots, \mathbf{f}_n $ are independent. 
 \end{enumerate}
   If we let $\bfSigma=\E(\tilde{\mathbf{y}}_1 \tilde{\mathbf{y}}_{1}\ttop)= \boldsymbol{L} \boldsymbol{L}^\top + \boldsymbol{\Psi}$ denote the covariance matrix of $\tilde{\mathbf{y}}_1$, then the model \eqref{eq_factor_model} may be written in the form 
\begin{align} \label{eq_factor_model_sep}
    \tilde{\mathbf y}_i = \bfSigma^{1/2} \bfx_i, \quad 1 \leq i \leq n,
\end{align}
where $\bfx_1, \ldots, \bfx_n$ are i.i.d.~samples from the standard Gaussian distribution on $\R^p$. It follows that under~\ref{ass_b1}-\ref{ass_b4}, the model \eqref{eq_factor_model} is captured by the one defined in~\ref{ass_mp_regime}-\ref{ass_lsd} of Section~\ref{sec_prelim}. Likewise, if the rows of the original data matrix $\mathbf{Y}$ from Section~\ref{sec_prelim} are replaced by $\tilde{\mathbf y}_1,\dots,\tilde{\mathbf y}_n$, and if the previously defined hypotheses are understood with respect to the eigenvalues of $\bfSigma=\boldsymbol{L} \boldsymbol{L}^\top + \boldsymbol{\Psi}$, then our earlier theory remains applicable. In particular, we have the following result on asymptotic level control, power consistency, and estimation of the number of supercritical eigenvalues $K$. As a clarification, it should be noted that $K$ is not necessarily the same as the number of factors $\ell$. 
\begin{corollary}\label{cor_factor_model} Let $\alpha\in (0,1)$, and suppose conditions \ref{ass_b1}-\ref{ass_b4} hold.  In addition, suppose that the statistics $T_n$ and $\hat K_n$ defined in~\eqref{eqn:Tndef} and~\eqref{eq_def_hat_K}  are computed from the observations $\tilde{\mathbf y}_1,\dots,\tilde{\mathbf y}_n$. 
    \begin{enumerate}[label=(\alph*)]
        \item If $\mathsf{H}_{0,n}$ holds for all large $n$, then as $n\to\infty$
        \begin{align*}
     \PR \big( T_n>q_{1-\alpha} \big) \to \alpha.
   \end{align*}
        \item  If $K\geq 1$ and $\lambda_1(\bfSigma)>\dots >\lambda_K(\bfSigma)$ are fixed with respect to $n$, and  $\mathsf{H}_{1,n}(K)$ holds for all large $n$, then as $n\to\infty$,
   \begin{align*}
     \PR \big( T_n >q_{1-\alpha} \big) &\to 1,\\
     \PR(\hat K_n=K)&\to 1-\alpha.
   \end{align*}
    \end{enumerate}
\end{corollary}
}

\rev{A central issue in the study of factor models is the estimation of the number of factors $\ell$, and so it is of interest to understand when $\hat K_n$ may be effective for this purpose. The previous corollary sheds some light on this question, as it shows that $\hat K_n$ correctly estimates $\ell$ with high probability when 
$\mathsf{H}_{n,1}(\ell)$ holds (under  \ref{ass_b1}-\ref{ass_b4}). In other words, $\mathsf{H}_{n,1}(\ell)$ means that the number of supercritical eigenvalues of $\bfSigma$ coincides with the number of factors, and this does occur in some well-established factor modeling scenarios. Indeed, it is common in the literature to assume that the non-zero eigenvalues of $\boldsymbol L \boldsymbol L\ttop $ diverge as $n\to\infty$, at a possibly slow rate, while the eigenvalues of $\boldsymbol \Psi$ remain bounded~\citep[cf.][and references therein]{bai2008large}. However, even when $\mathsf{H}_{n,1}(\ell)$ holds, an important detail with regard to Corollary~\ref{cor_factor_model} is that conditions~\ref{ass_b2} and \ref{ass_b3} require that the eigenvalues of \emph{both} $\boldsymbol L\boldsymbol L\ttop$ and $\boldsymbol \Psi$ remain bounded as $n\to\infty$. Consequently, this raises the question of whether Corollary~\ref{cor_factor_model} can be extended in a way that relaxes the boundedness condition $\lambda_1(\boldsymbol L\boldsymbol L\ttop)\lesssim 1$. More generally, considering that the literature on high-dimensional factor models is extensive and deals with numerous generalizations of the model in~\ref{ass_b1}-\ref{ass_b4}, we leave to future work the possibility of adapting the hypotheses $\mathsf{H}_{n,0}$ and $\mathsf{H}_{n,1}(\ell)$ and the statistics $T_n$ and $\hat K_n$ to such broader situations.  
}
\end{document}